\documentclass[11pt]{amsart}
\usepackage{pdfsync}
\usepackage{amssymb,bm,bbm,mathrsfs}
\usepackage[mathscr]{euscript}
\usepackage[all]{xy}
\usepackage{xcolor}
\usepackage{hyperref}
\hypersetup{colorlinks=true}

\newcommand{\map}[1]{\xrightarrow{#1}}
\newcommand{\iso}{\cong}
\newcommand{\define}{\stackrel{\mathrm{def}}{=}}
\newcommand{\Inv}{\mathrm{Inv}}
\newcommand{\kk}{{\bm{k}}}

\newcommand{\Hom}{\mathrm{Hom}}
\newcommand{\Aut}{\mathrm{Aut}}
\newcommand{\End}{\mathrm{End}}
\newcommand{\Spec}{\mathrm{Spec}}
\newcommand{\Spf}{\mathrm{Spf}}

\newcommand{\Z}{\mathbb Z}
\newcommand{\R}{\mathbb R}
\newcommand{\C}{\mathbb C}
\newcommand{\co}{\mathcal O}
\newcommand{\alg}{\mathrm{alg}}
\newcommand{\ord}{\mathrm{ord}}
\newcommand{\GL}{\mathrm{GL}}

\begin{document}
\author{Benjamin Howard and Qirui Li}
\title[Intersections in Lubin-Tate space]{Intersections in Lubin-Tate space and   biquadratic fundamental lemmas}
\date{}

\thanks{B.H.~was supported in part by NSF grants  DMS-1801905 and DMS-2101636}

\address{Department of Mathematics\\Boston College\\ 140 Commonwealth Ave. \\Chestnut Hill, MA 02467, USA}
\email{howardbe@bc.edu}

\address{Department of Mathematics\\University of Toronto\\ 40 St.~George St.\\Toronto, Ontario M5S 2E4 Canada }
\curraddr{Mathematisches Institut der Universit\"at Bonn, Endenicher Allee 60, D-53115 Bonn, Germany}
\email{qiruili@math.uni-bonn.de}

\begin{abstract}
We compute the intersection multiplicities of special cycles in  Lubin-Tate spaces, and formulate a new  
 arithmetic fundamental lemma  relating these intersections to derivatives of orbital integrals.
\end{abstract}

\maketitle

\setcounter{tocdepth}{1}
\tableofcontents

\theoremstyle{plain}
\newtheorem{theorem}{Theorem}[subsection]
\newtheorem{bigtheorem}{Theorem}[section]
\newtheorem{proposition}[theorem]{Proposition}
\newtheorem{lemma}[theorem]{Lemma}
\newtheorem{corollary}[theorem]{Corollary}
\newtheorem{conjecture}[theorem]{Conjecture}

\theoremstyle{definition}
\newtheorem{definition}[theorem]{Definition}
\newtheorem{hypothesis}[theorem]{Hypothesis}

\theoremstyle{remark}
\newtheorem{remark}[theorem]{Remark}
\newtheorem{example}[theorem]{Example}
\newtheorem{question}[theorem]{Question}

\numberwithin{equation}{subsection}
\renewcommand{\thebigtheorem}{\Alph{bigtheorem}}


\section{Introduction}


We compute the intersection multiplicity of special cycles in Lubin-Tate deformation spaces.
Interpolating ideas from \cite{HS, Li, Zhang},  we formulate a new arithmetic fundamental lemma expressing the intersection as the  derivative of an orbital integral.
The results and conjectures of \cite{Li} are special cases of those presented here.


\subsection{An intersection problem}
\label{ss:intro LT}


Fix a  nonarchimedean local field $F$ (of any characteristic), let $\breve{F}$ be the completion of its maximal unramified extension, and let $\kk$ be the residue field of $\co_{\breve{F}}$. 

Suppose   $\mathscr{G}$ is a one-dimensional formal $\co_F$-module of height $2h$ over $\kk$, and denote by $X \to \Spf(\co_{\breve{F}} )$  the Lubin-Tate deformation space of $\mathscr{G}$.  Being (noncanonically) isomorphic to the formal spectrum of a power series ring $\co_{\breve{F}} [[T_1,\ldots, T_{2h-1}]]$,  it is a formal scheme of dimension $2h$.

Let $K_1$ and $K_2$ be   separable quadratic field extensions   of $F$, and  fix $\co_F$-algebra maps
\begin{equation}\label{introLTpair}
\Phi_1 : \co_{K_1} \to \End_{\co_F}( \mathscr{G}) ,\qquad \Phi_2 : \co_{K_2} \to  \End_{\co_F}( \mathscr{G}).
\end{equation}
 The deformation space of $\mathscr{G}$ with its  extra $\co_{K_i}$-action is a closed formal subscheme $f_i : Y_i \hookrightarrow X$.

The  formal subschemes  $Y_1$ and $Y_2$ each have dimension $h$, and so it is natural to consider the   intersection multiplicity
\begin{equation}\label{intro intersection}
I(\Phi_1,\Phi_2) = \mathrm{len}_{\co_{\breve{F}} }\, H^0( X , f_{1*}\co_{Y_1} \otimes_{\co_X} f_{2*}\co_{Y_2}).
\end{equation}
In  this generality \eqref{intro intersection} could be infinite; this would be the case, for example, if $K_1=K_2$ and $\Phi_1=\Phi_2$.

One of the main results of this paper is a  formula for the above intersection multiplicity.  The formula involves a polynomial invariant  associated to the pair $(\Phi_1,\Phi_2)$, which we now describe.


\subsection{Invariant polynomials}
\label{ss:intro polynomial}


Suppose for the moment  that $K_1$ and $K_2$ are any quadratic \'etale $F$-algebras.  Equivalently, $K_i$ is either a quadratic Galois extension, or $K_i \iso F\times F$.
Denote by   $\sigma_i \in \Aut(K_i/F)$  the nontrivial automorphism.   

The automorphism group of the $F$-algebra 
$K= K_1\otimes_F K_2$  contains the Klein four subgroup
$
\{ \mathrm{id}, \tau_1,\tau_2,\tau_3\} \subset \Aut(K/F)
$
with nontrivial elements 
\begin{align*}
(x \otimes y)^{\tau_1}  = x \otimes y^{\sigma_2} ,\quad
( x \otimes y) ^{\tau_2} =  x^{\sigma_1} \otimes y ,\quad 
(x \otimes y) ^{\tau_3}  =  x^{\sigma_1} \otimes y^{\sigma_2} .
\end{align*}
Denote by $K_3\subset K$ the subalgebra  of elements fixed by $\tau_3$.  
The picture, along with generators of the automorphism groups, is 
\begin{equation}\label{biquadratic}
\xymatrix{
&  { K}  \\
{   K_1 }  \ar@{-}[ur]^{  \tau_1}  &     {   K_2   }  \ar@{-}[u]_{\tau_2} &  {  K_3  } \ar@{-}[ul]_{\tau_3  }  \\
&  {   F  . }  \ar@{-}[ul]^{\sigma_1}  \ar@{-}[u]_{\sigma_2} \ar@{-}[ur]_{\sigma_3}
}
\end{equation}

Elementary algebra shows that  $K_1\iso K_2$ if and only if $K_3\iso F\times F$.
If $K_1$ and $K_2$ are nonisomorphic field extensions, then $K$ is a biquadratic field extension of $F$, and $K_3$ is its unique third quadratic subfield.

Now suppose $B$ is any central simple $F$-algebra of dimension $4h^2$, and that we are given $F$-algebra embeddings 
\begin{equation}\label{intro12pair}
\Phi_1 : K_1 \to B ,\qquad \Phi_2:K_2 \to B.
\end{equation}
From this data we construct in \S \ref{ss:polynomial} a distinguished element $\mathbf{s}$ in the central simple $K_3$-algebra $C=B\otimes_F K_3$, and show that its reduced characteristic polynomial is a square.  The \emph{invariant polynomial}
\[
\Inv_{(\Phi_1,\Phi_2)} (T) \in K_3[T]
\]
is  its unique monic square root.

The pair $(\Phi_1,\Phi_2)$ is \emph{regular semisimple} (Definition \ref{def:rss}) if $\mathbf{s} \in C^\times$ and the subalgebra $K_3[\mathbf{s}] \subset C$ is an \'etale $K_3$-algebra of dimension $h$.  This is our way of making precise the notion that the pair $(\Phi_1,\Phi_2)$ is in general position; see also Proposition \ref{prop:truerss}.
 As evidence that the invariant polynomial is a natural quantity to consider, 
we prove in Corollary \ref{cor:conjugacy} that the $B^\times$-conjugacy class of a regular semisimple pair is completely determined by its invariant polynomial.


\subsection{The intersection formula}


Now return to the setting of \S \ref{ss:intro LT}, and apply the  constructions of \S \ref{ss:intro polynomial} to the central  division algebra 
\[
B=\End_{\co_F}(\mathscr{G}) \otimes F
\]
to  obtain a degree $h$ polynomial $\Inv_{(\Phi_1,\Phi_2)}$ with coefficients in $K_3$.

To any $\co_F$-algebra embeddings
\[
\Psi_1 : \co_{K_1} \to M_{2h}(\co_F) ,\qquad \Psi_2 : \co_{K_2} \to M_{2h}(\co_F),
\]
we may associate in the same way a degree $h$ polynomial $\Inv_{(\Psi_1,\Psi_2)}$ with coefficients in $K_3$. More generally,  for any $g\in \GL_{2h}(\co_F)$ we may form  the conjugate embedding $g \Psi_2 g^{-1}$ and  the polynomial $\Inv_{(\Psi_1, g \Psi_2 g^{-1})}$.
Let
\[
R(g) = \mathrm{Res}\big( \Inv_{(\Phi_1,\Phi_2)} , \Inv_{(\Psi_1, g \Psi_2 g^{-1})} \big) \in K_3
\]
be the resultant. 
It follows from Proposition \ref{prop:functional} that  $R(g)^2\in F$, and we  define \[|R(g)| = \sqrt{ | R(g)^2 |}.\]
 Here the absolute value on $F$ is  normalized by 
$| \pi | = q^{-1}$,   where $\pi \in \co_F$ is a uniformizing parameter and $q$ is the cardinality of the residue field.

\begin{bigtheorem}\label{thmA}
If   $(\Phi_1,\Phi_2)$ is regular semisimple then $|R(g)| \neq 0$ for all $g\in \GL_{2h}(\co_F)$, and   the intersection multiplicity \eqref{intro intersection}  satisfies
\[
I(\Phi_1,\Phi_2) = 
c(0)  \cdot    | d_1d_2| ^{  - \frac{ h^2}{2}}   \cdot  \int_{ \GL_{2h}(\co_F) }  \frac{ dg} {  | R(g) |  } .
\]
In particular,  the left hand side is finite.
Here $d_i \in \co_F$ is any generator of the discriminant of $K_i/F$,  and 
\[
c(0) = 
\frac{ \# \GL_{2h}(\co_F/\pi\co_F) }{ \# \GL_h (\co_{K_1}/\pi \co_{K_1} )   \cdot  \# \GL_h (\co_{K_2}/\pi \co_{K_2} ) } .
\]
\end{bigtheorem}

Theorem \ref{thmA}  is  a special case of Theorem \ref{thm:final intersection}, which gives a more general  intersection formula for cycles on  covers of $X$ obtained by adding Drinfeld level structures.  When $K_1 = K_2$, it specializes to the main result of \cite{Li}.


\subsection{Matching  and fundamental lemmas}
\label{ss:intro matching}


Let us return to the general situation of \S \ref{ss:intro polynomial}, where $K_1$ an $K_2$ are  quadratic \'etale $F$-algebras.
 From the pair  $(K_1,K_2)$  we constructed  a diagram  \eqref{biquadratic}  of \'etale $F$-algebras.
 There is another diagram of \'etale $F$-algebras, in some sense dual to the first.

Having already constructed $K_3$ from the pair $(K_1,K_2)$, we may set
\begin{equation*} 
K_0=F\times F  
\end{equation*} 
and repeat the construction of \eqref{biquadratic} with  the pair $(K_1,K_2)$ replaced by $(K_0 , K_3)$. 
Using the canonical  isomorphism $K_0 \otimes_F K_3 \iso K_3\times K_3$, the resulting picture is 
\begin{equation}\label{biquadratic2}
\xymatrix{
&  { K_3\times K_3}  \\
{   K_0 }  \ar@{-}[ur]^{  \nu_0}  &     {   K_3   }  \ar@{-}[u]_{\nu_3} &  {  K_3  } \ar@{-}[ul]_{\nu_0 \circ \nu_3  }  \\
&  {   F  ,}  \ar@{-}[ul]^{\sigma_0}  \ar@{-}[u]_{\sigma_3} \ar@{-}[ur]_{\sigma_3}
}
\end{equation}
where $(x,y)^{\nu_0}= (x^{\sigma_3} , y^{\sigma_3} )$ and $(x,y)^{\nu_3} = (y,x)$. 
To be completely explicit: the middle inclusion of $K_3$ into $K_3\times K_3$ is given by $z \mapsto (z,z)$, while the inclusion on the right is $z\mapsto (z,z^{\sigma_3})$.

The key point  is that   the  pairs $(K_1,K_2)$ and $(K_0,K_3)$ give rise to the \emph{same} third quadratic algebra $K_3$.   The constructions described in  \S \ref{ss:intro polynomial} work equally well with the  pair $(K_1,K_2)$  replaced by $(K_0 , K_3)$, and associate to  any  pair of $F$-algebra embeddings
\begin{equation}\label{intro03pair}
\Phi_0 : K_0 \to M_{2h}(F),\qquad  \Phi_3: K_3 \to M_{2h}(F) 
\end{equation}
 a degree $h$  monic polynomial  $\Inv_{(\Phi_0,\Phi_3)}(T) \in K_3[T]$.

Suppose the pairs  $(\Phi_1,\Phi_2)$ and $(\Phi_0,\Phi_3)$ of  \eqref{intro12pair} and \eqref{intro03pair} are regular semisimple and   \emph{matching}, in the sense that they have the same invariant polynomial.
Suppose also that the extension $K/K_3$ is unramified.

In Definition \ref{def:orbital 03} we associate an \emph{orbital integral} $O_{(\Phi_0,\Phi_3)} (f; s, \eta )$ 
to   any  compactly supported function
\begin{equation}\label{test function}
f : \GL_{2h}(\co_F) \backslash \GL_{2h}(F) / \GL_{2h}(\co_F) \to \C.
\end{equation}
Here $\eta : K_3^\times \to \{\pm 1\}$ is the unramified quadratic character determined by  $K/K_3$,
and $s$ is a complex variable.  We propose two conjectures on the behavior of this orbital integral at $s=0$.

Suppose first that the central simple algebra in \eqref{intro12pair} is $B=M_{2h}(F)$.  
In Definition \ref{def:orbital 12} we associate to \eqref{test function} another orbital integral 
$O_{(\Phi_1,\Phi_2)} (f )$, and conjecture that
\begin{equation}\label{intro value}
 O_{(\Phi_0,\Phi_3)} (f; 0, \eta ) =  \pm  O_{(\Phi_1,\Phi_2)} (f ).
\end{equation}
This is the \emph{biquadratic fundamental lemma} of Conjecture \ref{conj:BFL}.
The ambiguity in sign arises because  $O_{(\Phi_0,\Phi_3)} (f; s, \eta )$  depends on some additional choices, which make its value at $s=0$ well-defined only up to $\pm 1$.
When $K_1=K_2$ the biquadratic fundamental lemma is equivalent to the Guo-Jacquet fundamental lemma proposed in  \cite{Guo}.

Now suppose that the central simple algebra in \eqref{intro12pair} is the division algebra
$B=\End_{\co_F}(\mathscr{G}) \otimes F$ of \eqref{introLTpair}.
Because $\End_{\co_F}(\mathscr{G})$ is a (noncommutative) discrete valuation ring, the pair $(\Phi_1,\Phi_2)$ of  \eqref{intro12pair} automatically satisfies the  integrality conditions of \eqref{introLTpair}. 
In this case, we show in Proposition \ref{prop:orbital vanish} that the matching of $(\Phi_0,\Phi_3)$ with $(\Phi_1,\Phi_2)$ implies the vanishing of 
$O_{(\Phi_0,\Phi_3)} (f; s, \eta )$ at $s=0$ for all $f$.
When $f=\mathbf{1}$ is the characteristic function of $\GL_{2h}(\co_F)$, the work of Zhang \cite{Zhang}, its descendants \cite{LZ,Mih,RTZ,Zhang2},
 and the relative trace formula approach to the Gross-Kohnen-Zagier theorem \cite{HS}, suggests one should have an \emph{arithmetic biquadratic fundamental lemma} of (roughly) the form
 \begin{equation}\label{intro derivative}
  \frac{d}{ds} O_{(\Phi_0,\Phi_3)} ( \mathbf{1} ; s ,\eta)  \big|_{s=0} 
 \stackrel{?}{=}   I(\Phi_1,\Phi_2) \log(q).
\end{equation}

The  equality \eqref{intro derivative} was conjectured by the authors in an earlier version of this paper, but it was explained to us by Andreas Mihatsch that this is a bit too naive.  The correct conjecture should involve identifying  the Lubin-Tate space $X=X^{(0)}$ as one connected component of a larger Rapoport-Zink space
\[
X^\bullet = \bigsqcup_{\ell \in \Z} X^{(\ell)},
\]
extending the definition of the cycles $Y_i \to X$ to cycles $Y_i^\bullet \to X^\bullet$, and replacing the right hand side of \eqref{intro derivative} with an intersection number taking all connected components into account 
(carefully, as simply adding the intersection multiplicities on all components will always result in $\infty$).  
The reader can find the precise statement of this corrected version of \eqref{intro derivative} stated as Conjecture \ref{conj:ABFL}.

\begin{remark}
When $h\in \{1,2\}$, the authors have verified  both the biquadratic fundamental lemma (at least when $f = \mathbf{1}$ is the characteristic function of $\GL_{2h}(\co_F)$) and the arithmetic biquadratic fundamental lemma \eqref{intro value}.
The calculations for  $h=1$ appear in \S \ref{one case}.
Calculations for $h=2$   will appear in  future work.
\end{remark}


\subsection{Global analogues}


The results and conjectures of this work are purely local in nature, but we wish to give the reader at least some  indication of their global analogues.  As this is purely for motivational purposes, the following discussion will be somewhat impressionistic.  

The global problem to which the biquadratic fundamental lemma \eqref{intro value} should be applied is purely representation-theoretic.  Suppose we start with a global field $F$ and a central simple $F$-algebra $B$ of dimension $4h^2$.
Let $K_1$ and $K_2$ be quadratic \'etale $F$-algebras,  let $K_0$ and $K_3$ be as in \eqref{biquadratic2}, and suppose we are given pairs of $F$-algebra embeddings $(\Phi_1,\Phi_2)$ and $(\Phi_0,\Phi_3)$ as in \eqref{intro12pair} and \eqref{intro03pair}.


Let $\mathbb{A}$ be the adele ring of $F$.
Suppose we are given a cuspidal automorphic representation $\pi^B$ of $B_\mathbb{A}^\times$, 
and let $\pi$ be its Jacquet-Langlands lift to $\GL_{2h}(\mathbb{A})$.   
  The idea is that the biquadratic fundamental lemma should imply period relations of the form
\[
\int_{[ H_1]}  f_1 (h_1)\, dh_1  \int_{[ H_2]}  \overline{f_2 (h_2)} \, dh_2
= 
\int_{[ H_0]}  f_0(h_0)  \, dh_0  \int_{[ H_3]}  \overline{f_3(h_3)} \cdot \eta(h_3)  \, dh_3
\]
for suitable  $f_1,f_2 \in \pi^B$ and $f_0,f_3\in \pi$.  Here
\[
H_1,H_2 \subset B^\times \quad \mbox{and}\quad H_0,H_3 \subset \GL_{2h}(F)
\]
are the centralizers of the various $\Phi_i$'s,  so that $H_i \iso \mathrm{GL}_{2h}(K_i)$, and 
\[
[H_i]= H_i(F) \backslash H_i(\mathbb{A})/ \mathbb{A}^\times.
\]
The character $\eta : H_3(\mathbb{A}) \to \{ \pm 1\}$  is determined by the extension  $K/K_3$, as in \eqref{H3 character}.
In the special case  where  $B=M_2(F)$, $F$ has positive characteristic, and $\pi$ is unramified, such a period relation appears as \cite[Theorem D]{HS}.  

It is expected (and sometimes known) that the four periods above are related to special values of $L$-functions \cite{FJ, FMW, Guo}.
For example, the period integral over $[H_0]$ is related to the central value of the standard $L$-function of $\pi$.

Now we turn to the global analogue of the arithmetic biquadratic fundamental lemma \eqref{intro derivative}. 
Here one expects a formula relating central derivatives of $L$-functions to intersections of special cycles on unitary Shimura varieties, in the spirit of the Gross-Zagier \cite{GZ} and (especially) Gross-Kohnen-Zagier  \cite{GKZ} theorems on  the  N\'eron-Tate pairings of Heegner points in a modular Jacobian.

Take $F$ to be a quadratic imaginary field,  and let $K_1$ and $K_2$ be  quartic CM fields, each containing $F$.
Suppose we are given a hermitian space $W_i$ over $K_i$ of dimension $h$, whose signature at one archimedean place of its maximal totally real subfield is $(h-1,1)$, 
and whose signature at the other place is $(h,0)$.  One can associate to the unitary group  $U(W_i)$ a Shimura variety $Y_i$ of dimension $h-1$.   

If we view each $W_i$ as an $F$-vector space, and apply $\mathrm{Tr}_{K_i/F}$ to the hermitian form, we obtain a hermitian space $V_i$ over $F$ of signature $(2h-1,1)$.  To the unitary group  $U(V_i)$ we can associate a Shimura variety $X_i$ of dimension $2h-1$, and the inclusion $U(W_i) \subset  U(V_i)$  defines a  morphism
$Y_i \to X_i$.

 Suppose further that $V_1\iso V_2$, so that $X_1\iso X_2$.
 Call this common hermitian space $V$, and this common Shimura variety $X$.
We now have  two cycles $Y_1$ and $Y_2$ of codimension $h$ on  $X$.

The idea, roughly speaking, is that the Beilinson-Bloch height pairing of $Y_1$ and $Y_2$ in the codimension $h$ Chow group of $X$  should be related to derivatives of $L$-functions.  For certain cuspidal automorphic representations $\pi$ of $U(V)$ one should able to project $Y_1$ and $Y_2$ onto the $\pi$-component of the Chow group,
and the height pairing of these two projections should essentially be the central derivative of the standard $L$-function of $\pi$, multiplied by a period of $\pi$ over a smaller unitary group $U(W_3) \subset U(V)$, where 
 $W_3$ is a hermitian space of dimension $h$ over $K_3$.  The main result of \cite{HS} provides some evidence that such a relation should hold.

If one chooses a prime $p$ that is split in $F$, but with each prime above it nonsplit in both $K_1$ and $K_2$,
the above global cycles $Y_1,Y_2 \to X$ become  (after applying the Rapoport-Zink uniformization theorem)  the cycles on Lubin-Tate space whose intersection is the subject of \eqref{intro derivative}.


\subsection{Acknowledgements}


We thank Andreas Mihatsch for helpful correspondence, and the anonymous referee for a careful reading and constructive comments.


\section{Invariants of algebra embeddings}


In this section only,  we allow $F$ to be any field whatsoever.  
We will  attach a polynomial to any pair of $F$-algebra embeddings $(\Phi_1,\Phi_2)$  as in \eqref{intro12pair}.
When the pair is regular semisimple, in a sense we will make precise, the pair is determined  up to conjugacy by this polynomial.


\subsection{Noether-Skolem plus epsilon}


Suppose  $A$ is a semisimple  $F$-algebra of finite dimension; in other words, 
\[
A \iso A_1 \times \cdots\times A_r
\]
where each $A_i$ is a finite dimensional simple $F$-algebra (whose center may be strictly larger than $F$).
Let $B$ be a central simple $F$-algebra.

\begin{theorem}[Noether-Skolem]
\label{thm:ns}
Suppose $\varphi,\varphi' : A\to B$ are $F$-algebra homomorphisms, let $B(\varphi)=B$ with its left $A$-module structure induced by $\varphi$, and let  $B(\varphi')=B$ with its left $A$-module structure induced by $\varphi'$.
If \[ B(\varphi) \iso B(\varphi')\]  as $A$-modules, then $\varphi$ and $\varphi'$ are conjugate by an element of $B^\times$.
\end{theorem}

\begin{proof}
If $A$ is a simple algebra then the hypothesis $B(\varphi) \iso B(\varphi')$ is automatically satisfied, and this is the usual Noether-Skolem theorem \cite[Tag 074Q]{stacks-project}.  
The proof of this mild generalization follows the same argument, and we leave the details as an exercise to the reader.
\end{proof}

Suppose $L$ is a finite \'etale $F$-algebra; in other words,  a finite product of finite separable field extensions.

\begin{definition}\label{def:embedding}
By an  \emph{$F$-algebra embedding} $L \to B$ we mean  an injective map of $F$-algebras making $B$ into a free $L$-module.   
\end{definition}

\begin{corollary}\label{cor:ns}
Any two $F$-algebra embeddings $L \to B$ are $B^\times$-conjugate.
\end{corollary}

\begin{proof}
Take $A=L$ in Theorem \ref{thm:ns}.
\end{proof}

\begin{remark}
The corollary   is false if we drop the freeness condition in Definition \ref{def:embedding}. 
For example,  the  $F$-algebra maps $F\times F \to M_3(F)$ defined by 
\[
(x,y) \mapsto \begin{pmatrix} x \\ & x \\ & & y  \end{pmatrix}  \quad \mbox{and}\quad
(x,y) \mapsto  \begin{pmatrix} x \\ & y \\ & & y  \end{pmatrix} 
\]
are not $\GL_3(F)$-conjugate.
\end{remark}


\subsection{Invariant polynomials}
\label{ss:polynomial}


Fix a positive integer $h$, and let $B$ be a central simple $F$-algebra of dimension $4h^2.$
Let $K_1$ and $K_2$ be  quadratic \'etale extensions of $F$ as in \eqref{biquadratic}.
Our goal  is to attach to any pair of  $F$-algebra  embeddings
\begin{equation}\label{initial embeddings}
\Phi_1 : K_1 \to B,\qquad  \Phi_2: K_2 \to B .
\end{equation}
 a degree $h$  monic polynomial with coefficients in $K_3$.
 This generalizes constructions of  \cite{HS} in the special case $h=1$, and constructions of  \cite{Guo, Li} in the special case $K_1 \iso K_2$.  See \S \ref{ss:alt invariants}.

The construction of this polynomial uses the $K_3$-algebra
\[
C=B\otimes_F K_3.
\]
 The   canonical  maps 
 \begin{equation}\label{extension of scalars}
 K_1\otimes_F K_3 \map{ x_1 \otimes x_3\mapsto x_1x_3}  K ,\qquad K_2\otimes_F K_3 \map{ x_2 \otimes x_3\mapsto x_2x_3}  K
 \end{equation}
are both isomorphisms, and so the $F$-algebra embeddings \eqref{initial embeddings} extend uniquely to $K_3$-algebra embeddings
 \begin{equation}\label{induced embeddings}
 \Phi_1,\Phi_2 : K \to C.
 \end{equation}

\begin{lemma}
If $y\in K$ is a  $K_3$-algebra generator then
\begin{align}
\Phi_1(y) - \Phi_2(y^{\tau_3}) & = \Phi_2(y) - \Phi_1(y^{\tau_3}), \label{intertwine 1} 
\end{align}
and for all $x\in K$ we have
\begin{align}
(\Phi_1(y) - \Phi_2(y^{\tau_3}) ) \cdot \Phi_2(x)  & = \Phi_1(x)  \cdot (\Phi_1(y) - \Phi_2(y^{\tau_3}) ) \label{intertwine 2} \\
(\Phi_1(y) - \Phi_2(y^{\tau_3}) ) \cdot \Phi_1(x)  & =  \Phi_2(x)  \cdot (\Phi_1(y) - \Phi_2(y^{\tau_3}) ) \nonumber\\
(\Phi_1(y) - \Phi_2(y) ) \cdot \Phi_2(x)  & = \Phi_1(x^{\tau_3} )  \cdot (\Phi_1(y) - \Phi_2(y) )\nonumber \\
(\Phi_1(y) - \Phi_2(y) ) \cdot \Phi_1(x)  & = \Phi_2(x^{\tau_3} )  \cdot (\Phi_1(y) - \Phi_2(y) ) . \nonumber 
\end{align}
\end{lemma}

\begin{proof}
The two maps \eqref{induced embeddings} have the same restriction to $K_3$, and so  
\begin{align*}
\Phi_1(y+y^{\tau_3}) & = \Phi_2(y + y^{\tau_3}) \\
\Phi_1(y y^{\tau_3})  & = \Phi_2(y y^{\tau_3}).
\end{align*}
The relation  \eqref{intertwine 1}  is clear from the first of these.
If we write $x\in K$  as $x=a y+b$ with $a,b\in K_3$, then the first equality of \eqref{intertwine 2} follows from
\begin{align*}
(\Phi_1(y) - \Phi_2(y^{\tau_3}) ) \cdot \Phi_2(y) 
& =  \Phi_1(y)\Phi_2(y) - \Phi_2(y y^{\tau_3})\\
& =  \Phi_1(y)\Phi_2(y) -\Phi_1(yy^{\tau_3}) \\
& = \Phi_1(y)  \cdot (\Phi_1(y) - \Phi_2(y^{\tau_3}) ).
\end{align*}
The other equalities in  \eqref{intertwine 2}  are proved in the same way.
\end{proof}

Fix a $y\in K$ such that $K = K_3[y]$.  Noting that $(y-y^{\tau_3})^2 \in K_3^\times$,  define 
\begin{equation}\label{better s}
\mathbf{s} 
= \frac{ (  \Phi_1(y) - \Phi_2(y^{\tau_3}) )^2  }{(y-y^{\tau_3})^2} \in C
\end{equation}
and
\begin{equation}\label{better t}
\mathbf{t} = \frac{   \Phi_1(y)\Phi_2(y) - \Phi_2(y) \Phi_1(y)  }{(y-y^{\tau_3})^2 } \in C.
\end{equation}

\begin{proposition}\label{prop:st}
The elements \eqref{better s} and \eqref{better t} do not depend on the choice of $K_3$-algebra generator $y\in K$,  and satisfy
 \begin{equation}\label{st commutation}
  \mathbf{s} \cdot \Phi_i(x)=\Phi_i(x)   \cdot \mathbf{s},\qquad 
 \mathbf{t} \cdot \Phi_i(x)=\Phi_i(x^{\tau_3})   \cdot \mathbf{t}
 \end{equation}
 for all $x \in K$. In particular, the first equality implies that $\mathbf{s}\mathbf{t} = \mathbf{t} \mathbf{s}$.
 \end{proposition}

\begin{proof}
For the independence of $\mathbf{s}$ and $\mathbf{t}$ on $y$, we will actually prove something slightly stronger:
the elements 
 \begin{equation}\label{general s}
\mathbf{s} = \frac{  - ( y_1 y_2^{\tau_3} + y_2 y_1^{\tau_3}) + \Phi_1(y_1) \Phi_2(y_2) + \Phi_2(y_2^{\tau_3}) \Phi_1(y_1^{\tau_3} ) }{ ( y_1-y_1^{\tau_3} )( y_2-y_2^{\tau_3} ) }   
\end{equation}
and
\begin{equation}\label{general t}
\mathbf{t} =\frac{  \Phi_1(y_1) \Phi_2(y_2)  - \Phi_2(y_2) \Phi_1(y_1)  }{ ( y_1-y_1^{\tau_3} )( y_2-y_2^{\tau_3} )  } 
\end{equation}
are independent of the choices of $K_3$-algebra generators $y_1,y_2\in K$, and agree with \eqref{better s} and \eqref{better t}.
  Indeed,  by direct calculation one can see that the right hand sides of \eqref{general s} and \eqref{general t}  are unchanged by substitutions of the form
$y_1 \mapsto a_1 y_1+b_1$ and $y_2 \mapsto a_2 y_2+b_2$  with $a_1,a_2\in K_3^\times$ and $b_1,b_2\in K_3$, and hence both are independent of the choices of $y_1$ and $y_2$.  
If we  set $y_1=y_2=y$  then  \eqref{general s}  simplifies  to 
\[
\mathbf{s} =  \frac{ (  \Phi_1(y) - \Phi_2(y^{\tau_3}) ) \cdot (  \Phi_2(y) - \Phi_1(y^{\tau_3}) )  }{(y-y^{\tau_3})^2}
\]
which is equal to \eqref{better s} by \eqref{intertwine 1}.  Similarly,  the right hand side of   \eqref{general t} simplifies to \eqref{better t}.

 It follows from  \eqref{intertwine 2} that $\Phi_i(x)$ commutes with 
$(  \Phi_1(y) - \Phi_2(y^{\tau_3}) )^2$, and so also commutes with $\mathbf{s}$. This proves the first equality in \eqref{st commutation}.
 The second is proved in the same way, after noting that  \eqref{better t} can be rewritten as
 \begin{equation}\label{another t}
 \mathbf{t} 
	=\frac{   (\Phi_1(y)-\Phi_2(y))    (\Phi_1(y)-\Phi_2(y^{\tau_3}))  }{   (y-y^{\tau_3})^2 }. \qedhere
\end{equation}
\end{proof}

\begin{proposition}\label{prop:square}
 The reduced characteristic polynomial  $P_{\mathbf{s}} \in K_3[T]$ of $\mathbf{s}\in C$ is a square, and if we write
 $P_{\mathbf{s}} = Q_\mathbf{s}^2$ then  $Q_{\mathbf{s}}(\mathbf{s})=0$ in $C$.
\end{proposition}

\begin{proof}
The claim can be checked after extending scalars to a separable extension, so we may assume that $F$ itself is separably closed.  Fix  isomorphisms $K_1 \iso F\times F$ and $B \iso M_{2h}(F)$ in such a way that $\Phi_1 : K_1 \to M_{2h}(F)$ is 
\[
\Phi_1(a,b) = \begin{pmatrix} a I \\ & b I   \end{pmatrix}
\]
where $I\in M_h(F)$ is the identity matrix.  
Using  \eqref{st commutation}, one can see that    $\mathbf{s} , \mathbf{t}  \in C\iso M_{2h}(K_3)$  have the form 
\[
\mathbf{s} =  \begin{pmatrix} \mathbf{s}_+ \\ & \mathbf{s}_- \end{pmatrix} 
\qquad \mathbf{t} = \begin{pmatrix} & \mathbf{u} \\  \mathbf{v}&  \end{pmatrix} 
\]
for some $\mathbf{u} ,\mathbf{v},\mathbf{s}_\pm   \in M_h(K_3)$.  The condition that $\mathbf{s}$ and $\mathbf{t}$ commute translates to  
$\mathbf{s}_+ \mathbf{u} = \mathbf{u} \mathbf{s}_-$ and  $\mathbf{v}   \mathbf{s}_+ =  \mathbf{s}_-\mathbf{v}$.

 Let $P_\mathbf{s}^\pm$ be the characteristic polynomial of $\mathbf{s}_\pm$.
  As $P_\mathbf{s} = P^+_\mathbf{s}P^-_\mathbf{s}$, to prove the proposition it suffices to prove 
$
 P^+_\mathbf{s}=P^-_\mathbf{s} . 
$
 If   $\mathbf{t}\in C$ is invertible then so are $\mathbf{u}$ and $\mathbf{v}$, hence  $\mathbf{s}_+$ and $\mathbf{s}_-$ are similar.  The general case is by a Zariski density argument.  For any $g \in  \GL_{2h}(F)$ we may form the elements  
 $
 \mathbf{s} , \mathbf{t}  \in M_{2h}(K_3)
 $
  associated to the pair $(\Phi_1 , g \Phi_2 g^{-1})$, and view them as functions of $g$. 
 We have seen that  $P_\mathbf{s}^+ =  P_\mathbf{s}^-$ for all $g$ in the open dense subset defined by 
 $\det( \mathbf{t} ) \neq 0$,
%
so the same equality also holds at $g=I$.  
 \end{proof}

We define the invariant polynomial of  $(\Phi_1,\Phi_2)$ to be the degree $h$ polynomial $Q_\mathbf{s}$ of 
Proposition  \ref{prop:square}.  In other words:

\begin{definition}\label{def:invariant polynomial}
The   \emph{invariant polynomial}   
\begin{equation}\label{invariant}
\Inv_{(\Phi_1,\Phi_2)}  \in K_3[T]
\end{equation}
of the pair \eqref{initial embeddings} is the unique monic square root of the reduced characteristic polynomial of $\mathbf{s}\in C$.
\end{definition}

\begin{remark}\label{rem:conj invariant}
If  $\Phi_1$ and $\Phi_2$ are simultaneously conjugated by an element of $ B^\times$, then $\mathbf{s},\mathbf{t}\in C$ are  simultaneously conjugated by that same element.
Thus  the invariant polynomial only depends on the $B^\times$-conjugacy class of  $(\Phi_1,\Phi_2)$.
  \end{remark}

\begin{remark}
The centralizer  $C(\Phi_i)$ of the image of $\Phi_i : K \to C$ is a central simple $K$-algebra of rank $h^2$,
and $\mathbf{s} \in C(\Phi_i)$ by \eqref{st commutation}.   The same argument used in the proof of Proposition \ref{prop:square} shows that the invariant polynomial  is equal to  reduced characteristic polynomial of $\mathbf{s}$ when viewed as an element of either one of the $K$-algebras $C(\Phi_1)$ or $C(\Phi_2)$.
\end{remark}

\begin{remark}\label{screwy 1}
We always allow the possibility that $K_1= K_2$, but be warned:
when  $K_1=K_2$ the two isomorphisms in \eqref{extension of scalars} are not equal, as they arise from embeddings $K_1\to K$ and $K_2\to K$ with different images.  
Because of this, even if the two $F$-algebra maps in  \eqref{initial embeddings} are  equal, the $K_3$-algebra maps in \eqref{induced embeddings} will not be.  
\end{remark}


\subsection{The functional equation}


We will show that the  invariant polynomial $\Inv_{(\Phi_1,\Phi_2)}(T) \in K_3[T]$ 
of Definition \ref{def:invariant polynomial}  satisfies a functional equation in $T\mapsto 1-T$.

By mild abuse of notation, we denote again by $\sigma_3$ the $F$-linear automorphism $\mathrm{id}\otimes \sigma_3$ of $C=B\otimes_F K_3$.

\begin{lemma}\label{lem:s conj}
The element $\mathbf{s}\in C$ of \eqref{better s} satisfies 
\[
\mathbf{s}^{\sigma_3}  = \frac{   (\Phi_1(y)  - \Phi_2(y))^2 }{ ( y-y^{\tau_3})^2 }.
\]
\end{lemma}

\begin{proof}
Recall from Proposition \ref{prop:st} that $\mathbf{s}$ is independent of the choice of $K_3$-algebra generator $y\in K$ in \eqref{better s}.
As $y^{\tau_1}$ is also such a generator, we therefore have 
\[
\mathbf{s} = \frac{ (  \Phi_1(y^{\tau_1}) - \Phi_2(y^{\tau_2}) )^2  }{(y^{\tau_1}-y^{\tau_2})^2}.
\]
Now note that every $a \in K$ satisfies $\Phi_i(a)^{\sigma_3} = \Phi_i(a^{\tau_i})$, so that 
\[
 \big[ \Phi_1(y^{\tau_1}) - \Phi_2(y^{\tau_2}) \big]^{\sigma_3} = \Phi_1(y) - \Phi_2(y),
\]
and that  $\sigma_3 = \tau_1|_{K_3}$, so that 
\[
\big[ (y^{\tau_1}-y^{\tau_2})^2 \big] ^{\sigma_3}=\big[ (y^{\tau_1}-y^{\tau_2})^2 \big] ^{\tau_1} =  (y-y^{\tau_3})^2.
\]
Applying $\sigma_3$ to the above expression for $\mathbf{s}$ therefore proves the claim.
\end{proof}

\begin{proposition}\label{prop:st properties}
The elements $\mathbf{s}$ and $\mathbf{t}$ satisfy
\[
 \mathbf{s}+\mathbf{s}^{\sigma_3}=1 , \qquad \mathbf{s}\mathbf{s}^{\sigma_3}+\mathbf{t}^2=0 .
 \]
\end{proposition}

\begin{proof}
Abbreviate $\alpha= \Phi_1(y)  - \Phi_2(y)$ and,  recalling \eqref{intertwine 1},
\[
\beta= \Phi_1(y) - \Phi_2( y^{\tau_3} ) = \Phi_2(y) -\Phi_1(y^{\tau_3}).
\]
The relation \eqref{intertwine 2} shows that 
\[
\alpha\beta = \Phi_1(y) \beta - \Phi_2(y)\beta = \beta \Phi_2(y) - \beta \Phi_1(y) = - \beta \alpha,
\]
from which it follows that
$
\alpha^2+\beta^2=(\alpha+\beta)^2 =  (y-y^{\tau_3})^2.
$
Combining this with   Lemma \ref{lem:s conj} yields
\[
 \mathbf{s} + \mathbf{s}^{\sigma_3} 
 = \frac{ \beta^2 }{ (y-y^{\tau_3})^2  }  + \frac{ \alpha^2  }{  (y-y^{\tau_3})^2   }   = 1.
\]
In particular $\mathbf{s}$ and $\mathbf{s}^{\sigma_3}$ commute.  
Using the formula for $\mathbf{t}$ found in  \eqref{another t},  we obtain
\[
 \mathbf{s}^{\sigma_3}\mathbf{s}+\mathbf{t}^2 
= \frac{    \alpha^2  \beta^2 } {  (y-y^{\tau_3})^4 } + \frac{  \alpha\beta \alpha\beta }{   (y-y^{\tau_3})^4 }
=  \frac{  \alpha( \alpha\beta+\beta\alpha)\beta }{   (y-y^{\tau_3})^4 }=0,
\]
completing the proof.
\end{proof}

\begin{proposition}\label{prop:functional}
The invariant polynomial $\Inv=\Inv_{(\Phi_1,\Phi_2)}$  satisfies the functional equation
\begin{equation*}
(-1)^h  \cdot \Inv (1-T) =  \Inv^{\sigma_3}  (T),
\end{equation*}
where $\Inv^{\sigma_3}  $ is obtained from $\Inv$ by applying the nontrivial automorphism $\sigma_3\in \Aut(K_3/F)$ coefficient-by-coefficient.
\end{proposition}

\begin{proof}
In general,  if  $a,b\in C$ satisfy $a+b=1$, then their  reduced characteristic polynomials $P_a , P_b  \in K_3[T]$ satisfy  $P_a(1-T) =  P_b(T)$.
 Indeed, after extending scalars one may assume that $C \iso M_{2h}(K_3)$,  and that $a$ and $b$ 
 are upper triangular matrices.  In this case the claim is obvious.

Applying this  to the relation $\mathbf{s}+\mathbf{s}^{\sigma_3}=1$ of Proposition \ref{prop:st properties} shows that 
\[
 P_{\mathbf{s}}(1-T) = P_{\mathbf{s}}^{\sigma_3} (T) ,
\] 
and taking the unique monic square root of each side proves the claim.
\end{proof}


\subsection{The elements $\mathbf{w}$ and $\mathbf{z}$}
\label{ss:wz}


The elements $\mathbf{s}$ and $\mathbf{t}$ defined by \eqref{better s} and \eqref{better t}, being independent of the choice of $y\in K$ used in their construction, are canonically attached to the pair of embeddings \eqref{initial embeddings}.  However, they have the disadvantage that they live in the base change $C=B\otimes_F K_3$ rather than in $B$ itself.  

We now  construct  substitutes for $\mathbf{s}$ and $\mathbf{t}$ that live in $B$, but which depend on  noncanonical choices.   Namely, fix  $F$-algebra generators $x_1\in K_1$ and $x_2\in K_2$,  and define 
\begin{align}
	\mathbf{w} & =  \Phi_1(x_1)\Phi_2(x_2)+\Phi_2(x_2^{\sigma_2})\Phi_1(x_1^{\sigma_1})   \in B  \label{w def}\\
\mathbf{z} & =  \Phi_1(x_1)\Phi_2(x_2)-\Phi_2(x_2)\Phi_1(x_1)  \in B . \label{z def}
\end{align}
If  we  view both $x_1=x_1\otimes 1$ and $x_2=1\otimes x_2$ as elements of $K=K_1 \otimes K_2$,  then
\[
x_1x_2^{\sigma_2}+x_2x_1^{\sigma_1}    \in K_3 \qquad\mbox{and} \qquad
(x_1-x_1^{\sigma_1})(x_2-x_2^{\sigma_2})  \in K_3^\times,
\]
and  so both may be viewed as central elements in $C$.  Similarly, we  view $\mathbf{w},\mathbf{z} \in C$.

\begin{proposition}\label{prop:alt st}
The elements \eqref{better s} and \eqref{better t} satisfy
\begin{align*} 
	\mathbf{s}   = \frac{-(x_1x_2^{\sigma_2}+x_2x_1^{\sigma_1})+  \mathbf{w}  }{(x_1-x_1^{\sigma_1})(x_2-x_2^{\sigma_2})}  \quad \mbox{and} \quad
	\mathbf{t}   = \frac{ \mathbf{z} }{(x_1-x_1^{\sigma_1})(x_2-x_2^{\sigma_2})}   . 
\end{align*}
\end{proposition}

\begin{proof}
Take $y_1=x_1$ and $y_2=x_2$ in \eqref{general s} and \eqref{general t}. 
\end{proof}

\begin{proposition}\label{prop:wz properties}
The elements $\mathbf{w}$ and  $\mathbf{z}$ commute,  and satisfy 
\[
  \mathbf{w} \cdot \Phi_i(x)   =\Phi_i(x)   \cdot \mathbf{w} , \qquad 
 \mathbf{z} \cdot \Phi_i(x)    =\Phi_i(x^{\sigma_i})   \cdot \mathbf{z}
\]
 for all $x \in K_i$.  Moreover, they are related by 
\[
 \mathbf{z}^2
= \mathbf{w}^2 - \mathrm{Tr}(x_1)\mathrm{Tr}(x_2)\mathbf{w} + \mathrm{Tr}(x_1^2)\mathrm{Nm}(x_2)+\mathrm{Tr}(x_2^2)\mathrm{Nm}(x_1) ,
\]
where $\mathrm{Tr}$ and $\mathrm{Nm}$ denote   the trace and norm from $K_1$ or $K_2$ to $F$, as appropriate.
\end{proposition}

\begin{proof}
Everything except the final claim follows from  Proposition \ref{prop:alt st} and the analogous properties of $\mathbf{s}$ and $\mathbf{t}$ proved in Proposition \ref{prop:st}.  
For the final claim define elements $u,v \in K_3$ by 
\[
u = x_1x_2^{\sigma_2}+x_2x_1^{\sigma_1} \quad \mbox{and}\quad  v = (x_1-x_1^{\sigma_1})(x_2-x_2^{\sigma_2}),
\]
so that $v \mathbf{s} = \mathbf{w} -u$  and $  \mathbf{z}^2  =  v^2 \mathbf{t}^2$ by Proposition \ref{prop:alt st}.
 Using the relation $\mathbf{s}\mathbf{s}^{\sigma_3}=-\mathbf{t}^2$ from Proposition \ref{prop:st properties}  and  $v^{\sigma_3} =v^{\tau_1}= -v$, we compute
\[
  \mathbf{z}^2  =  v^2 \mathbf{t}^2 
    = ( v \mathbf{s} )( v  \mathbf{s})^{\sigma_3} =    ( \mathbf{w} - u )( \mathbf{w} - u ^{\sigma_3 } ) .
\]
To obtain the desired formula for $\mathbf{z}^2$,  expand  the right hand side  and use the relation
$u^{\sigma_3} = u^{\tau_1}= x^{\sigma_1}_1x_2^{\sigma_2}+ x_2 x_1$.
\end{proof}


\subsection{Regular semisimple pairs}


We have already noted in Remark \ref{rem:conj invariant} that the invariant polynomial of the pair $(\Phi_1,\Phi_2)$ 
fixed in  \S \ref{ss:polynomial}  depends only on its $B^\times$-conjugacy class. 
 In this subsection we will show that for a pair that is regular semisimple, in the following sense, the invariant polynomial of Definition \ref{def:invariant polynomial} determines the conjugacy class.

\begin{definition}\label{def:rss}
The pair  $(\Phi_1,\Phi_2)$ is \emph{regular semisimple} if $\mathbf{s} \in C^\times$  and  $K_3[\mathbf{s}] \subset C$ is  an \'etale $K_3$-subalgebra of dimension $h$.
\end{definition}

\begin{remark}
In Proposition \ref{prop:truerss} we will show that Definition \ref{def:rss} is compatible with the usual notion of regular semisimple from geometric invariant theory.
\end{remark}

\begin{remark}\label{rem:ssunits}
Using the relation $\mathbf{t}^2 = - \mathbf{s}\mathbf{s}^{\sigma_3}$ of Proposition \ref{prop:st properties}, we see that $\mathbf{s} \in C^\times$ if and only if  $\mathbf{t}\in C^\times$.
\end{remark}

The following proposition shows that one can characterize regular semisimple pairs using the invariant polynomial.

\begin{proposition}\label{prop:rss}
The pair $(\Phi_1,\Phi_2)$ is regular semisimple if and only if for every  $F$-algebra map $\rho : K_3 \to F^\alg$ the image  
\begin{equation*}
\Inv_{(\Phi_1,\Phi_2)}^\rho (T)  \in F^\alg[T]
\end{equation*}
of the invariant polynomial has $h$ distinct nonzero roots.
\end{proposition}

\begin{proof}
Recall from Proposition \ref{prop:square}  that $\mathbf{s} \in C$ is a zero  of the invariant polynomial
$Q_\mathbf{s} = \Inv_{(\Phi_1,\Phi_2)}$, and so there is a surjection
\begin{equation}\label{s presentation}
K_3[T] / ( Q_\mathbf{s}  )  \map{T \mapsto \mathbf{s}}  K_3[\mathbf{s} ]  .
\end{equation}

 If $(\Phi_1,\Phi_2)$ is regular semisimple then the surjection  \eqref{s presentation}
is an isomorphism for dimension reasons, and so the domain is an \'etale $F$-algebra in which $T$ is a unit.
It follows that $Q_\mathbf{s}^\rho$ has $h$ distinct nonzero roots for any $\rho : K_3 \to F^\alg$.

For the converse,  suppose that $Q_\mathbf{s}^\rho$ has $h$ distinct nonzero roots for any $\rho$.
Let $M_\mathbf{s} \in K_3[T]$ be the minimal polynomial of $\mathbf{s}\in C$, so that  $M_\mathbf{s} \mid Q_\mathbf{s}$, and  \eqref{s presentation} factors as
\[
K_3[T] / ( Q_\mathbf{s}  ) \to K_3[T] / ( M_\mathbf{s} )   \map{T \mapsto \mathbf{s}}  K_3[\mathbf{s} ]  
\]
with the second arrow an isomorphism. 
By elementary linear algebra, the roots of $M_\mathbf{s}^\rho$ in $F^\alg$ are the same as the roots of the characteristic polynomial $P_\mathbf{s}^\rho = Q_\mathbf{s}^\rho \cdot Q_\mathbf{s}^\rho$ of $\mathbf{s} \in C^\rho \iso M_{2h}(F^\alg)$, which are same as the  roots of $Q_\mathbf{s}^\rho$.  
It now follows from our assumptions on $Q_\mathbf{s}$ that  $M_\mathbf{s} =Q_\mathbf{s}$, and so 
  \eqref{s presentation} is an isomorphism whose domain is an \'etale $F$-algebra of dimension $h$ in which $T$ is a unit.  Thus $(\Phi_1,\Phi_2)$ is regular semisimple.
\end{proof}

The  $F$-subalgebra generated by $\Phi_1(K_1) \cup \Phi_2(K_2) \subset B$ is denoted
\[
F(\Phi_1,\Phi_2) \subset B.
\]
Recall the elements $\mathbf{w},\mathbf{z} \in F(\Phi_1,\Phi_2)$ of \eqref{w def} and \eqref{z def}, and set
\[
L = F[\mathbf{w}] \subset B.
\]
Although the element $\mathbf{w}$ depends on the choices of $x_1\in K_1$ and $x_2\in K_2$, the subalgebra $L$ does not.
For example, using Proposition \ref{prop:alt st} we see that 
\begin{equation}\label{canonical subalgebra}
L\otimes_F K_3 = K_3[ \mathbf{w} ] = K_3 [ \mathbf{s} ] \subset C
\end{equation}
 does not depend on the choices of $x_1$ and $x_2$ (because $\mathbf{s}$ does not), and $L$ is characterized as the elements in \eqref{canonical subalgebra} fixed by the $F$-linear automorphism $\sigma_3=\mathrm{id} \otimes \sigma_3$ of $C=B\otimes_F K_3$.

\begin{proposition}\label{prop:central algebra}
If the pair  $(\Phi_1,\Phi_2)$ is regular semisimple, the following properties hold.
\begin{enumerate}
\item
The element $\mathbf{z} \in B$ is a unit.
\item
 The $F$-algebra $L$  is  \'etale  of dimension $h$.
  \item
 The centralizer of $F(\Phi_1,\Phi_2) \subset B$ is $L$.
 \item
 The centralizer of $L \subset B$ is $F(\Phi_1,\Phi_2)$.
 \item
The ring  $F(\Phi_1,\Phi_2)$ is a quaternion algebra over its center $L$.
Writing $L \iso \prod L_i$ as a product of fields,  this  means that each $F(\Phi_1,\Phi_2) \otimes_L L_i$ is a central simple $L_i$-algebra of dimension $4$.
\item
The ring $B$ is free as an $F(\Phi_1,\Phi_2)$-module.
 \end{enumerate}
\end{proposition}

\begin{proof}
For property (1), note that $\mathbf{t} \in C^\times$ by Remark \ref{rem:ssunits}, and so $\mathbf{z} \in B^\times$ by Proposition \ref{prop:alt st}.
For  property (2),   the $K_3$-algebra \eqref{canonical subalgebra} is \'etale of dimension $h$ by hypothesis, and so  $L$ is \'etale over $F$ of the same dimension.
The remaining properties  rely on the following lemma.  

\begin{lemma}\label{lem:good w}
There is a separable field extension $F'/F$ such that, abbreviating 
\[
B'=B\otimes_FF' , \qquad L'=L\otimes_F F' ,
\]
there exists an isomorphism $B' \iso M_{2h}(F')$ identifying 
 \[
L'
 = \left\{
 \begin{pmatrix} 
 x_1 I_2 & \\ 
 & \ddots \\
 & & x_h I_2
 \end{pmatrix}  : x_1,\ldots, x_h \in F'
 \right\}.
 \]
 Here $I_2  \in M_2(F')$ is the $2\times 2$ identity matrix.
\end{lemma}

\begin{proof}
Let $P_\mathbf{w} \in F[T]$ be reduced characteristic polynomial of $\mathbf{w} \in B$, and let $Q_\mathbf{w} \in F[T]$ be its minimal polynomial.   We know from Proposition \ref{prop:alt st} that 
\[
\mathbf{w} = c + d \mathbf{s} \in C
\]
for scalars $c\in K_3$ and $d\in K_3^\times$.   This implies that $ P_{\mathbf{w}}( c+dT)\in K_3[T]$ is the reduced characteristic polynomial of $\mathbf{s}$, while $Q_{\mathbf{w}}( c+dT)\in K_3[T]$ is its minimal polynomial. 
 These latter polynomials are precisely the $P_\mathbf{s}$ and $Q_\mathbf{s}$ of Proposition \ref{prop:square} (the first by definition of $P_\mathbf{s}$, and the second by the proof of Proposition \ref{prop:rss}).  As $P_\mathbf{s} = Q_\mathbf{s}^2$, we deduce that $P_\mathbf{w} = Q_\mathbf{w}^2$.

 Choose any $F'/F$ large enough that $B' \iso M_{2h}(F')$.  
 Let   $V$ be the unique simple left $B'$-module; in other words, the standard representation of $M_{2h}(F')$.
It follows from $P_\mathbf{w} = Q_\mathbf{w}^2$  that the list of invariant factors   of the matrix $\mathbf{w} \in B'$ can only be $Q_\mathbf{w} \mid Q_\mathbf{w}$.   Thus
\begin{equation}\label{jordan}
V\iso L'\oplus L'
\end{equation}
as a left modules over the subring $L'=F'[\mathbf{w}] \subset B'$.  

Using (2), we may enlarge $F'$ to assume that $L' \iso F'\times \cdots\times F'$ ($h$ factors).
If $e_1,\ldots, e_h \in L'$ are the orthogonal idempotents inducing this decomposition, it follows from 
\eqref{jordan} that 
\[
V = e_1 V \oplus \cdots \oplus e_h V
\]
with each summand a $2$-dimensional $F'$-subspace on which $L'$ acts through scalars. 
Choose a basis for each summand, and use this to identify $B' \iso \End_{F'}(V) \iso M_{2d}(F')$.
This isomorphism has the desired properties.
\end{proof}

Now we complete the proof of Proposition \ref{prop:central algebra}.
Keep the notation of the lemma, and  abbreviate
\[
F'(\Phi_1,\Phi_2) = F(\Phi_1,\Phi_2) \otimes_F F'.
\]
By Proposition \ref{prop:wz properties}, $\mathbf{w}$ commutes with the image of $\Phi_i : K_i \to B$, and hence $F(\Phi_1,\Phi_2)$ is contained in the centralizer of $L$.   Hence  $F'(\Phi_1,\Phi_2)$ is contained in the centralizer of $L'$, or, in other words, 
\begin{equation}\label{quat inclusion}
F'(\Phi_1,\Phi_2)  \subset \underbrace{ M_2(F') \times \cdots \times M_2(F') }_{h\ \mathrm{times}} ,
\end{equation}
embedded block diagonally into $M_{2h}(F')$.

The $F'$-subalgebra on the left hand side of \eqref{quat inclusion} contains
\[
L' \iso \underbrace{ F' \times \cdots \times F' }_{h\ \mathrm{times}} ,
\]
as well as $\Phi_1(K_1)$, and  a unit  $\mathbf{z} \in  F'(\Phi_1,\Phi_2)$ such that 
$
\mathbf{z}   \Phi_1(x) =  \Phi_1(x^{\sigma_1}) \mathbf{z}
$
 for all $x\in K_1$ (Proposition \ref{prop:wz properties}).  
 From this, it is not hard to see
  first  that equality holds in \eqref{quat inclusion}, and  then that   the $F'$-algebras
 \[
 L' \subset F'(\Phi_1,\Phi_2) \subset B'
 \]
satisfy the obvious analogues of properties (3)-(6).  The $F$-algebras  
 \[
 L \subset F(\Phi_1,\Phi_2) \subset B
 \]
therefore  satisfy those  properties, completing the proof.
\end{proof}

We next show that when $(\Phi_1,\Phi_2)$ is regular semisimple,  the isomorphism class of $F(\Phi_1,\Phi_2)$ as an $F$-algebra is completely determined by the invariant polynomial.

\begin{proposition}\label{prop:quaternion recovery}
Assume that $(\Phi_1,\Phi_2)$ is regular semisimple. 
If  $B'$ is a central simple $F$-algebra of the same dimension as $B$, and if  we are given $F$-algebra embeddings
\[
\Phi_1' : K_1 \to B' ,\qquad \Phi_2' : K_2 \to B'
\]
such that $\Inv_{(\Phi_1,\Phi_2)} = \Inv_{(\Phi'_1,\Phi'_2)}$, then there is an isomorphism of $F$-algebras
\[
F(\Phi_1,\Phi_2) \iso F(\Phi_1',\Phi_2')
\]
identifying $\Phi_1=\Phi_1'$ and $\Phi_2=\Phi_2'$.
\end{proposition}

\begin{proof}
The validity of the proposition does not depend on the choices of  $x_1\in K_1 $ and $x_2 \in K_2$ made in \S \ref{ss:wz}.
If $\mathrm{char}(F) \neq 2$ we choose them so that 
\[
x_1^{\sigma_1} = - x_1  \quad \mbox{ and }\quad x_2^{\sigma_2} =-x_2.
\]
If $\mathrm{char}(F)=2$, we  use the surjectivity of   $\mathrm{Tr}:K_i \to F$ to choose them so that
\[
x_1^{\sigma_1} =  x_1 +1   \quad \mbox{ and } \quad x_2^{\sigma_2} = x_2+1 .
\]
In either case, for  $i\in \{1,2\}$ we define $a_i,b_i\in F$ by 
\[
x_i^{\sigma_i} = a_i x_i + b_i \in K_i .
\]

Using only the data $\Inv_{(\Phi_1,\Phi_2)}$ and the choices of $x_1$ and $x_2$, we will  construct an abstract $F$-algebra $\mathcal{F}(\varphi_1,\varphi_2)$ and  embeddings $\varphi_i :K_i \to \mathcal{F}(\varphi_1,\varphi_2)$ in such a way that 
\[
\mathcal{F}(\varphi_1,\varphi_2) \iso F(\Phi_1,\Phi_2)
\]
and $\varphi_i$ is identified with $\Phi_i$.   Once this is done, the claim follows immediately: by symmetry (note that $(\Phi_1',\Phi_2')$ is also regular semisimple, by Proposition \ref{prop:rss} and the assumption on the matching of invariant polynomials), there is an analogous isomorphism
$
\mathcal{F}(\varphi_1,\varphi_2) \iso F(\Phi'_1,\Phi'_2).
$

Let $\mathbf{w},\mathbf{z} \in B$  be the elements  \eqref{w def} and \eqref{z def}. 
We saw in the proof of Lemma \ref{lem:good w} that the minimal polynomial $Q_\mathbf{w} \in F[T]$ of $\mathbf{w}$ is related to the invariant polynomial $\Inv_{(\Phi_1,\Phi_2)} \in K_3[T]$ by a change of variables
\[
\Inv_{(\Phi_1,\Phi_2)}(T) = Q_\mathbf{w}(c+dT),
\]
where the scalars $c\in K_3$ and $d\in K_3^\times$  depend on the choices of $x_1\in K_1$ and $x_2 \in K_2$, but not on the pair  $(\Phi_1,\Phi_2)$.   
Thus the   $F$-algebra 
\[
\mathcal{L} = F[W] / (Q_\mathbf{w}(W))
\]
of dimension $h$,  where $F[W]$ is a polynomial ring in one variable, depends only on $\Inv_{(\Phi_1,\Phi_2)}$ and the choices of $x_1$ and $x_2$.
There is a  surjective morphism $ \mathcal{L} \to L $ characterized by $W \mapsto \mathbf{w}$.

Now let 
\[
\mathcal{F}(\varphi_1,\varphi_2) = \mathcal{L}[ Z,X_1,X_2 ]
\]
  be the free $\mathcal{L}$-algebra generated by three noncommuting variables, modulo  the two-sided ideal generated by the relations 
 \begin{itemize}
   \item
$X_i^2 - \mathrm{Tr}(x_i) X_i + \mathrm{Nm}(x_i) =0$,
 \item
  $Z^2 = W^2 - \mathrm{Tr}(x_1)\mathrm{Tr}(x_2) W + \mathrm{Tr}(x_1^2)\mathrm{Nm}(x_2)+\mathrm{Tr}(x_2^2)\mathrm{Nm}(x_1)$, 
  \item
$Z X_i = (a_i X_i + b_i)  Z$,
\item
$Z  =  W -   [ a_2 X_2+ b_2] [ a_1 X_1 + b_1] - X_2X_1 $.
\end{itemize}
The first of these relations allows us to define 
\[
\varphi_i : K_i \to \mathcal{F}(\varphi_1,\varphi_2)
\]
 by $\varphi_i(x_i) = X_i$.  Moreover, there is a unique surjection 
 \begin{equation}  \label{universal quaternion}
\mathcal{F}(\varphi_1,\varphi_2)  \to F(\Phi_1,\Phi_2)  ,
\end{equation}
satisfying $W \mapsto  \mathbf{w}$,   $Z \mapsto   \mathbf{z}$, and $X_i \mapsto \Phi_i(x_i)$.
 Indeed, one only needs to check that $\mathbf{w}$, $\mathbf{z}$, and $\Phi_i(x_i)$ satisfy the same four relations as  $W$, $Z$, and $X_i$.  The first relation is clear, the second and third are found in   Proposition \ref{prop:wz properties},   and the fourth follows directly from \eqref{w def} and \eqref{z def}, which imply 
 \begin{align*}
\mathbf{w} - \mathbf{z} 
&= \Phi_2(x_2^{\sigma_2}) \Phi_1(x_1^{\sigma_1}) + \Phi_2(x_2)\Phi_1(x_1)  \\
&=   [ a_2 \Phi_2(x_2)+ b_2] [ a_1 \Phi_1( x_1) + b_1] + \Phi_2(x_2)\Phi_1(x_1)  .
\end{align*}

We next claim that   $\mathcal{F}(\varphi_1,\varphi_2)$ is generated as an $\mathcal{L}$-algebra by $Z$ and $X_1$ alone.   To see this, rewrite
\[
Z  =  W -   [ a_2 X_2+ b_2] [ a_1 X_1 + b_1] - X_2X_1 
\]
as 
\begin{equation}\label{W-Z}
W-Z = X_2 \cdot    (a_1a_2   X_1 + X_1+ b_1a_2 )  +     (a_1b_2  X_1 + b_1b_2 ),
\end{equation}
and recall our particular choices of $x_1$ and $x_2$.  If $\mathrm{char}(F) \neq 2$ then 
\[
a_1a_2  x_1 + x_1   + b_1a_2 = 2 x_1  \in K_1^\times,
\]
while if $\mathrm{char}(F) =2$ then
\[
a_1a_2  x_1 + x_1   + b_1a_2 = 1 \in K_1^\times.
\]
In either case, applying $\varphi_1$ shows that  $a_1a_2   X_1 + X_1+ b_1a_2$ is a unit in  $F[X_1]$, 
allowing us to  solve \eqref{W-Z} for $X_2$ in terms of $W$, $Z$, and $X_1$.

Finally, from the relations  satisfied by $X_1$ and $Z$   it is clear that
\[
\mathcal{L}[Z,X_1] = \mathcal{L} + \mathcal{L} X_1 +\mathcal{L} Z + \mathcal{L} X_1 Z
\]
as $\mathcal{L}$-modules, and so the $F$-dimension of $\mathcal{L}[ Z,X_1] = \mathcal{F}(\varphi_1,\varphi_2)$ is at most $4h$.  
As the  $F$-algebra $F(\Phi_1,\Phi_2)$ has dimension  $4h$ by Proposition \ref{prop:central algebra},
 the  surjection \eqref{universal quaternion} is an isomorphism.
\end{proof}

\begin{corollary}\label{cor:conjugacy}
Suppose we are given $F$-algebra embeddings
\[
\Phi_1 ,\Phi_1' : K_1 \to B,\qquad \Phi_2 ,\Phi_2': K_2 \to B
\]
such that  $(\Phi_1,\Phi_2)$ and $(\Phi_1',\Phi_2')$ are both regular semisimple.
Then $(\Phi_1,\Phi_2)$ and $(\Phi_1',\Phi_2')$  have the same invariant polynomial if and only if they lie in the same $B^\times$-conjugacy class. 
 In other words, if and only if there is a $b \in B^\times$ such that 
\begin{equation*}
b\Phi_1b^{-1} = \Phi_1'  \quad  \mbox{ and }\quad  b\Phi_2 b^{-1} = \Phi_2'.
\end{equation*}
\end{corollary}

\begin{proof}
One direction is Remark \ref{rem:conj invariant}, so assume that $(\Phi_1,\Phi_2)$ and $(\Phi_1',\Phi_2')$ have the same invariant polynomial.  By Proposition \ref{prop:quaternion recovery}, there is an isomorphism
\[
F(\Phi_1,\Phi_2) \iso F(\Phi_1',\Phi_2')
\]
identifying $\Phi_1=\Phi_1'$ and $\Phi_2=\Phi_2'$. 
By  Theorem \ref{thm:ns} (whose hypotheses are satisfied  by Proposition \ref{prop:central algebra}),  
the  inclusions $F(\Phi_1,\Phi_2) \to B$ and $F(\Phi'_1,\Phi'_2) \to B$ are $B^\times$-conjugate, and  the proposition follows. 
 \end{proof}

\begin{remark}
 If $h=1$, so that $B$ is a quaternion algebra over $F$,  the invariant polynomial has the form  
\[
\Inv_{(\Phi_1,\Phi_2)} (T) = T -\xi
\]
 for some $\xi \in K_3$ satisfying (by Proposition \ref{prop:functional})  $\mathrm{Tr}_{K_3/F}(\xi)=1$.
 The proof that  $\xi$ determines the $B^\times$-conjugacy class of $(\Phi_1,\Phi_2)$ already appears in  \cite{HS}, although the construction of $\xi$ described there is quite different.
\end{remark}


\section{The biquadratic fundamental lemma}


From this point on we assume that $F$ is a nonarchimedean local field, and fix 
 quadratic \'etale $F$-algebras  $K_1$ and $K_2$.
In this section we define two kinds of orbital integrals, and formulate a conjectural  fundamental lemma
relating them.


\subsection{Orbital integrals for $(\Phi_1,\Phi_2)$}


Given $F$-algebra embeddings
\begin{equation*}
\Phi_1 : K_1 \to M_{2h}(F),\qquad  \Phi_2: K_2 \to M_{2h}(F) ,
\end{equation*}
let $H_i \subset \GL_{2h}(F)$ be the centralizer of 
$\Phi_i(K_i)^\times$.  If we use $\Phi_i$  to view $F^{2h}$ as a $K_i$-module, the natural action of $H_i$ on $F^{2h}$ determines  isomorphisms
\begin{equation*}
H_1 \iso \GL_h(K_1)   ,\qquad H_2 \iso \GL_h(K_2),
\end{equation*}
well-defined up to conjugacy.

Assume that $(\Phi_1,\Phi_2)$ is regular semisimple, in the sense of Definition \ref{def:rss}.     
This implies, by Proposition \ref{prop:central algebra}, that  $F(\Phi_1,\Phi_2) \subset M_{2h}(F)$ has centralizer   an \'etale $F$-algebra $L \subset M_{2h}(F)$ of dimension $h$.
 In particular
\begin{equation}\label{toric intersection}
L^\times = H_1 \cap H_2
\end{equation}
is a torus.

\begin{definition}\label{def:orbital 12}
Given a compactly supported  $f$ as in \eqref{test function},  define the \emph{orbital integral} 
\begin{equation*}
O_{(\Phi_1,\Phi_2)}(f) 
= \int_{  H_1\cap H_2 \backslash H_1\times H_2  } f( g_1^{-1} h_1^{-1} h_2 g_2 ) \, dh_1\, dh_2, 
\end{equation*}
where $g_1,g_2\in \GL_{2h}(F)$ are chosen to satisfy the integrality condition
\begin{equation*}
 \Phi_i( \co_{K_i})  \subset   g_i  M_{2h}(\co_F)  g_i^{-1},
\end{equation*}
and $H_1\cap H_2 \subset  H_1\times H_2$ is embedded diagonally.
\end{definition}

\begin{remark}\label{rem:haar}
The $\Phi_i(\co_{K_i})$-stable lattices $\Lambda \subset F^{2h}$ form a single $H_i$-orbit, and 
the Haar measure on $H_i$ is normalized by
\begin{equation*}
 \mathrm{vol}( \mathrm{Stab}_{H_i} (\Lambda) ) =1 .
\end{equation*}
The Haar measure on $H_1\cap H_2$ is normalized, using  \eqref{toric intersection}, by
$\mathrm{vol}(\co_L^\times) =1$.
\end{remark}

\begin{remark}\label{rem:classic orbital}
The orbital integral is independent of the choice of $g_1$ and $g_2$.
Moreover,  if we set  $g =g_1^{-1} g_2$ and $H'_i = g_i^{-1} H_i g_i$,  the change of variables $h_i \mapsto g_i h_i g_i^{-1} $ puts the orbital integral into the more familiar form
\begin{equation*}
O_{(\Phi_1,\Phi_2)}(f) = \int_{   \{ ( h_1 , h_2) \, :\,  g h_2 = h_1 g \}    \backslash  H'_1\times H'_2 } 
f(  h_1^{-1} g h_2  ) \, dh_1\, dh_2 .
\end{equation*}
\end{remark}


\subsection{Orbital integrals  for $(\Phi_0,\Phi_3)$}


As in \S \ref{ss:intro matching}, we may set 
\begin{equation*} 
K_0=F\times F  
\end{equation*} 
and repeat the construction of \eqref{biquadratic} with  the pair $(K_1,K_2)$ replaced by $(K_0 , K_3)$. 
This gives another diagram of $F$-algebras \eqref{biquadratic2}.  Repeating the constructions of  \S \ref{ss:polynomial}, we associate to  any  pair of $F$-algebra embeddings
\begin{equation*}
\Phi_0 : K_0 \to M_{2h}(F),\qquad  \Phi_3: K_3 \to M_{2h}(F) 
\end{equation*}
  a degree $h$  monic polynomial 
$\Inv_{(\Phi_0,\Phi_3)} \in K_3[T]$ satisfying the functional equation 
\begin{equation*}
(-1)^h  \cdot \Inv_{(\Phi_0,\Phi_3)}(1-T) =  \Inv_{(\Phi_0,\Phi_3)}^{\sigma_3}  (T).
\end{equation*}

Denote by $H_i \subset \GL_{2h}(F)$  the centralizer of  $\Phi_i(K_i)^\times$.  
If we use $\Phi_i$  to view $F^{2h}$ as a $K_i$-module, then choices of $K_i$-bases  determine isomorphisms
\begin{equation}\label{03 centralizers}
H_0 \iso \GL_h(K_0)   ,\qquad H_3 \iso \GL_h(K_3).
\end{equation}
Composing the  absolute value $|\cdot| :F^\times \to \R^\times$ with the character
\begin{equation*}
H_0 \iso \GL_h(K_0) = \GL_h(F) \times \GL_h(F) \map{ (a,b) \mapsto  \frac{\det(a)}{\det(b)} } F^\times
\end{equation*}
determines a character,  again denoted 
\begin{equation}\label{H0 character}
| \cdot | : H_0 \to \R^\times.
\end{equation}
We denote by $\eta : K_3^\times \to \{ \pm 1\}$  the character associated to the \'etale quadratic extension $K/K_3$ by class field theory (if $K_3\iso F\times F$,  then class field theory associates to $K/K_3$ a quadratic character on each copy of $F^\times$, and $\eta$ is defined as their product), and note that $\eta|_{F^\times}$ is the trivial character.
  Composing $\eta$ with the determinant  $H_3 \iso \GL_h(K_3) \to K_3^\times$ yields a character, again denoted
   \begin{equation}\label{H3 character}
 \eta : H_3 \to \{\pm 1\} .
 \end{equation}

\begin{remark}\label{rem:3 character}
The character \eqref{H3 character} admits two natural extensions to $\GL_{2h}(F)$, given by the compositions
\[
 \GL_{2h}(F)  \map{\det}  F^\times \map{ \eta_{K_i/F} }    \{ \pm 1\} 
\]
for $i\in \{1,2\}$,  where $\eta_{K_i/F}$ is the quadratic character associated to $K_i/F$.
\end{remark}

\begin{remark}\label{rem:0 character}
In contrast, the character \eqref{H0 character} does not extend to $\GL_{2h}(F)$.  Indeed, if $z \in \GL_{2h}(F)$ is any element such that $z \cdot \Phi_0(x) = \Phi_0(x^{\sigma_0}) \cdot z$ for all $x\in K_0$, then conjugation by $z$ preserves $H_0$, but
\[
| z h_0 z^{-1} | = | h_0|^{-1}.
\]
\end{remark}

For the remainder of this subsection we assume that  $(\Phi_0,\Phi_3)$ is regular semisimple.
By  Proposition \ref{prop:central algebra}, the subalgebra $F(\Phi_0,\Phi_3) \subset M_{2h}(F)$ has centralizer   an \'etale $F$-algebra $L$ of dimension $h$, and hence
\begin{equation}\label{toric intersection 2}
L^\times = H_0 \cap H_3
\end{equation}
is a torus.

\begin{lemma}\label{lem:nocharacters}
The characters  \eqref{H0 character} and \eqref{H3 character} are trivial on  \eqref{toric intersection 2}.
\end{lemma}

\begin{proof}
We used choices of isomorphisms \eqref{03 centralizers} to define the characters in question, but it is easy to see that the resulting characters do not depend on these choices.  We are therefore free to make them in a particular way.

Using the embedding $\Phi_i: K_i \to M_{2h}(F)$, we may view the standard representation $F^{2h}$ as a free $K_i$-module of rank $h$.  We claim that there is an $F$-subspace $V \subset F^{2h}$ such that  the natural maps 
\[
K_0 \otimes_F V \to F^{2h} \quad  \mbox{and} \quad K_3 \otimes_F V \to F^{2h}
\]
are isomorphisms.   Indeed, using  Proposition \ref{prop:central algebra} it is not hard to see   that $F^{2h}$  is free  of rank one over both $K_0 \otimes_F L$ and $K_3 \otimes_F L$.   There is therefore a  Zariski  dense set of elements $e \in F^{2h}$ such that 
\[
( K_0 \otimes_F L) e = F^{2h} = ( K_3 \otimes_F L) e,
\]
and for any such $e$ the subspace $V= Le$ has the required properties.
 
 Let $f_1,\ldots, f_h$ be an $F$-basis for the subspace $V \subset F^{2h}$.
These same vectors form a basis for both the  $K_0$-module and $K_3$-module structures on $F^{2h}$. 
 If we use this  basis to define the isomorphisms \eqref{03 centralizers}, then both isomorphisms identify $L^\times = H_0 \cap H_3$ with a subgroup of    $\GL_h(F)$.
   It is  clear from their constructions that \eqref{H0 character} and \eqref{H3 character}  have trivial restriction to this subgroup.
\end{proof}

The preceding lemma allows us to make the following definition.

\begin{definition}\label{def:orbital 03}
For every compactly supported function \eqref{test function},
 define the \emph{orbital integral} 
\begin{equation*}
O_{(\Phi_0,\Phi_3)} (f; s, \eta ) = \int_{ H_0\cap H_3 \backslash H_0\times H_3 } f( g_0^{-1} h_0^{-1} h_3 g_3 )
\cdot | h_0 |^s \cdot \eta( h_3) \, dh_0\, dh_3, 
\end{equation*}
where $g_0,g_3\in \GL_{2h}(F)$ are chosen to satisfy the integrality conditions
\begin{equation}\label{integrality shift 03}
 \Phi_i( \co_{K_i})  \subset   g_i  M_{2h}(\co_F)  g_i^{-1},
\end{equation}
and   $H_0\cap H_3 \subset H_0\times H_3$ is embedded diagonally.
\end{definition}

\begin{remark}
The Haar measures on $H_0$, $H_3$,  and $H_0\cap H_3$ are normalized as in Remark \ref{rem:haar}.
\end{remark}

\begin{remark}\label{rem:orbital invariant}
The orbital integral depends on the choices of $g_0$ and $g_3$ satisfying \eqref{integrality shift 03}, but not in a significant way. Different choices  have the form $g_i' = z_i g_i k_i$ with
$z_i \in H_i$ and $k_i \in \GL_{2h}(\co_F)$,  and such a change multiplies the orbital integral by $ \eta(z_3) |z_0|^{-s}$.   
In particular, the value of the orbital integral at $s=0$ is well-defined up to $\pm 1$.
\end{remark}

\begin{proposition}\label{prop:no ramified orbitals}
If  $K/K_3$  is ramified,   then  $O_{(\Phi_0,\Phi_3)} (f; s, \eta ) =0$ for every $f$ as in \eqref{test function}.
\end{proposition}

\begin{proof}
For any $u \in H_3\cap  g_3 \GL_{2h}( \co_F)   g_3^{-1}$, making the change of variables $h_3\mapsto h_3 u$ in 
Definition \ref{def:orbital 03} shows that 
\begin{equation*}
O_{(\Phi_0,\Phi_3)} (f; s, \eta ) = \eta(u) \cdot O_{(\Phi_0,\Phi_3)} (f; s, \eta ).
\end{equation*}
Recalling that $g_3$ was chosen so that 
\begin{equation*}
 \Phi_3( \co_{K_3})  \subset   g_3 M_{2h}(\co_F)  g_3^{-1},
\end{equation*}
we may use $\Phi_3$ to view $g_3 \co_F^{2h}$ as a free $\co_{K_3}$-module of rank $h$.
Our assumption on $\eta$ guarantees the surjectivity of 
\begin{equation*}
   H_3\cap  g_3 \GL_{2h}( \co_F)   g_3^{-1}   =   \Aut_{\co_{K_3}} ( g_3 \co_F^{2h} ) \map{\det} \co_{K_3}^\times \map{\eta} \{ \pm 1\},
\end{equation*}
and so we may choose $u$ as above with $\eta(u)=-1$.
\end{proof}

Still assuming that $(\Phi_0,\Phi_3)$ is regular semisimple, we show that the orbital integral of Definition \ref{def:orbital 03} satisfies a functional equation in $s\mapsto -s$.  
This will be needed in the proof of Proposition \ref{prop:orbital vanish} below.

Let $g_0,g_3 \in \GL_{2h}(F)$ be as in \eqref{integrality shift 03}.  We can use both the embedding
$\Phi_i : K_i \to M_{2h}(F)$  and its Galois conjugate $\Phi_i \circ  \sigma_i$ to make the $\co_F$-lattice $g_i \co_F^{2h}$
into an $\co_{K_i}$-module.  These two $\co_{K_i}$-modules are isomorphic, and it follows that there is 
a  $u_i\in g_i \GL_{2h}(\mathcal O_F)g_i^{-1}$ such that 
\[
u_i\Phi_i(x) u_i^{-1} =\Phi_i(x^{\sigma_i})
\]
for all $x\in K_i$.  

 As in \eqref{z def}, fix an $F$-algebra generators $x_0 \in K_0$ and $x_3\in K_3$,  and define 
\begin{equation}\label{z03}
\mathbf{z}  =  \Phi_0(x_0)\Phi_3(x_3)-\Phi_3(x_3)\Phi_0(x_0)  \in M_{2h}(F).
\end{equation}
Recall from Proposition \ref{prop:st properties} that    $\mathbf{z} \Phi_i(x) = \Phi_i(x^{\sigma_i}) \mathbf{z}$ for all $x\in K_i$,  so that $\mathbf{z} u_i$ commutes with the image of $\Phi_i$.
From   Proposition \ref{prop:central algebra} we know that $\det(\mathbf{z}) \neq 0$, and so  $\mathbf{z} u_i\in H_i$.

\begin{proposition}\label{prop:functional equation}
For any  Hecke function  $f$ as in \eqref{test function}, we have the functional equation
\[
O_{(\Phi_0,\Phi_3)}(f;s,\eta)=| \mathbf{z} u_0 |^s  \cdot \eta( \mathbf{z} u_3 ) \cdot O_{(\Phi_0,\Phi_3)}( f ;-s,\eta).
\]
\end{proposition}

\begin{proof}
If we define  another  Hecke function $f^*$  by 
\[
f^*( g )=f(g_0^{-1}u_0^{-1} g_0  \cdot g \cdot  g_3^{-1}   u_3g_3),
\]
then
\begin{align*}
& O_{(\Phi_0,\Phi_3)}(f;s,\eta)  \\
& = \int_{  H_0\cap H_3 \backslash H_0\times H_3  } f( g_0^{-1} h_0^{-1} h_3 g_3 ) \cdot|h_0|^s\cdot\eta(h_3)\, dh_0\, dh_3    \\
&= 
\int_{  H_0\cap H_3 \backslash H_0\times H_3  } f^*( g_0^{-1} u_0 h_0^{-1} h_3 u_3^{-1} g_3 ) \cdot  
|h_0|^s  \cdot  \eta(h_3)  \, dh_0\, dh_3 \\
&=
| \mathbf{z} u_0 |^s\eta( \mathbf{z} u_3)
\int_{  H_0\cap H_3 \backslash H_0\times H_3  } 
f^*( g_0^{-1} \mathbf z^{-1}h_0^{-1} h_3\mathbf z g_3 ) \cdot  |h_0|^s\cdot\eta(h_3)\, dh_0\, dh_3,
\end{align*}
where the final equality is obtained by the substitution $h_i\mapsto h_i\mathbf zu_i$.
Making the further  substitution  $h_i \mapsto \mathbf{z} h_i\mathbf{z}^{-1}$, and using Remarks \ref{rem:3 character} and \ref{rem:0 character}, yields
\[
O_{(\Phi_0,\Phi_3)}(f;s,\eta) 
= 
 | \mathbf{z} u_0 |^s  \cdot \eta( \mathbf{z} u_3 )   \cdot O_{(\Phi_0,\Phi_3)}( f^*;-s,\eta).
\]
Now note that $g_0^{-1}u_0^{-1} g_0$ and  $g_3^{-1}u_3^{-1} g_3$ lie in  $\GL_{2h}(\co_F)$, and so  $f^*=f$.
\end{proof}


\subsection{Matching pairs}


Suppose  we are given $F$-algebra embeddings
\begin{align*}
\Phi_0 &: K_0 \to M_{2h}(F), &  \Phi_1&: K_1 \to B \\
\Phi_3 &: K_3 \to M_{2h}(F), &   \Phi_2&: K_2 \to B ,
\end{align*}
 where $B$ is a central simple $F$-algebra of dimension $4h^2$. 
 Assume that both $(\Phi_0,\Phi_3)$ and $(\Phi_1,\Phi_2)$ are  regular semisimple.

 \begin{definition}\label{def:matching}
The pairs $(\Phi_0,\Phi_3)$ and $(\Phi_1,\Phi_2)$   \emph{match} if
\begin{equation*}
\Inv_{(\Phi_0,\Phi_3)} = \Inv_{(\Phi_1,\Phi_2)}.
\end{equation*}
\end{definition}

Denote by  $\mathbf{w}_{12}, \mathbf{z}_{12} \in B$  the elements \eqref{w def} and \eqref{z def} corresponding to the pair $(\Phi_1,\Phi_2)$, to distinguish them from the similarly defined  elements $\mathbf{w}_{03},\mathbf{z}_{03} \in M_{2h}(F)$ corresponding to $(\Phi_0,\Phi_3)$. 
These depend on choices of $F$-algebra generators $x_i \in K_i$ for $i\in \{ 0,1,2,3\}$, which we now choose in a compatible way.
If $\mathrm{char}(F) \neq 2$, first choose 
$x_1$ and $x_2$ in such a way that 
$
x_i^{\sigma_i} = -x_i
$
 and then set
\[
x_0 = (1,-1) \in K_0 ,\qquad x_3 = x_1\otimes x_2 \in K_3 \subset K_1\otimes_F K_2.
\]
If $\mathrm{char}(F) = 2$, first choose  $x_1$ and $x_2$ so that
$
x_i^{\sigma_i} = x_i +1 
$
and then set
\[
x_0 = (0,1) \in K_0 ,\qquad x_3 = x_1\otimes  1 + 1 \otimes x_2 \in K_3 \subset K_1\otimes_F K_2 .
\]
We now have  $F$-subalgebras 
\[
F[\mathbf{w}_{12},\mathbf{z}_{12}] \subset B , \qquad F[\mathbf{w}_{03},\mathbf{z}_{03}] \subset M_{2h}(F).
\]
Our  particular choices of generators are  justified by the following result.

\begin{lemma}\label{lem:match iso}
If $(\Phi_0,\Phi_3)$ and $(\Phi_1,\Phi_2)$ match,  there is an isomorphism of $F$-algebras
\[
F[\mathbf{w}_{12},\mathbf{z}_{12}]  \iso F[\mathbf{w}_{03},\mathbf{z}_{03}]
\]
 sending $\mathbf{w}_{12} \mapsto \mathbf{w}_{03}$ and $\mathbf{z}_{12} \mapsto \mathbf{z}_{03}$.  
\end{lemma}

\begin{proof}
The proof of Lemma \ref{lem:good w} explains how to reconstruct the minimal polynomials of $\mathbf{w}_{12}\in B$ and $\mathbf{w}_{03} \in M_{2h}(F)$  from 
the  invariant polynomials $\Inv_{(\Phi_1,\Phi_2)}$ and $\Inv_{(\Phi_0,\Phi_3)}$, using the scalars $c_{12},d_{12},c_{03},d_{03} \in K_3$ characterized by the equalities
\[
\mathbf{w}_{12} = c_{12} + d_{12} \mathbf{s}_{12}, \qquad \mathbf{w}_{03} = c_{03}+d_{03} \mathbf{s}_{03}
\]
of Proposition \ref{prop:alt st}.
We have chosen $(x_1,x_2)$ and $(x_0,x_3)$ in such a way that $c_{12}=c_{03}$ and $d_{12}=d_{03}$, and so the matching of invariant polynomials implies the matching of minimal polynomials.   
Therefore there exists an isomorphism of (\'etale, by Proposition \ref{prop:central algebra})  $F$-algebras
\begin{equation}\label{w matching}
F[\mathbf{w}_{12}]  \iso F[\mathbf{w}_{03}]
\end{equation}
 sending $\mathbf{w} _{12}\mapsto \mathbf{w}_{03}$.   
 Using Proposition \ref{prop:wz properties}  one sees that $\mathbf{z}_{12}$ and $\mathbf{z}_{03}$ have the same  (quadratic) minimal polynomial over \eqref{w matching}, and the  lemma follows. 
 \end{proof}

\begin{proposition}\label{prop:orbital vanish}
Assume that $B$ is a division algebra.
If $(\Phi_0,\Phi_3)$ and $(\Phi_1,\Phi_2)$ match,  then  
\begin{equation*}
 O_{(\Phi_0,\Phi_3)} (f; 0,\eta) =   0 
\end{equation*}
for all $f$ as in \eqref{test function}.  
\end{proposition}

\begin{proof}
Identifying $K_1$ and $K_2$ with their images under $\Phi_i : K_i \to B$, and 
using Propositions \ref{prop:st properties} and \ref{prop:central algebra}, we find inside of $B$ the diagram of field extensions
\[
\xymatrix{
 {  K_1[\mathbf{w}_{12}]  }   \ar@{-}[dr]_2 &  {F[\mathbf{w}_{12},\mathbf{z}_{12}]}   \ar@{-}[d]^2& {K_2[\mathbf{w}_{12}] }  \ar@{-}[dl]^2 \\
  & { F[\mathbf{w}_{12}]} \\ 
  { K_1  }  \ar@{-}[uu]^h    \ar@{-}[dr]_2&  & {  K_2 }  \ar@{-}[uu]_h   \ar@{-}[dl]^2 \\
  & {  F } \ar@{-}[uu]_h
}
\]
of the indicated degrees.  
Let $\sigma_i$ be the nontrivial automorphism of $K_i[\mathbf{w}_{12}]$ fixing  $F[\mathbf{w}_{12}]$, and 
recall from Proposition \ref{prop:wz properties} that 
\[
\mathbf{z}_{12}^2 \in F[\mathbf{w}_{12}]
\]
 and $\mathbf{z}_{12} a = a^{\sigma_i} \mathbf{z}_{12}$ for all $a\in K_i[\mathbf{w}_{12}]$.

If $\mathbf{z}_{12}^2$ were a norm from $K_i[\mathbf{w}_{12}]$, say $\mathbf{z}_{12}^2 =a a^{\sigma_i}$,  then
\[(\mathbf{z}_{12}-a)(\mathbf{z}_{12}+a^{\sigma_i} ) =0\] would contradict $B$ being a division algebra.  This gives the final equality in
\begin{align*}
\eta_{ K_i/F}(  -  \det ( \mathbf{z}_{03} ) )  
& =
\eta_{ K_i/F}(  - \mathrm{Nm}_{ F[\mathbf{w}_{03},\mathbf{z}_{03} ]  /F} ( \mathbf{z}_{03}) )\\
& = \eta_{ K_i/F}(  - \mathrm{Nm}_{ F[ \mathbf{w}_{12}, \mathbf{z}_{12}] /F} ( \mathbf{z}_{12} ) )   \\
&=  \eta_{ K_i/F}( \mathrm{Nm}_{ F[\mathbf{w}_{12}] /F} ( \mathbf{z}_{12}^2 ) )    \\
& =   \eta_{K_i[\mathbf{w}_{12}] / F[\mathbf{w}_{12}]}(  \mathbf{z}_{12}^2)    \\
& =-1 ,  \nonumber
\end{align*}
where $\eta_{ \bullet }$ denotes the character associated to a quadratic extension by local class field theory, and we have used Lemma \ref{lem:match iso}.

To complete the proof, we may assume (after Proposition \ref{prop:no ramified orbitals}) that  $K/K_3$ is unramified.
This implies that at least one of   $K_1$ or $K_2$ is unramified over $F$.  Without loss of generality, let us suppose  $K_1$ is unramified.

If $u_3\in \GL_{2h}(F)$ is as in Proposition \ref{prop:functional equation}, so that $\det(u_3) \in \co_F^\times$, then Remark \ref{rem:3 character} provides the first equality in 
\[
\eta( \mathbf{z}_{03} u_3) =  \eta_{K_1/F}( \det ( \mathbf{z}_{03} u_3) ) =\eta_{K_1/F}( \det ( \mathbf{z}_{03} ) )   =  -1 .
\]
The claim now follows from the functional equation of Proposition \ref{prop:functional equation}.
\end{proof}


\subsection{The central value conjecture}


Assume that $K/K_3$ is unramified, and suppose we are given $F$-algebra embeddings
\begin{align*}
\Phi_0 &: K_0 \to M_{2h}(F), &  \Phi_1&: K_1 \to M_{2h}(F)  ,  \\
\Phi_3 &: K_3 \to M_{2h}(F), &   \Phi_2&: K_2 \to M_{2h}(F) .
\end{align*}

\begin{conjecture}[Biquadratic fundamental lemma]\label{conj:BFL}
If the pairs $(\Phi_0,\Phi_3)$ and $(\Phi_1,\Phi_2)$ are regular semisimple  and  matching (Definitions \ref{def:rss} and  \ref{def:matching}), then for any  $f$ as in \eqref{test function} we have
\begin{equation*}
 \pm O_{(\Phi_0,\Phi_3)} (f; 0,\eta) =   O_{(\Phi_1,\Phi_2)} (f) .
\end{equation*}
\end{conjecture}

\begin{remark}\label{rem:FLequivalence}
When  $K_1\iso K_2$, Proposition \ref{prop:split orbital comparison} implies that  this conjecture is equivalent to the Guo-Jacquet fundamental lemma, which was proved by Guo \cite{Guo} when $f = \mathbf{1}$ is the characteristic function of $\GL_2(\co_F)$.  
\end{remark}

There is one case in which the biquadratic fundamental lemma is trivial.

\begin{proposition}\label{prop:split FL}
If either one of $K_1$ or $K_2$ is isomorphic to $F\times F$, then 
 Conjecture \ref{conj:BFL} is true.
\end{proposition}

\begin{proof}
Assume for simplicity that  $K_1\iso F\times F$, the other case being entirely similar.
In this case there are isomorphisms
\[
K_1 \iso K_0,  \quad K_2 \iso K_3, \quad K\iso K_2 \times K_2
\]
(which we now fix), and the character $\eta : K_3^\times \to \{ \pm 1\}$ is trivial. 
Thus we need only prove the equality 
\begin{align*}\lefteqn{
 \int_{ H_0\cap H_3 \backslash H_0\times H_3 } f( g_0^{-1} h_0^{-1} h_3 g_3 )
 \, dh_0\, dh_3 } \\
& =
 \int_{  H_1\cap H_2 \backslash H_1\times H_2  } f( g_1^{-1} h_1^{-1} h_2 g_2 ) \, dh_1\, dh_2 
\end{align*}
for any Hecke function \eqref{test function}.

 Recalling Remark \ref{rem:orbital invariant}, both integrals are independent of the choices of  $g_0, g_1,g_2,g_3\in \GL_{2h}(F)$ satisfying
\begin{equation*}
 \Phi_i( \co_{K_i})  \subset   g_i  M_{2h}(\co_F)  g_i^{-1}.
\end{equation*}
More to the point, the integral on the left is unchanged if we replace the pair $(\Phi_0,\Phi_3)$ by a pair that is  $\GL_{2h}(F)$-conjugate to it.   Corollary \ref{cor:conjugacy} implies that  $(\Phi_0,\Phi_3)$ and  $(\Phi_1,\Phi_2)$ are conjugate, and the equality of integrals follows.
\end{proof}


\section{Intersections in Lubin-Tate space}


We continue to let $F$ be a nonarchimedean local field, and now assume that $K_1$ and $K_2$ are separable quadratic field extensions of $F$.

In the Lubin-Tate deformation space of a formal $\co_F$-module, one can construct cycles of formal $\co_{K_1}$-modules and $\co_{K_2}$-modules.  We  prove a formula  expressing the intersection multiplicities of such  cycles in terms of the invariant polynomials of \S \ref{ss:polynomial}; when $K_1= K_2$, this recovers the main result of \cite{Li}.
We then state a conjectural  arithmetic fundamental lemma, in the spirit of \cite{Zhang}, relating the intersection multiplicity to the central derivative of an orbital integral.


\subsection{Initial data}
\label{ss:cycle data}


Let $\breve{F}$ be the completion of the maximal unramified extension of $F$, and let $\kk$ be the residue field of $\co_{\breve{F}}$. 
Choose an extension of the reduction  $\co_F \to \kk$ to a ring homomorphism
\begin{equation}\label{structure map}
\co_K   \to \kk.
\end{equation}
In particular, $\kk$ is both an $\co_{K_1}$-algebra and an $\co_{K_2}$-algebra.

Fix a pair $(\mathscr{G},\mathscr{M})$ in which
\begin{itemize}
\item
$\mathscr{G}$ is a formal  $\co_F$-module  over $\kk$ of dimension $1$ and height $2h$,
\item
$\mathscr{M}$ is a free $\co_F$-module  of rank $2h$.  
 \end{itemize} 
  Fix also pairs $(\mathscr{H}_1,\mathscr{N}_1)$ and $(\mathscr{H}_2,\mathscr{N}_2)$ in which
 \begin{itemize}
 \item
$\mathscr{H}_i$ is a formal  $\co_{K_i}$-module  over $\kk$ of dimension $1$ and height $h$,
\item
 $\mathscr{N}_i$ is a free $\co_{K_i}$-module  of rank $h$.
 \end{itemize} 
 
 \begin{remark}\label{rem:strict}
 A one dimensional formal $\co_F$-module $G$ over an  $\co_F$-algebra $A$ is always assumed to be \emph{strict}, in the sense that the induced action 
 $\co_F \to \End_A(\mathrm{Lie}(G) ) \iso A$ agrees with the structure map.  
 A similar assumption is made for formal $\co_{K_i}$-modules, which is why we have fixed an $\co_K$-algebra structure \eqref{structure map} on $\kk$.
 \end{remark}

There are $\co_F$-linear isomorphisms  $\mathscr{G} \iso \mathscr{H}_i$ and $\mathscr{M}\iso \mathscr{N}_i$.
Rather than fixing such isomorphisms, we work in slightly more generality and fix 
 pairs $( \phi_1 , \psi_1)$ and $( \phi_2 , \psi_2)$   of invertible elements 
\begin{align}\label{cycle input}
\phi_1 \in \Hom_{\co_F}( \mathscr{H}_1,\mathscr{G}  ) [1/\pi],  & \qquad  \psi_1 \in \Hom_{\co_F}( \mathscr{N}_1,  \mathscr{M}  ) [1/\pi], \\
 \phi_2 \in \Hom_{\co_F}( \mathscr{H}_2,\mathscr{G}  ) [1/\pi],   &\qquad    \psi_2 \in \Hom_{\co_F}(  \mathscr{N}_2,\mathscr{M} ) [1/\pi] . \nonumber
\end{align}
Denote by $\mathrm{ht}(\phi_i)$ the height of $\phi_i$ as a quasi-isogeny of formal $\co_F$-modules, and define
\begin{equation*}
\mathrm{ht}(\psi_i) = \ord_F ( \det(  \eta_i \circ\psi_i))
\end{equation*}
for any $\co_F$-linear isomorphism $\eta_i :   \mathscr{M}  \iso \mathscr{N}_i$.

\begin{remark}\label{rem:simplePhiPsi}
Although for now we work in the generality described above, eventually we will  assume that 
\[
\phi_i : \mathscr{H}_i \iso \mathscr{G} \quad \mbox{and}\quad \psi_i : \mathscr{N}_i \iso \mathscr{M}
\] 
are $\co_F$-linear isomorphisms, so that $\mathrm{ht}(\phi_i) =0= \mathrm{ht}(\psi_i)$.
The reader will lose little by restricting to this special case throughout.
\end{remark}

Define central simple $F$-algebras
\begin{equation*}
B(\mathscr{G}) = \End_{\co_F}( \mathscr{G}) [1/\pi],  \qquad 
B(\mathscr{M}) = \End_{\co_F}( \mathscr{M}) [1/\pi] 
\end{equation*}
 of dimension $4h^2$.  The first is a division algebra.
 The second is isomorphic to $M_{2h}(F)$, but we do not fix such an isomorphism.

The data \eqref{cycle input}, together with the natural actions of  $\co_{K_i}$  on $\mathscr{H}_i$ and $\mathscr{N}_i$, determine $F$-algebra embeddings
\begin{align}\label{four embeddings}
\Phi_1 : K_1 \to B(\mathscr{G}) , & \qquad \Psi_1 : K_1 \to B( \mathscr{M} )   \\ 
\Phi_2 : K_2 \to B(\mathscr{G})  , & \qquad \Psi_2 : K_2 \to  B(  \mathscr{M} )   . \nonumber
\end{align}
The constructions of \eqref{invariant}  then  provide us with monic degree $h$ polynomials
\begin{equation*}
\Inv_{(\Phi_1,\Phi_2)} \in K_3[t],\qquad \Inv_{(\Psi_1,\Psi_2)} \in K_3[t].
\end{equation*}
We denote by 
\begin{equation}\label{resultant}
R(\Phi_1,\Phi_2,\Psi_1,\Psi_2) =  \mathrm{Res} \left( \Inv_{( \Phi_1,\Phi_2)} , \Inv_{(\Psi_1,\Psi_2)}  \right) \in K_3
\end{equation}
their resultant, and, when no confusion can arise,   abbreviate this to 
\begin{equation*}
R=R(\Phi_1,\Phi_2,\Psi_1,\Psi_2).
\end{equation*}
It  follows from  the functional equations (Proposition \ref{prop:functional})
satisfied by both invariant polynomials that 
$R ^{\sigma_3} = (-1)^h \cdot R.$  In particular $R^2\in F$, and  we abbreviate $|R| = \sqrt{ |R^2|}$.

\begin{proposition}\label{prop:good resultant}
If $(\Phi_1,\Phi_2)$ is regular semisimple, then $|R| \neq 0$.
\end{proposition}

\begin{proof}
First suppose that $K_1\not\iso K_2$, so that $K_3$ is a field.  If $|R|=0$ then $R=0$, and so the polynomials 
$\Inv_{( \Phi_1,\Phi_2)}$ and  $\Inv_{(\Psi_1,\Psi_2)}$ share a common root in an algebraic closure of $K_3$.

Consider the $F$-subalgebra
\[
F(\Phi_1,\Phi_2) \subset B(\mathscr{G})
\]
generated by  $\Phi_1(K_1) \cup \Phi_2(K_2)$.   As $B(\mathscr{G})$ is a division algebra,   Proposition \ref{prop:central algebra} implies that $F(\Phi_1,\Phi_2)$ is a quaternion division algebra over its center 
$L=F[\mathbf{w}]$, which is a degree $h$ field extension of $F$. In particular, the minimal polynomial of $\mathbf{w}$ over $F$ is irreducible.   As in the proof of Lemma \ref{lem:good w}, this minimal polynomial is related to the invariant polynomial $Q_\mathbf{s}=\Inv_{(\Phi_1,\Phi_2)}$ by a  change of variables, and so $\Inv_{(\Phi_1,\Phi_2)}$ is itself irreducible.

As $\Inv_{( \Phi_1,\Phi_2)}$ and  $\Inv_{(\Psi_1,\Psi_2)}$ have the same degree, share a common root, and the first is irreducible, they must be equal.  By Proposition \ref{prop:quaternion recovery}, the $F$-subalgebra
\[
F(\Psi_1,\Psi_2) \subset B(\mathscr{M})
\]
 is also quaternion division algebra over $L$, and so any representation of it has $F$-dimension a multiple of $4h$.  As $B(\mathscr{M}) \iso M_{2h}(F)$, we have arrived at a contradiction.

 Now suppose that $K_1 \iso K_2$, so that $K_3 \iso F\times F$.  This case is not really different. If we write
 $R=(R_1,R_2) \in F\times F$,  the relation  $R^{\sigma_3}=(-1)^h R$ noted above implies $R_1=(-1)^h R_2$.
 Thus if  $|R|= |R_1| \cdot |R_2|$ vanishes both $R_1=0$ and $R_2=0$.  
 
As  $K_3[T] \iso F[T] \times F[T]$, we may view
 $\Inv_{(\Phi_1,\Phi_2)}$ and $\Inv_{(\Psi_1,\Psi_2)}$  as pairs of polynomials with coefficients in $F$. 
  The argument above shows  that each component of $\Inv_{(\Phi_1,\Phi_2)}$  is irreducible, and the vanishing of $R_1$ and $R_2$ imply that those components agree with the components of  $\Inv_{(\Psi_1,\Psi_2)}$.  This again implies that $F(\Psi_1,\Psi_2)$ is a quaternion division algebra over a degree $h$ field extension of $F$, contradicting $F(\Psi_1,\Psi_2) \subset M_{2h}(F)$.
 \end{proof}


\subsection{Height calculations}
\label{ss:height}


Define   a formal $\co_F$-module 
\begin{equation*}
\mathbf{X} =  \Hom_{\co_F} ( \mathscr{M} ,   \mathscr{G} )
\end{equation*}
over $\kk$, noncanonically isomorphic to $\mathscr{G}^{ 2h}$.  
That is to say, $\mathbf{X}$ is the formal scheme over $\kk$ whose functor of points assigns to an Artinian $\kk$-algebra the $\co_F$-module
\[
\mathbf{X}(A) =  \Hom_{\co_F} ( \mathscr{M} ,   \mathscr{G}(A) ).
\]
Similarly,  for $i\in \{ 1,2\}$  define formal $\co_F$-modules  
\begin{equation*}
\mathbf{Y}_i  =  \Hom_{\co_{K_i}} (  \mathscr{N}_i ,  \mathscr{H}_i  ),  \qquad 
\overline{ \mathbf{Y}} _i  =  \Hom_{ \co_{K_i}}  (   \mathscr{N}_i ,  \overline{\mathscr{H}}_i  ),
\end{equation*}
where $\overline{\mathscr{H}}_i =  \mathscr{H}_i $ endowed with its conjugate $\co_{K_i}$-action.
  Each  is noncanonically isomorphic to $\mathscr{G}^h$.
There are natural morphisms 
\[
\xymatrix{
 { \mathbf{Y}_i }  \ar[dr]_{s_i}  &  &  { \overline{\mathbf{Y}}_i } \ar[dl]^{\overline{s}_i}   \\
 & \Hom_{\co_F} (  \mathscr{N}_i  ,  \mathscr{H}_i  ) ,
}
\]
and the composition  
\begin{equation}\label{full splitting comp}
\mathbf{Y}_i \times \overline{\mathbf{Y}}_i 
\map{s_i\times \overline{s}_i } 
\Hom_{\co_F} (   \mathscr{N}_i  ,  \mathscr{H}_i )
\map{  x \mapsto   \phi_i \circ  x \circ \psi_i^{-1} } \Hom_{\co_F} (   \mathscr{M}  ,\mathscr{G} )
= \mathbf{X}
 \end{equation}
 defines a quasi-isogeny 
 \begin{equation}\label{full splitting}
 \Delta_i   \in \Hom_{\co_F} (   \mathbf{Y}_i \times  \overline{\mathbf{Y}}_i ,  \mathbf{X}    )  [1/\pi] .
 \end{equation}

\begin{proposition}\label{prop:heights 1}
The quasi-isogeny  \eqref{full splitting} has  height
\[
\mathrm{ht}(\Delta_i) =  h^2 \cdot \ord_F(d_i)  + 2h \cdot \mathrm{ht}(\phi_i) - 2h \cdot \mathrm{ht}(\psi_i) ,
\]
 where  $d_i \in \co_F$ is a generator of  the discriminant of $K_i/F$.
 \end{proposition}

 \begin{proof}
By choosing an $\co_F$-basis $1, \eta \in \co_{K_i}$,   we identify the natural map 
   \begin{equation}\label{toy splitting 1}
   \Hom_{\co_{K_i}}( \co_{K_i} , \mathscr{H}_i) \times   \Hom_{ \co_{K_i}}( \co_{K_i} ,  \overline{\mathscr{H}}_i)
   \to \Hom_{\co_F} ( \co_{ K_i} , \mathscr{H}_i) 
 \end{equation}
 with the morphism of formal $\co_F$-modules
 \begin{equation}\label{toy splitting 2}
 \mathscr{H}_i \times \mathscr{H}_i \map{ (x,y) \mapsto (x+y, \eta x + \eta^{\sigma_i} y)  }   \mathscr{H}_i \times \mathscr{H}_i .
 \end{equation}
 The   different $\mathfrak{D}_i$  of $K_i/F$ is generated by $\eta - \eta^{\sigma_i}$, and so the kernel of  \eqref{toy splitting 2} is the image of 
 \[
 \mathscr{H}_i[\mathfrak{D}_i] \map{ x\mapsto (x,-x) } \mathscr{H}_i \times \mathscr{H}_i .
 \]
It follows that the isogenies \eqref{toy splitting 1} and \eqref{toy splitting 2}  have $\co_F$-height $h \cdot \ord_F(d_i)$.
As $\mathscr{N}_i$ is free of rank $h$ over $\co_{K_i}$, we deduce that that first arrow in \eqref{full splitting comp} is an isogeny of $\co_F$-height
$h^2 \cdot \ord_F(d_i)$.  The claim follows easily from this.
 \end{proof}

 The quasi-isogenies \eqref{full splitting} for $i\in \{1,2\}$ determine an element 
 \[
\Delta_2^{-1} \circ  \Delta_1  \in   \Hom_{\co_F} (  \mathbf{Y}_1 \times  \overline{\mathbf{Y}}_1  ,  \mathbf{Y}_2 \times  \overline{\mathbf{Y}}_2 ) [ 1/\pi],
 \]
which is  encoded by  four components.  
The two that interest us are
 \begin{equation} \label{matrix coefficients 1}
 \alpha_{12} \in  \Hom_{\co_F}( \mathbf{Y}_1 , \overline{\mathbf{Y}}_2) [1/\pi] ,\qquad  \beta_{12}\in  \Hom_{\co_F}( \overline{\mathbf{Y}}_1 , \mathbf{Y}_2) [1/\pi] . 
 \end{equation}
Similarly, 
  \[
\Delta_1^{-1} \circ  \Delta_2  \in  \Hom_{\co_F} (  \mathbf{Y}_2 \times  \overline{\mathbf{Y}}_2  ,  \mathbf{Y}_1 \times  \overline{\mathbf{Y}}_1 ) [ 1/\pi]
 \]
 is encoded by four components, and the two that  interest us are
  \begin{equation} \label{matrix coefficients 2}
 \alpha_{21} \in  \Hom_{\co_F}( \mathbf{Y}_2 , \overline{\mathbf{Y}}_1) [1/\pi] ,\qquad  \beta_{21}\in  \Hom_{\co_F}( \overline{\mathbf{Y}}_2 , \mathbf{Y}_1) [1/\pi] . 
 \end{equation}
 
 \begin{lemma}\label{lem:munu ht}
We have the equalities 
\[
\mathrm{ht}(\alpha_{12})  = \mathrm{ht}(\beta_{12}) \quad \mbox{and}\quad \mathrm{ht}(\alpha_{21})  = \mathrm{ht}(\beta_{21}).
\]
\end{lemma}

\begin{proof}
We first claim that there is a $\gamma \in B(\mathscr{G})^\times$ such that
 \[
\Phi_1( x^{\sigma_1} ) = \gamma \cdot  \Phi_1(x) \cdot  \gamma^{-1},\qquad
 \Phi_2( y^{\sigma_2} ) = \gamma  \cdot \Phi_2(y)  \cdot \gamma^{-1}.
\]
for all $x\in K_1$ and $y\in K_2$.  (We remark that if the pair $(\Phi_1,\Phi_2)$ is regular semisimple, the element 
 $\mathbf{z}\in B(\mathscr{G})$ defined by  \eqref{z def}  is a unit by Proposition \ref{prop:central algebra}, and   Proposition \ref{prop:wz properties} allows us to take $\gamma=\mathbf{z}$).
Abbreviate $B=B(\mathscr{G})$.
The embeddings $\Phi_1$ and  $\Phi_2$ of \eqref{four embeddings} determine two $\Z/2\Z$-gradings, exactly as in \eqref{grading}, 
\begin{equation*}
B=B_+^{\Phi_1} \oplus B_-^{\Phi_1} ,\qquad B=B_+^{\Phi_2} \oplus B_-^{\Phi_2}.
\end{equation*}
These  satisfy
\begin{equation*}
B_-^{\Phi_1} = ( B_+^{\Phi_1})^\perp ,\qquad B_-^{\Phi_2} = (B_+^{\Phi_2})^\perp,
\end{equation*}
where $\perp$ is orthogonal complement with respect to the nondegenerate bilinear form $(b_1 ,  b_2) \mapsto  \mathrm{Trd}(b_1b_2)$
determined by the reduced trace $\mathrm{Trd}:B \to F$.
If $B_-^{\Phi_1} \cap B_-^{\Phi_2}=0$ then $B_-^{\Phi_1} + B_-^{\Phi_2}=B$.
Applying $\perp$ to both sides of this last equality yields $B_+^{\Phi_1}\cap B_+^{\Phi_2}=0$, which contradicts  $1\in B_+^{\Phi_1}\cap B_+^{\Phi_2}$.
Any nonzero $\gamma \in B_-^{\Phi_1} \cap B_-^{\Phi_2}$ is contained in $B^\times$ (recall that $B$ is a division algebra) and satisfies the desired properties.

Using the quasi-isogenies of \eqref{cycle input}, we obtain quasi-isogenies
\begin{align*}
\gamma_1 = \phi_1^{-1} \circ \gamma \circ \phi_1  & \in \End_{\co_F}(  \mathscr{H}_1  )  [1/\pi]  \\
\gamma_2 = \phi_2^{-1} \circ \gamma \circ \phi_2  
& \in \End_{\co_F}( \mathscr{H}_2  )  [1/\pi] ,
\end{align*}
which do not commute with the actions of $\co_{K_1}$ and $\co_{K_2}$.
Instead, they define quasi-isogenies
\begin{align*}
\gamma_1  \in \Hom_{\co_{K_1}}( \mathscr{H}_1 , \overline{\mathscr{H}}_1 )  [1/\pi] ,  & \qquad
\overline{\gamma}_1   \in \Hom_{\co_{K_1}}(   \overline{\mathscr{H}}_1 , \mathscr{H}_1  )  [1/\pi] ,  \\
\gamma_2  \in \Hom_{\co_{K_2}}( \mathscr{H}_2 ,  \overline{\mathscr{H}}_2 )  [1/\pi] ,  &\qquad
\overline{\gamma}_2  \in \Hom_{\co_{K_2}}(   \overline{\mathscr{H}}_2  , \mathscr{H}_2  )  [1/\pi] ,
\end{align*}
which in turn  determine  quasi-isogenies
\begin{align*}
\Gamma_1  \in \Hom_{\co_F}( \mathbf{Y}_1 , \overline{\mathbf{Y}}_1)  [1/\pi] ,  & \qquad
\overline{\Gamma}_1   \in \Hom_{\co_F}(  \overline{\mathbf{Y}}_1 , \mathbf{Y}_1  )  [1/\pi] ,  \\
\Gamma_2 \in \Hom_{\co_F}( \mathbf{Y}_2 , \overline{\mathbf{Y}}_2 )  [1/\pi] ,  &\qquad
\overline{\Gamma}_2  \in \Hom_{\co_F}(   \overline{\mathbf{Y}}_2  , \mathbf{Y}_2  )  [1/\pi] ,
\end{align*}
all of the same height, and  making the diagram 
\begin{equation*}
\xymatrix{
{   \mathbf{Y}_1 \times \overline{\mathbf{Y}}_1 } \ar[d]^{\Delta_1}  \ar[rrrr]^{ (a,b) \mapsto ( \overline{\Gamma}_1(b),\Gamma_1(a))}&  & & & {   \mathbf{Y}_1 \times \overline{\mathbf{Y}}_1 } \ar[d]_{\Delta_1}  \\
{ \mathbf{X} }  \ar[rrrr]^\gamma &  & &  &  {   \mathbf{X}  }   \\
{   \mathbf{Y}_2 \times \overline{\mathbf{Y}}_2 } \ar[u]_{\Delta_2}   \ar[rrrr]_{ (a,b) \mapsto (\overline{\Gamma}_2(b),\Gamma_2(a))} & & &  & {   \mathbf{Y}_2 \times \overline{\mathbf{Y}}_2 }   \ar[u]^{\Delta_2}  
}
\end{equation*}
commute.
The commutativity implies that $\beta_{12} \circ \Gamma_1 = \overline{\Gamma}_2 \circ \alpha_{12}$, so  $\mathrm{ht}(\alpha_{12})=\mathrm{ht}(\beta_{12})$.   The equality $\mathrm{ht}(\alpha_{21})=\mathrm{ht}(\beta_{21})$ is proved similarly.
\end{proof}

 \begin{proposition}\label{prop:heights 2}
The heights of \eqref{matrix coefficients 1}  and \eqref{matrix coefficients 2}  are related to the resultant \eqref{resultant} by 
\[
  \mathrm{ht}(\alpha_{12})  +   \mathrm{ht}(\alpha_{21}) 
=  \ord_F ( R^2) =  \mathrm{ht}(\beta_{12})  +   \mathrm{ht}(\beta_{21}) .
\]
\end{proposition}

\begin{proof}
If we define idempotent elements
\begin{equation*}
e_1,f_1,e_2,f_2 \in \End_{\co_F}(\mathbf{X})[1/\pi] 
\end{equation*}
by the commutativity of the diagrams
\begin{equation*}
\xymatrix{
{   \mathbf{Y}_i \times \overline{\mathbf{Y}}_i} \ar[rr]^{  (a,b) \mapsto (  a,0) }  \ar[d]_{ \Delta_i} 
&   & {  \mathbf{Y}_i \times \overline{\mathbf{Y}}_i }  \ar[d]_{\Delta_i} 
&  {   \mathbf{Y}_i \times \overline{\mathbf{Y}}_i} \ar[rr]^{  (a,b) \mapsto (  0,b) }  \ar[d]_{ \Delta_i}  
&   & {  \mathbf{Y}_i \times \overline{\mathbf{Y}}_i }  \ar[d]_{\Delta_i}  \\
{  \mathbf{X} } \ar[rr]^{  e_i   }  
 &   &  \mathbf{X}  
&   {  \mathbf{X} } \ar[rr]^{  f_i   }  
&   &  \mathbf{X}    ,
 }
\end{equation*} 
then the composition
\[
 \mathbf{Y}_1 \times \overline{\mathbf{Y}}_1 \map{\Delta_1 } \mathbf{X} \map{ f_2e_1+e_2 f_1}  \mathbf{X} \map{ \Delta_2^{-1}}  \mathbf{Y}_2 \times \overline{\mathbf{Y}}_2
\]
is given by $(a,b) \mapsto ( \beta_{12}(b) , \alpha_{12}(a) )$.  It follows that 
\[
 \mathrm{ht}(\alpha_{12})  + \mathrm{ht}(\beta_{12})    =
\mathrm{ht}  ( f_2e_1+e_2f_1) +  \mathrm{ht}( \Delta_1) - \mathrm{ht}(  \Delta_2) .
 \]
The same equality holds with  the indices $1$ and $2$ switched everywhere.
Adding these  together and using Lemma \ref{lem:munu ht}  shows that 
\begin{equation}  \label{toCSA}
 2 \mathrm{ht}(\alpha_{12})  +  2 \mathrm{ht}(\alpha_{21})  
 =   \mathrm{ht} (f_2e_1+e_2f_1)   +   \mathrm{ht} (f_1e_2+e_1f_2)  . 
\end{equation}

Some elementary algebra shows that  $e_i$ and $f_i$ are the images of 
$(1,0)$ and $(0,1)$, respectively, under
\[
K_i \times K_i \iso K_i \otimes_F K_i  \map{ \Phi_i \otimes\Psi_i } B(\mathscr{G}) \otimes_F B(\mathscr{M})^{\mathrm{op}}
\iso  \End_{\co_F}(\mathbf{X})[1/\pi]
\]
where the first isomorphism is the inverse of 
\[
K_i \otimes_F K_i \map{ a \otimes b \mapsto (ab , a b^{\sigma_i} ) } K_i \times K_i,
\]
and $\mathrm{op}$ indicates the opposite algebra.

Extend \eqref{four embeddings} to $K_3$-algebra embeddings  
\begin{align*}
\Phi_i & : K \to C(\mathscr{G}) = B(\mathscr{G})  \otimes_F K_3 \\
  \Psi_i & : K \to C(\mathscr{G}) = B(\mathscr{G})  \otimes_F K_3
\end{align*}
 as in  \eqref{induced embeddings}, and  denote by 
$ \mathrm{Nrd} : C(\mathscr{G}) \otimes_{K_3} C (\mathscr{M})^\mathrm{op} \to K_3$
 the  reduced norm as a central simple $K_3$-algebra.  
If we fix any $K_3$-algebra generator $y\in K$, the images of these idempotents in
 $C(\mathscr{G}) \otimes_{K_3} C(\mathscr{M})^\mathrm{op}$ are given by the explicit formulas
\begin{align*}
e_i & =  [  \Phi_i( y-y^{\tau_3})  \otimes 1]^{-1} \cdot [   \Phi_i(y) \otimes 1 - 1 \otimes \Psi_i(y^{\tau_3})  ] \\
f_i  & =  [  \Phi_i( y-y^{\tau_3})  \otimes 1]^{-1} \cdot [   \Phi_i(y) \otimes 1 - 1 \otimes \Psi_i(y)  ] .
\end{align*}

In general, if $a \in C(\mathscr{G})$ and $b \in C(\mathscr{M})^\mathrm{op}$ have reduced  characteristic polynomials 
  $P_a, P_b \in K_3[t]$, then 
\begin{equation*}
\mathrm{Res}(P_a,P_b)  = \mathrm{Nrd} ( a \otimes 1-1\otimes b)   .
\end{equation*}
Indeed, after  extending scalars we may assume that  both $C (\mathscr{G})$ and  $C(\mathscr{M})^\mathrm{op}$ are matrix algebras.
If $a$ and $b$ are diagonalizable then the equality is obvious, and the general case follows by a
  Zariski density argument.

  The construction  \eqref{better s} attaches to $(\Phi_1,\Phi_2)$ and $(\Psi_1,\Psi_2)$ elements
\[
\mathbf{s}_\Phi \in C(\mathscr{G}) \quad\mbox{and}\quad \mathbf{s}_\Psi \in C(\mathscr{M})^\mathrm{op} ,
\]
whose   reduced characteristic polynomials $P_{\mathbf{s}_\Phi }, P_{\mathbf{s}_\Psi } \in K_3[T]$ 
are the  squares of the invariant polynomials of $(\Phi_1,\Phi_2)$ and $(\Psi_1,\Psi_2)$, respectively.
Thus
\begin{align*}
R^4  & = \mathrm{Res}( \Inv_{(\Phi_1,\Phi_2)}^2  ,  \Inv_{(\Psi_1,\Psi_2)}^2 )   \\
&  =  \mathrm{Res}(P_{\mathbf{s}_\Phi},P_{\mathbf{s}_\Psi } )    \\
& =  \mathrm{Nrd} ( \mathbf{s}_\Phi \otimes 1 -  1 \otimes \mathbf{s}_\Psi )  .
\end{align*}
It follows from Proposition \ref{prop:st properties} that 
\[
-( \mathbf{s}_\Phi \otimes 1-  1 \otimes \mathbf{s}_\Psi)
=
(1-\mathbf{s}_\Phi) \otimes 1 -  1 \otimes (1-\mathbf{s}_\Psi)
=
 \mathbf{s}^{\sigma_3}_\Phi \otimes 1 -  1 \otimes \mathbf{s}^{\sigma_3}_\Psi,
\]
and using this and Lemma \ref{lem:s conj} we find
\begin{equation}\label{first resultant}
R^4 
=
\mathrm{Nrd} \left[
 \frac{ (  \Phi_1(y) - \Phi_2(y) )^2  }{(y-y^{\tau_3})^2} \otimes 1 - 1 \otimes  \frac{ (  \Psi_1(y) - \Psi_2(y) )^2  }{(y-y^{\tau_3})^2} \right] .
\end{equation}

Define  $S,T \in C(\mathscr{G}) \otimes_{K_3} C(\mathscr{M})^\mathrm{op}$ by 
\begin{align*}
S & = [ \Phi_1(y) -\Phi_2(y) ] \otimes 1 + 1 \otimes [ \Psi_1(y) -\Psi_2(y) ] \\
& = [ \Phi_1(y-y^{\tau_3}) \otimes 1] \cdot e_1  - [ \Phi_2(y-y^{\tau_3}) \otimes 1] \cdot e_2 \\
T & =  [ \Phi_1(y) -\Phi_2(y) ] \otimes 1 - 1 \otimes [ \Psi_1(y) -\Psi_2(y) ] \\
& =  [ \Phi_1(y-y^{\tau_3}) \otimes 1] \cdot f_1  - [ \Phi_2(y-y^{\tau_3}) \otimes 1] \cdot f_2 .
\end{align*}
On the one hand, we have
\[
ST = [ \Phi_1(y) -\Phi_2(y) ]^2 \otimes 1 - 1 \otimes [ \Psi_1(y) -\Psi_2(y) ]^2,
\]
and comparison with \eqref{first resultant} shows that 
\begin{equation}\label{resultant 3}
R^4=  \frac{  \mathrm{Nrd}(ST)} { \mathrm{Nrd} ( \Phi_1( y-y^{\tau_3}) \otimes 1 ) \cdot  \mathrm{Nrd} ( \Phi_2( y-y^{\tau_3})  \otimes 1 ) }  .
\end{equation}
On the other hand, 
\begin{align*}
S T &= ST e_2 + ST f_2 \\
& = -[ \Phi_2(y-y^{\tau_3}) \otimes 1]   \cdot e_2  \cdot f_1  \cdot [ \Phi_1(y-y^{\tau_3}) \otimes 1 ]  \cdot e_2\\
& \quad  -[ \Phi_2(y-y^{\tau_3}) \otimes 1]  \cdot  f_2  \cdot  e_1  \cdot  [ \Phi_1(y-y^{\tau_3}) \otimes 1 ] \cdot  f_2 \\
& =   -[ \Phi_2(y-y^{\tau_3}) \otimes 1]  \cdot  ( e_2 f_1+f_2e_1)  \cdot  [ \Phi_1(y-y^{\tau_3}) \otimes 1]   \cdot ( e_1 f_2+ f_1 e_2) .
\end{align*}
Combining this with  \eqref{resultant 3} shows that  
\[
R^4=  \mathrm{Nrd} (f_2 e_1 + e_2 f_1) \cdot \mathrm{Nrd} (f_1 e_2 + e_1 f_2),
\]
and combining this with \eqref{toCSA} proves the proposition.
\end{proof}


\subsection{Cycles on a formal module}
\label{ss:linear cycles}


We now use the calculations of \S \ref{ss:height} to compute the intersection multiplicity of two  cycles  on the formal scheme
\begin{equation*}
\mathbf{X} \iso \Spf (  \kk [[  x_1,\ldots, x_{2h}  ]] ).
\end{equation*}

For $i\in\{1 , 2\}$, we use  the quasi-isogeny $\Delta_i$ of \eqref{full splitting} to define
\begin{equation}\label{quasi-cycle}
f_i  \in \Hom_{\co_F} (   \mathbf{Y}_i  , \mathbf{X}    )  [1/\pi] 
\end{equation}
as the composition
\begin{equation*}
\mathbf{Y}_i \map{  y\mapsto (y,0) }   \mathbf{Y}_i \times \overline{\mathbf{Y}}_i \map{ \Delta_i } \mathbf{X}.
\end{equation*} 
Choosing $k_i \in \Z$  large enough that $\pi^{k_i}$ clears the denominator in \eqref{quasi-cycle}, we obtain  finite morphisms
 \begin{equation}\label{formal cycles}
 \xymatrix{
 {  \mathbf{Y}_1 } \ar[dr]_{ \pi^{k_1}  f_1 } &  &  { \mathbf{Y}_2  } \ar[dl]^{  \pi^{k_2}  f_2  }   \\
 & {  \mathbf{X}  }.
 }
 \end{equation}
 
 \begin{remark}
If $\phi_i$ and $\psi_i$ are chosen as in Remark \ref{rem:simplePhiPsi} then one can take $k_i=0$, and the resulting maps $f_i :\mathbf{Y}_i \to \mathbf{X}$ are closed immersions.
 \end{remark}
 
 We wish to compute the intersection multiplicity  of $\mathbf{Y}_1$ and $\mathbf{Y}_2$, viewed as codimension $h$ cycles  on  $\mathbf{X}$.
 This is, by definition,  the dimension of the $\kk$-vector space of global sections of the $\co_\mathbf{X}$-module tensor product of
\begin{equation*}
 \mathcal{F}_1 = (  \pi^{k_1}  f_1 )_* \co_{\mathbf{Y}_1}   \quad  \mbox{and}\quad
  \mathcal{F}_2 = (  \pi^{k_2}  f_2)_* \co_{\mathbf{Y}_2}.
\end{equation*}
Of course the sheaf $\mathcal{F}_i$  depends on the choices of $k_i$, $\phi_i$, and $\psi_i$,  but we suppress this from the notation.
  The following theorem gives an explicit formula for this dimension  in terms of the resultant  \eqref{resultant}.

\begin{theorem}\label{thm:formal intersection}
Recall that  $q=|\pi| ^{-1}$ is the cardinality of the residue field of $\co_F$, and that $d_1,d_2\in \co_F$ generate the discriminants of $K_1/F$ and $K_2/F$.  The intersection multiplicity of the cycles \eqref{formal cycles} is 
\[
\dim_{\kk}  H^0\big( \mathbf{X},   \mathcal{F}_1    \otimes   \mathcal{F}_2 \big)   
 =     \frac{   q^{ 2 h^2 (k_1+k_2)  }    \cdot    q^{   h \cdot  [ \mathrm{ht}(\phi_1) +  \mathrm{ht}(\phi_2 )-\mathrm{ht}(\psi_1) -  \mathrm{ht}(\psi_2 )  ]  }} {   |R |   \cdot   | d_1d_2|^{ h^2  / 2 }     }   .
\]
In particular,  the left hand side is finite if and only if $|R| \neq 0$.
\end{theorem}

\begin{proof}
The proof will follow easily from the height calculations of \S \ref{ss:height} and the following  lemma, which shows that  $ f_i$ admits a particularly nice factorization
\begin{equation*}
\mathbf{Y}_i \map{A_i  } \mathbf{Z}_i \hookrightarrow \mathbf{X}
\end{equation*}
as a quasi-isogeny of formal $\co_F$-modules followed by a  closed immersion.

\begin{lemma}
There is a (non-unique) decomposition
$\mathbf{X} = \mathbf{Z}_i \times \overline{\mathbf{Z}}_i$ of formal $\co_F$-modules such that the quasi-isogeny $\Delta_i $ has the form
 \begin{equation*}
   \begin{pmatrix}  
   A_i & B_i \\ 0 & D_i 
   \end{pmatrix}  :
   \mathbf{Y}_i\times \overline{\mathbf{Y}}_i \to      \mathbf{Z}_i \times \overline{\mathbf{Z}}_i   
  \end{equation*}
for some 
 \begin{align*}
A_i & \in \Hom_{\co_F}( \mathbf{Y}_i , \mathbf{Z}_i  ) [1/\pi] \\
B_i & \in \Hom_{\co_F}(   \overline{\mathbf{Y}}_i  ,   \mathbf{Z}_i   ) [1/\pi] \\
D_i & \in \Hom_{\co_F}(  \overline{\mathbf{Y}}_i   , \overline{\mathbf{Z}}_i   ) [1/\pi] .
\end{align*}
\end{lemma}

\begin{proof}
 Let $\overline{\mathscr{N}}_i = \mathscr{N}_i$, but endowed with its conjugate $\co_{K_i}$ action.
 As $\End_{\co_F}(\mathscr{G})$ is the unique maximal order in $B(\mathscr{G})$, the map $\Phi_i$ of \eqref{four embeddings} restricts to 
$
\Phi_i : \co_{K_i} \to \End_{\co_F}(\mathscr{G}).
$
Let $\mathscr{G}_i = \mathscr{G}$ endowed with this action of $\co_{K_i}$.
Tracing through the definitions, we have canonical identifications of quasi-isogenies
 \begin{equation*}
 \xymatrix{
  {\mathbf{Y}_i \times \overline{\mathbf{Y}}_i } \ar@{=}[rr]   \ar[d]_{ \Delta_i} 
  & &
   {  \Hom_{\co_{K_i}} (     \mathscr{N}_i \times \overline{\mathscr{N}}_i  ,   \mathscr{H}_i   ) } 
   \ar[d]^{  x\mapsto  \phi_i \circ x \circ \gamma_i  }  \\ 
 { \mathbf{X} } \ar@{=}[rr]  
 & &
    {  \Hom_{\co_{K_i} } (   \co_{K_i} \otimes_{\co_F} \mathscr{M} , \mathscr{G}_i  )   , }    
 }
 \end{equation*}
where $\gamma_i$ is the ${K_i}$-linear composition 
\begin{equation*}
{K_i} \otimes_{F} \mathscr{M}[1/\pi] \map{ \mathrm{id} \otimes \psi_i^{-1} } {K_i}\otimes_{F} \mathscr{N}_i [1/\pi]
\map{  x\otimes n \mapsto ( xn , x^{\sigma_i}n) } \mathscr{N}_i[1/\pi] \times \overline{\mathscr{N}}_i [1/\pi].
\end{equation*}
Using the Iwasawa decomposition in $\GL_{2h}(K_i)$, we may find a decomposition of $\co_{K_i}$-modules
\[
\co_{K_i} \otimes_{\co_F} \mathscr{M}  \iso \mathscr{P} \times \mathscr{Q}
\]
such that $\gamma_i( \mathscr{Q}[1/\pi]) \subset \overline{\mathscr{N}}_i[1/\pi]$.
This induces a decomposition of the lower right corner of the diagram, and the induced decomposition of the lower left corner has the desired properties.
\end{proof}

The isogeny $ \pi^{k_i}A_i  : \mathbf{Y}_i \to \mathbf{Z}_i$ is a finite flat  morphism of  degree
\begin{equation*}
e_i \define q^{  2h^2 k_i   +  \mathrm{ht}(A_i)  }.
\end{equation*}
Thus  $(  \pi^{k_i}  f_i)_* \co_{\mathbf{Y}_i} \iso \co_{\mathbf{Z}_i}^{e_i}$ as $\co_{\mathbf{X}}$-modules, and 
\begin{equation*}
 \mathcal{F}_1    \otimes  \mathcal{F}_2
   \iso (  \co_{\mathbf{Z}_1} \otimes \co_{\mathbf{Z}_2}  )^{ e_1e_2 }.
\end{equation*}
The tensor product on the right  is the structure sheaf of 
\begin{equation*}
\mathbf{Z}_1 \times_{ \mathbf{X}} \mathbf{Z}_2 
= \mathrm{ker}(  \mathbf{Z}_1 \to \mathbf{X} \to \overline{\mathbf{Z}}_2),
\end{equation*}
which, by the definition of height,  is the formal spectrum of  a   $\kk$-algebra of dimension $q^{ \mathrm{ht}( \mathbf{Z}_1\to \mathbf{X} \to \overline{\mathbf{Z}}_2 )}$.  Thus
\begin{eqnarray*}
\dim_{\kk}  H^0\big( \mathbf{X},   \mathcal{F}_1   \otimes   \mathcal{F}_2 \big)   
& = &
e_1 e_2 \cdot \dim_{\kk}  H^0( \mathbf{X},  \co_{\mathbf{Z}_1} \otimes \co_{\mathbf{Z}_2}    )  \\
& = & q^{  2h^2 (k_1+k_2 )  }\cdot q^{     \mathrm{ht}(A_1)  +  \mathrm{ht}(A_2)  }
  \cdot q^{ \mathrm{ht}( \mathbf{Z}_1\to \mathbf{X} \to \overline{\mathbf{Z}}_2 )}    .
\end{eqnarray*}

The composition
\begin{equation*}
 \mathbf{Y}_1 \map{A_1}\mathbf{Z}_1\to  \mathbf{X} \to \overline{\mathbf{Z}}_2  \map{D_2^{-1}} \overline{\mathbf{Y}}_2
\end{equation*}
is precisely the map $\alpha_{12}$ of \eqref{matrix coefficients 1}, and so 
\begin{align*}
\mathrm{ht}( \mathbf{Z}_1\to  \mathbf{X} \to \overline{\mathbf{Z}}_2) 
& =   \mathrm{ht}(\alpha_{12}) -   \mathrm{ht}(A_1) + \mathrm{ht}(D_2)   \\
&= \mathrm{ht}(\alpha_{12}) -   \mathrm{ht}(A_1) - \mathrm{ht}(A_2) + \mathrm{ht}(\Delta_2).
\end{align*}
This leaves us with
\begin{equation}\label{half degree 1}
\dim_{\kk}  H^0\big( \mathbf{X},   \mathcal{F}_1   \otimes   \mathcal{F}_2 \big)  
= q^{  2h^2 (k_1+k_2)    }
q^{   \mathrm{ht}(\alpha_{12})  + \mathrm{ht}(\Delta_2)  }.
\end{equation}
As  the same reasoning holds with the indices $1$ and $2$  reversed throughout, 
\begin{equation}\label{half degree 2}
\dim_{\kk}  H^0\big( \mathbf{X},   \mathcal{F}_1   \otimes   \mathcal{F}_2 \big)  
= q^{  2h^2 (k_1+k_2)    }
q^{   \mathrm{ht}(\alpha_{21})  + \mathrm{ht}(\Delta_1)  }.
\end{equation}

Multiplying \eqref{half degree 1} and \eqref{half degree 2} together and using Proposition \ref{prop:heights 2} 
yields 
\begin{align*}
\dim_{\kk}  H^0\big( \mathbf{X},   \mathcal{F}_1   \otimes   \mathcal{F}_2 \big)  
&  = q^{  2h^2 (k_1+k_2)    }
q^{  \frac{ \mathrm{ht}(\alpha_{12}) + \mathrm{ht}( \alpha_{21}) + \mathrm{ht}(\Delta_1) + \mathrm{ht}(\Delta_2) }{2} } \\
& = 
q^{  2h^2 (k_1+k_2)    }
q^{  \frac{ \ord_F(R^2)  + \mathrm{ht}(\Delta_1) +  \mathrm{ht}(\Delta_2) }{2} } \\
& = 
 | R |^{-1}   \cdot q^{  2h^2 (k_1+k_2)    }
q^{  \frac{   \mathrm{ht}(\Delta_1) + \mathrm{ht}(\Delta_2) }{2} }.
\end{align*}
Theorem \ref{thm:formal intersection} now follows from the formulas for $ \mathrm{ht}(\Delta_1)$ and  $\mathrm{ht}(\Delta_2)$ found in Proposition \ref{prop:heights 1}. 
 \end{proof}


\subsection{Cycles on the Lubin-Tate tower}
\label{ss:LT cycles}


Assume now that the elements \eqref{cycle input} are chosen as in Remark \ref{rem:simplePhiPsi}.  In other words, for $i\in \{1,2\}$ we fix $\co_F$-linear isomorphisms
\[
\phi_i : \mathscr{H}_i \iso \mathscr{G} \quad \mbox{and}\quad \psi_i : \mathscr{N}_i \iso \mathscr{M}.
\]

Let $\mathrm{Nilp}( \co_{\breve{F}} )$ be the category of $\co_{\breve{F}} $-schemes on which the uniformizer $\pi\in \co_F$ is locally nilpotent.  For any $S\in \mathrm{Nilp}(\co_{\breve{F}} )$ we abbreviate
  \begin{equation*}
\bar{S}  = S \times_{ \Spec(\co_{\breve{F}} ) } \Spec(\kk).
\end{equation*}

 Associated to the pair $( \mathscr{G} , \mathscr{M})$ and an integer $m\ge 0$ one has the Lubin-Tate deformation space 
\begin{equation*}
X(\pi^m)   \to \Spf( \co_{\breve{F}}  )
\end{equation*}
 classifying triples  $(G,\varrho, t_m)$ over $S \in \mathrm{Nilp}(\co_{\breve{F}} )$ consisting of
   \begin{itemize}
 \item
a formal $\co_F$-module  $G$  over $S$,
 \item
an $\co_F$-linear quasi-isogeny   $\varrho : \mathscr{G}_{\bar{S}} \to G_{  \bar{S}}$ of height $0$,
  \item
an $\co_F$-linear  Drinfeld level structure   $t_m :  \pi^{-m} \mathscr{M}/\mathscr{M}  \to G [\pi^m]$.
  \end{itemize}

For each $i\in \{1,2\}$ the above isomorphisms $\phi_i$ and $\psi_i$  determine a closed formal subscheme 
\begin{equation}\label{CM with level}
f_i(m) : Y_i(\pi^m) \hookrightarrow X(\pi^m),
\end{equation}
defined as the locus of points $(G,\varrho, t_m)$ for which 
there exists a (necessarily unique) action  $\co_{K_i} \to \End_{\co_F}(G)$ making both 
\[
    \mathscr{H}_{i,\overline{S}} \map{\phi_i} \mathscr{G}_{\overline{S}}
 \map{\varrho}     G_{\overline{S}}  
\]
and
\[
\pi^{-m} \mathscr{N}_i / \mathscr{N}_i  \map{\psi_i} \pi^{-m} \mathscr{M} / \mathscr{M} 
  \map{t_m} G[\pi^m]
\]
  $\co_{K_i}$-linear.

We think of $Y_1(\pi^m)$ and $Y_2(\pi^m)$ as cycles on $X(\pi^m)$.
 Their intersection multiplicity is, by definition, the  length of the $\co_{\breve{F}}$-module of global sections of the tensor product of coherent $\co_{X(\pi^m)}$-modules
  \begin{equation*}
\mathcal{F}_1( \pi^m)  = f_1(m)_*\co_{ Y_1(\pi^m) } \quad\mbox{and}\quad
\mathcal{F}_2( \pi^m)  =   f_2(m)_*\co_{ Y_2(\pi^m) }  .  
 \end{equation*}

\begin{proposition}\label{prop:high level}
If $|R|\neq 0$, then 
\[
\mathrm{len}_{\co_{\breve{F}} }\, H^0 \big(  X(\pi^m) ,  \mathcal{F}_1( \pi^m)  \otimes  \mathcal{F}_2( \pi^m) \big)  
 =     |R |^{-1}   \cdot    | d_1d_2|^{ - h^2  /2  }  
\]
for all $m\gg 0$.   
\end{proposition}

\begin{proof}
Because we have now chosen the data \eqref{cycle input} as in Remark \ref{rem:simplePhiPsi}, we may take $k_1=k_2=0$ throughout \S \ref{ss:linear cycles}.
Theorem \ref{thm:formal intersection} implies that the $\kk$-vector space $H^0( \mathbf{X},  \mathcal{F}_1 \otimes \mathcal{F}_2 )$ has finite dimension, and so \cite[Theorem 4.1]{Li} implies that 
there is an isomorphism of $\co_{\breve{F}} $-modules
\begin{equation*}
H^0 \big(  X (\pi^m) ,  \mathcal{F}_1( \pi^m)  \otimes \mathcal{F}_2( \pi^m) \big) \iso H^0( \mathbf{X},  \mathcal{F}_1 \otimes \mathcal{F}_2 )
\end{equation*}
for all $m\gg 0$.  Now use  the equality  of Theorem \ref{thm:formal intersection}.
\end{proof}

For any $g\in B (\mathscr{M})^\times$, we can replace the embedding $\Psi_2 : K_2 \to B(\mathscr{M})$ with its conjugate $g \Psi_2 g^{-1}$ in \eqref{resultant}  to define
\begin{equation}\label{moving resultant}
R (g) =R(\Phi_1,\Phi_2,\Psi_1, g \Psi_2  g^{-1} ) \in K_3.
\end{equation}
As in \S \ref{ss:cycle data}, $R(g)^2 \in F$ and we abbreviate 
\[
|R(g)| = \sqrt{ |R(g)^2| }.
\]
If  $(\Phi_1,\Phi_2)$ is regular semisimple,  then $|R(g)|\neq 0$  by Proposition \ref{prop:good resultant}.

\begin{theorem}\label{thm:final intersection}
Suppose  $m\ge 0$, and set 
\begin{equation*}
U(\pi^m)= \mathrm{ker} \big( \Aut_{\co_F} (\mathscr{M}) \to  \Aut_{\co_F} (\mathscr{M}/\pi^m \mathscr{M}) \big)
\subset B(\mathscr{M})^\times .
\end{equation*}
If  the pair $(\Phi_1,\Phi_2)$ is regular semisimple,  then 
\begin{align*}
\mathrm{len}_{\co_{\breve{F}}} \, H^0 \big( X (\pi^m) ,  \mathcal{F}_1(  \pi^m)  \otimes \mathcal{F}_2( \pi^m) \big)  
=
   \frac{ c(m) \cdot   | d_1d_2| ^{ -h^2/2}    }{ \mathrm{vol}(  U(\pi^m))}
 \int_{ U(\pi^m) }  \frac{ dg} {  | R(g) |  }  .
\end{align*}
In particular, the left hand side is finite.
Here we have defined  $c(m)=1$ if $m>0$, and
\begin{equation*}
c(0) = \frac{ \# \Aut_{\co_F}(\mathscr{M}/\pi \mathscr{M}) }{ \# \Aut_{\co_{K_1}}(\mathscr{N}_1/\pi \mathscr{N}_1)  \cdot \# \Aut_{\co_{K_2}}(\mathscr{N}_2/\pi \mathscr{N}_2 )  }.
\end{equation*}
\end{theorem}

\begin{proof}
This follows from Proposition \ref{prop:high level}, exactly as in  the proof of  \cite[Proposition 6.6]{Li}.
\end{proof}


\subsection{The central derivative conjecture}
\label{ss:AFL}


We change the setup slightly from \S \ref{ss:cycle data}.   Fix $\co_F$-linear isomorphisms
\[
\phi_1 : \mathscr{H}_1 \iso \mathscr{G}  ,\qquad 
\phi_2: \mathscr{H}_2 \iso \mathscr{G}
 \]
such that the induced $F$-algebra embeddings 
\[
\Phi_1 : K_1 \to B(\mathscr{G}) ,\qquad \Phi_2 : K_2 \to B(\mathscr{G}) 
\]
form a regular semisimple pair $(\Phi_1,\Phi_2)$.  

Suppose we are also given $F$-algebra embeddings
\[
\Phi_0 : K_0 \to M_{2h}(F) ,\qquad \Phi_3 : K_3 \to M_{2h}(F)
\]
such  that the pair   $(\Phi_0,\Phi_3)$ is  regular semisimple and matches   $(\Phi_1,\Phi_2)$ in the sense  of Definition \ref{def:matching}.
As $B(\mathscr{G})$ is a division algebra, Proposition \ref{prop:orbital vanish} implies that 
\[
O_{(\Phi_0,\Phi_3)} (f; 0 ,\eta) =0 
\]
 for all $f$ as in \eqref{test function}.   In particular, this holds when $f$ is  the characteristic function of $\GL_{2h}(\co_F)$, which we denote by 
 $
 \mathbf{1} : \GL_{2h}(F) \to \C .
 $
 
 We now consider a modified version of the constructions of \S \ref{ss:LT cycles}, in which no Drinfeld level structure is added, but the quasi-isogeny $\varrho$ in the moduli problem  is  allowed to be of arbitrary height.

 Let  $X^\bullet$  be the formal scheme over $\co_{\breve{F}} $ classifying pairs $(G,\varrho)$ consisting of a formal $\co_F$-module  $G$  over  $S \in \mathrm{Nilp}(\co_{\breve{F}} )$  and  an $\co_F$-linear quasi-isogeny   $\varrho : \mathscr{G}_{\bar{S}} \to G_{  \bar{S}}$.  For every $\ell \in \Z$ the open and closed formal subscheme
 \[
 X^{(\ell)} \subset X^\bullet
 \]
 defined by $\mathrm{ht}(\varrho)=\ell$ is isomorphic to the formal spectrum of a power series ring in $2h-1$ variables over $\co_{\breve{F}}$.

For  $i\in \{1,2\}$, let $\breve{K}_i$ be the  completion of the maximal unramified extension of  $K_i$.  Consider the formal $\co_{\breve{K}_i}$-scheme $Y_i^\bullet$ classifying pairs $(H,\varrho)$ consisting of a formal $\co_{K_i}$-module  $H$  over  $S \in \mathrm{Nilp}(\co_{\breve{K}_i} )$  and  an $\co_{K_i}$-linear quasi-isogeny   $\varrho : \mathscr{H}_{i, \bar{S}} \to H_{  \bar{S}}$. 
Denote by 
\begin{equation}\label{CM cycle with components}
f_i : Y_i^\bullet \to X^\bullet
\end{equation}
the morphism of $\co_{\breve{F}}$-schemes
 sending $(H,\varrho)$ to the formal $\co_F$-module  $G=H$ endowed with the quasi-isogeny
\[
 \mathscr{G}_{\bar{S}} \map{\phi_1^{-1}}  \mathscr{H}_{i,\bar{S}} \map{\varrho} H_{  \bar{S}}  = G_{\bar{S}} .
\]
The restriction of $Y_i^\bullet$ to the connected component $X^{(\ell)}$ is denoted
\[
Y_i^{(\ell)} = Y_i^\bullet \times_{X^\bullet} X^{(\ell)}.
\]
Note that some $Y_i^{(\ell)}$ may be empty.

\begin{proposition}
The morphism \eqref{CM cycle with components} is a  closed immersion.
Moreover, the pair   $Y_i^{(0)} \subset  X^{(0)}$  is canonically identified with the pair
$Y_i(\pi^m) \subset X(\pi^m)$ from the previous subsection determined by $m=0$.
\end{proposition}

\begin{proof}
Similar to \eqref{CM with level}, we may define a closed formal subscheme 
\[
Y^\bullet_{i , \mathrm{naive}} \subset X^\bullet
\]
  as the locus of   points $(G,\varrho)$ in $X^\bullet$ for which 
there exists a (necessarily unique) action  $\co_{K_i} \to \End_{\co_F}(G)$ making  
\[
    \mathscr{H}_{i,\overline{S}} \map{\phi_i} \mathscr{G}_{\overline{S}}
 \map{\varrho}     G_{\overline{S}}  
\]
  $\co_{K_i}$-linear.   Essentially by definition, the morphism \eqref{CM cycle with components} factors through this closed formal subscheme.

  The ring $\co_{K_i}$ acts on the Lie algebra of the universal object over $Y^\bullet_{i , \mathrm{naive}}$.
  As this Lie algebra is a line bundle on $Y^\bullet_{i , \mathrm{naive}}$, this action must be   through some $\co_F$-algebra morphism
  \[
  \co_{K_i} \to \co_{ Y^\bullet_{i , \mathrm{naive}} } .
  \]
This morphism endows $Y^\bullet_{i , \mathrm{naive}}$ with the structure of a formal scheme not just over $\co_{\breve{F}}$, but over $\co_{K_i} \otimes_{\co_F} \co_{\breve{F}}$.
One can check (this is essentially  the strictness condition implicit in the definition of $Y_i^\bullet$, as in Remark \ref{rem:strict}) that the diagram 
\[
\xymatrix{
{ Y_i^\bullet  }  \ar[r]^{f_i} \ar[d]  & {   Y^\bullet_{i , \mathrm{naive}}  } \ar[d]   \\
{  \mathrm{Spf}(\co_{\breve{K}_i} )  } \ar[r]  & {  \mathrm{Spf}( \co_{K_i} \otimes_{\co_F} \co_{\breve{F}}   )    } 
}
\]
is cartesian, where the bottom horizontal arrow is induced by the map $\co_{K_i} \otimes_{\co_F} \co_{\breve{F}}   \to  \co_{\breve{K}_i} $ sending $\alpha \otimes x\mapsto \alpha x$.
As this bottom arrow is a closed immersion, so is the top one,  and hence so is \eqref{CM cycle with components}.

If  $K_i/F$ is ramified, the bottom horizontal arrow in the diagram is an isomorphism, and hence so is the top horizontal arrow.
On the other hand, if $K_i/F$ is unramified then 
\[
\mathrm{Spf}(\co_{\breve{K}_i} )  \to   \mathrm{Spf}( \co_{K_i} \otimes_{\co_F} \co_{\breve{F}}   ) \iso 
\mathrm{Spf}(\co_{\breve{F}} ) \sqcup \mathrm{Spf}(\co_{\breve{F}} )
\]
is an isomorphism onto  one of the two components, and the top horizontal arrow is an open and closed immersion.

When $m=0$ the equality  $X(\pi^m)=X^{(0)}$ holds simply by definition, and similarly for
\[
Y_i(\pi^m) = Y_{i,\mathrm{naive}}^\bullet \otimes_{ X^\bullet } X^{(0)}. 
\]
The previous paragraph shows that
\[
Y_i^{(0)} \subset Y_{i,\mathrm{naive}}^\bullet \otimes_{ X^\bullet } X^{(0)}
\]
as a union of connected components.   As the underlying reduced scheme of $X^{(0)}$ is a point, \emph{any} closed formal subscheme of it is connected.  In particular the right hand side of the above inclusion is connected.
The formal scheme  $Y_i^{(0)}$ is nonempty, as it contains the $\kk$-valued point $(\mathscr{H}_i,\mathrm{id})$, and so the  inclusion is an equality.
\end{proof}

We would like to relate the derivative  of  $O_{(\Phi_0,\Phi_3)} (f; s ,\eta) $  at $s=0$   to the intersection multiplicity of the cycles $Y_1^\bullet$ and $Y_2^\bullet$ on $X^\bullet$, but this is not defined because $Y_1^\bullet \times_{X^\bullet} Y_2^\bullet$ is an infinite disjoint union of Artinian schemes.  
We must carefully take connected components into account.

Let $L \subset B(\mathscr{G})$ be the the centralizer of $F$-subalgebra generated by $\Phi_1(K_1) \cup \Phi_2(K_2)$.
Recall from Proposition  \ref{prop:central algebra} that it is an  \'etale $F$-algebra of dimension $h$.
The group $B(\mathscr{G})^\times$ acts on $X^\bullet$ by changing the quasi-isogeny $\varrho$ in the moduli problem, and the action of the subgroup $L^\times$ preserves each of the closed subschemes $Y_i^\bullet \subset X^\bullet$.

\begin{conjecture}[Arithmetic  biquadratic  fundamental lemma]\label{conj:ABFL} 
If $K/K_3$ is unramified then 
\begin{align*}
&   \frac{ \pm 1}{ \log(q)  }\frac{d}{ds} O_{(\Phi_0,\Phi_3)} ( \mathbf{1} ; s ,\eta)  \big|_{s=0}     \\
& =
\sum_{  X^{(\ell)} \in L^\times \backslash \pi_0( X^\bullet)  } 
\mathrm{len}_{\co_{\breve{F}} }\, H^0 \big( X^{(\ell)}   ,  f_{ 1*} \co_{ Y^{(\ell)}_1 } \otimes_{\co_{X^{(\ell)}}} f_{ 2*} \co_{ Y^{(\ell)}_1 }  \big)  ,
\end{align*}
where the sum is over a set of representatives $X^{(\ell)}$ for the $L^\times$ orbits of connected components of $X^\bullet$.
We remark that there are only finitely many such orbits, as  there are  $2h$ orbits under the action of the subgroup $F^\times$. 
 \end{conjecture}

\begin{remark}\label{rem:AFLequivalence}
In the special case where $K_1\iso K_2$, Conjecture \ref{conj:ABFL} is equivalent to the linear arithmetic 
fundamental lemma of \cite[Conjecture 1]{Li} (which should be corrected to incorporate all connected components of $X^\bullet$, as we have done above).   This equivalence uses Proposition \ref{prop:split orbital comparison} below, as the orbital integrals $O_{(\Phi_0,\Phi_3)}( f; s,\eta)$ defined here do not quite agree with the orbital integrals of \cite{Li}.
\end{remark}


\section{Calculations when $h=1$}
\label{one case}


Assume that $h=1$.  We will prove  Conjecture  \ref{conj:BFL}  when $f=\mathbf{1}$ is the characteristic function of $\GL_2(\co_F)$,  and also prove Conjecture \ref{conj:ABFL}.
Throughout, we assume that $F$ is a local field, and that $K_1$ and $K_2$ are quadratic \'etale extensions such that 
$K/K_3$ is unramified.


\subsection{Preliminaries}
\label{ss:h=1 prelim}


In  \S \ref{ss:h=1 prelim} we assume that  $K_1$ and $K_2$ are fields, with $K_1/F$  unramified and $K_2/F$  ramified.
In particular, $K=K_1\otimes_F K_2$ is a biquadratic field extension of $F$.
Fix $\co_F$-algebra generators $x_1 \in \co_{K_1}$ and $x_2 \in \co_{K_2}$ with
\[
\ord_{K_1}(x_1)=0, \qquad \ord_{K_2}(x_2)=1.
\]

Let $B$ be a central simple $F$-algebra of dimension $4$.  
Thus $B$ is either the algebra $M_2(F)$, or  the unique quaternion division algebra over $F$.
Fix $F$-algebra embeddings
\begin{align*}
\Phi_1 : K_1 \to  B, \qquad  \Phi_2 : K_2 \to  B .
\end{align*}
As in \eqref{w def} and \eqref{z def}, define elements of $B$ by 
\begin{align*}
\mathbf{w}  & =  \Phi_1(x_1)\Phi_2(x_2)+\Phi_2(x_2^{\sigma_2})\Phi_1(x_1^{\sigma_1})   \\
\mathbf{z}  & =  \Phi_1(x_1)\Phi_2(x_2)-\Phi_2(x_2 )\Phi_1(x_1)  .
\end{align*}
Similarly, let $\mathbf{s} \in C=B\otimes_F K_3$  be  as in \eqref{better s},

\begin{proposition}\label{prop:h=1 invariant}
We have  $\mathbf{s}  \in K_3^\times$,   and  $\Inv_{(\Phi_1,\Phi_2)}(T) = T - \mathbf{s}$.  
\end{proposition}

\begin{proof}
  Recall from \S \ref{ss:polynomial} that $\Inv_{(\Phi_1,\Phi_2)}(T)  \in K_3 [T]$  is monic of degree $h=1$, and satisfies $\Inv_{(\Phi_1,\Phi_2)}( \mathbf{s}) =0$.   Hence $\mathbf{s} \in K_3$ and 
\[
\Inv_{(\Phi_1,\Phi_2)}(T) = T - \mathbf{s}.
\]
The relation $\mathbf{s}+\mathbf{s}^{\sigma_3}=1$ of   Proposition \ref{prop:st properties} shows that $\mathbf{s} \neq 0$.  As our assumptions on $K_1$ and $K_2$ imply that $K_3$ is a field, we have $\mathbf{s} \in K_3^\times$.
\end{proof}

\begin{corollary}
The pair  $(\Phi_1,\Phi_2)$ is  regular semisimple.
\end{corollary}

\begin{proof}
Use Proposition \ref{prop:h=1 invariant} and the criterion of Proposition \ref{prop:rss}.  
\end{proof}

\begin{lemma}\label{lem:quadratic simplified}
We have $\mathbf{w} , \mathbf{z}^2 \in F$, and 
\begin{equation}\label{z quadratic}
\mathbf{z}^2 = \mathbf{w}^2 + \pi \cdot(  a    \mathbf{w} + b )
\end{equation}
for some  $a,b\in \co_F$ satisfying  $\ord(a) \ge \ord(b)=0$.   Moreover,  $\mathbf{z} \in B^\times$.
\end{lemma}

\begin{proof}
Consider the element 
\begin{equation}\label{preord estimate}
\frac{\mathrm{Tr}(x_1^2)}{\mathrm{Nm}(x_1)} + \frac{\mathrm{Tr}(x_2^2)}{\mathrm{Nm}(x_2)} 
=   \frac{ ( x_1-x_1^{\sigma_1})^2 }{x_1 x_1^{\sigma_1}}   +  \frac{ ( x_2+x_2^{\sigma_2})^2 }{x_2 x_2^{\sigma_2}} . 
\end{equation}
The first term on the right hand side lies in  $\co^\times_F$, while 
\[
\ord_{K_2}(x_2+x_2^{\sigma_2}) = 2 \ord_F (  x_2+x_2^{\sigma_2}  )  > 1 = \ord_{K_2}(x_2)
\]
implies that the   second term on the right lies in $\pi \co_F$.  Thus \eqref{preord estimate}  lies in $\co_F^\times$.
Combining this with the relation 
\[
\mathrm{Tr}(x_1^2)\mathrm{Nm}(x_2)+\mathrm{Tr}(x_2^2)\mathrm{Nm}(x_1) 
= \mathrm{Nm}(x_2)\mathrm{Nm}(x_1)\left(\frac{\mathrm{Tr}(x_1^2)}{\mathrm{Nm}(x_1)} + \frac{\mathrm{Tr}(x_2^2)}{\mathrm{Nm}(x_2)} \right),
\]
we find that 
\[
\ord_F\big(  \mathrm{Tr}(x_1^2)\mathrm{Nm}(x_2)+\mathrm{Tr}(x_2^2)\mathrm{Nm}(x_1)  \big)  =1.
\]
Combining  this with Proposition \ref{prop:wz properties}  shows that \eqref{z quadratic} holds with
\[
a = -  \frac{ \mathrm{Tr}(x_1)\mathrm{Tr}(x_2) }{\pi} , \qquad b= \frac{\mathrm{Tr}(x_1^2)\mathrm{Nm}(x_2)+\mathrm{Tr}(x_2^2)\mathrm{Nm}(x_1)}{\pi}.
\]

Recall from Proposition \ref{prop:wz properties} that $\mathbf{w}$ commutes with both $\Phi(K_1)$ and $\Phi(K_2)$.
Each of these subalgebras  is equal to its own centralizer in $B$, and hence $\mathbf{w} \in \Phi(K_1) \cap \Phi(K_2) =F$.  
The inclusion $\mathbf{z}^2 \in F$ follows from this and \eqref{z quadratic}. 
 For the final claim, Proposition \ref{prop:h=1 invariant} tells us that  $\mathbf{s}\in K_3^\times$, hence $\mathbf{t} \in C^\times$ by Proposition \ref{prop:st properties}, hence $\mathbf{z} \in B^\times$ by Proposition \ref{prop:alt st}.
\end{proof}

\begin{lemma}\label{lem:quaternion valuation}
If $B$ is a matrix algebra then  $\ord_F(\mathbf{z}^2)$ is even. 
 If $B$ is a division algebra then  $\ord_F(\mathbf{z}^2)$ is odd.
\end{lemma}

\begin{proof}
The essential point is the relation 
$
\mathbf{z}   \Phi_1(x) = \Phi_1(x^{\sigma_1}) \mathbf{z}
$
of Proposition \ref{prop:wz properties}.
If $\ord_F(\mathbf{z}^2)$ is even then, as $K_1/F$ is unramified, there is an $x\in K_1$ such that 
$xx^{\sigma_1} = \mathbf{z}^2$.  This implies 
\[
( \Phi_1(x) - \mathbf{z}) ( \Phi_1(x^{\sigma_1}) + \mathbf{z}) =0 ,
\]
and so $B\iso M_2(F)$.  
Conversely, if $B\iso M_2(F)$ then pick any nonzero $v \in F^2$.  
The embedding $\Phi_1:K_1 \to M_2(F)$ makes $F^2$ into a $K_1$-vector space of dimension $1$, and so
$\mathbf{z}   v = \Phi_1(x)  v$ for some $x\in K_1$.  This implies 
\[
\mathbf{z}^2  v = \mathbf{z} \Phi_1(x)   v = \Phi_1(x^{\sigma_1}) \mathbf{z}   v = \Phi_1(x x^{\sigma_1})  v,
\]
and so $\mathbf{z}^2 = x x^{\sigma_1}$.  Thus $\ord_F(\mathbf{z}^2)$ is even.
\end{proof}


\subsection{Calculation of an orbital integral}
\label{ss:h=1 orbital calc}


We keep the notation and assumptions of the previous subsection.  
In particular, we continue to assume that $K_1/F$ is unramified, while $K_2/F$ is ramified.

If we view  $x_1\in \co_{K_1}$ and $x_2\in \co_{K_2}$ as elements of $K$,  then
\[
x_3=x_1x_2^{\sigma_2}+x_1^{\sigma_1}x_2 \in K_3
\]
generates $K_3$ as an $F$-algebra.  
In fact, it is easy to see that 
\[
\ord_{K_3}(x_3) =1
\]
and hence  $\co_{K_3} = \co_F[x_3]$, as $K_3/F$ is ramified.
Recalling that $K_0 = F\times F$,  define 
\[
\Phi_0  :   K_0 \to M_2(F) ,\qquad 
\Phi_3 :  K_3 \to M_2(F)
\]
by 
\[
\Phi_0(a,b)=\begin{pmatrix}b&0\\0&a\end{pmatrix}, 
\qquad
\Phi_3(x_3)=\begin{pmatrix}\mathbf{w}&1\\\ \mathbf{m} &\mathrm{Tr}(x_3)-\mathbf w \end{pmatrix},
\]
where $\mathbf{m} = \mathbf{w}   \cdot (\mathrm{Tr}(x_3)-\mathbf{w})-\mathrm{Nm}(x_3)\in F.$

\begin{lemma}
The pair  $(\Phi_0,\Phi_3)$ matches $(\Phi_1,\Phi_2)$.
\end{lemma}

\begin{proof}
Set  $x_0=(0,1)\in K_0$, and let 
\[
\mathbf{w}'   =  \Phi_0(x_0)\Phi_3(x_3)+\Phi_3(x_3^{\sigma_3})\Phi_0(x_0^{\sigma_0})  \in M_2(F)
\]
be the element associated to the pair $(\Phi_0,\Phi_3)$ by \eqref{w def}.
Using 
\[
\Phi_0(x_0)\Phi_3(x_3)
=\begin{pmatrix}\mathbf{w}&1\\0&0\end{pmatrix},
\qquad
\Phi_3(x_3^{\sigma_3})\Phi_0(x_0^{\sigma_0})
=\begin{pmatrix} 0&-1\\0&\mathbf{w} \end{pmatrix}
\]
we see that $\mathbf{w}'=\mathbf{w}$. 

Recalling Proposition \ref{prop:alt st}, consider the elements 
\[
\mathbf{s}   = \frac{-(x_1x_2^{\sigma_2}+x_2x_1^{\sigma_1})+  \mathbf{w}  }{(x_1-x_1^{\sigma_1})(x_2-x_2^{\sigma_2})} ,\qquad 
\mathbf{s}'   = \frac{-(x_0x_3^{\sigma_3}+x_3x_0^{\sigma_0})+  \mathbf{w}'  }{(x_0-x_0^{\sigma_0})(x_3-x_3^{\sigma_3})}
\]	
of $K_3 \subset M_2(K_3)$ associated to the pairs $(\Phi_1,\Phi_2)$ and $(\Phi_0,\Phi_3)$.

Somewhat confusingly, $\mathbf{s}$ and $\mathbf{s}'$ are viewed as elements of the rightmost copies of $K_3$ in the diagrams
\eqref{biquadratic} and \eqref{biquadratic2}, which we identify.
In particular,  one must be mindful of the conventions explained after \eqref{biquadratic2}.
If we identify $K_0$ and $K_3$ as subalgebras of $K_3 \times K_3$ via 
$x_0 \mapsto (0,1)$ and $x_3 \mapsto (x_3,x_3)$, then
\[
x_0x_3^{\sigma_3} + x_3 x_0^{\sigma_0} = (x_3  , x_3^{\sigma_3} ) \in K_3 \times K_3
\] 
is identified with the element $x_3=x_1x_2^{\sigma_2}+x_1^{\sigma_1}x_2$ in the \emph{other} copy of $K_3$ in the diagram \eqref{biquadratic2}.
In other words, 
\[
x_0x_3^{\sigma_3} + x_3 x_0^{\sigma_0}  = x_1x_2^{\sigma_2}+x_1^{\sigma_1}x_2. 
\]
Similarly,
\[
(x_0-x_0^{\sigma_0})(x_3-x_3^{\sigma_3})
= (  x_3^{\sigma_3}-x_3 ,  x_3 - x_3^{\sigma_3}   ) \in K_3 \times K_3 
\]
is identified with $x_3^{\sigma_3}-x_3 = (x_1-x_1^{\sigma_1})(x_2-x_2^{\sigma_2})$ in the  rightmost copy of $K_3$ in the diagram \eqref{biquadratic2}.  In other words,
\[
(x_0-x_0^{\sigma_0})(x_3-x_3^{\sigma_3}) = (x_1-x_1^{\sigma_1})(x_2-x_2^{\sigma_2}).
\]

Having now shown that  $\mathbf{s}=\mathbf{s}'$, the pairs $(\Phi_1,\Phi_2)$ and $(\Phi_0,\Phi_3)$ match.
\end{proof}

\begin{proposition}
If $\mathbf{1}$ is the characteristic function of $\GL_2(\co_F)$, then
\[
O_{(\Phi_0,\Phi_3)} ( \mathbf{1} ; s ,\eta )  = 
\begin{cases}
 1  & \mbox{if }  \ord_F(\mathbf{w}) =0 \\
 1 - q^{-s} & \mbox{if }  \ord_F(\mathbf{w}) >0 \\
 0 & \mbox{if }\ord_F(\mathbf{w}) <0.
\end{cases}
\]
Note that when $\ord_F(\mathbf{w}) \ge 0$,  both $\Phi_0(\co_{K_0})$ and $\Phi_3(\co_{K_3})$ are contained in $M_2(\co_F)$.  This allows us to take  $g_0=g_3=1$ in \eqref{integrality shift 03}, and so 
 remove the ambiguity in the orbital integral noted in Remark \ref{rem:orbital invariant}.
\end{proposition}

\begin{proof}

In the notation of \eqref{03 centralizers}, we have  $H_0= \Phi_0(K_0^\times)$ and $H_3=\Phi_3(K_3^\times)$.
As  $K_3/F$ is ramified, and $x_3\in K_3$ is a uniformizing parameter,  we have 
\[
F^\times\backslash K_3^\times 
=  \mathcal{O}_F^\times\backslash  (\mathcal{O}_{K_3}^\times\sqcup x_3 \mathcal{O}_{K_3}^\times ).
\]
Choosing a $g_3 \in \GL_2(F)$ such that  
\[
\Phi_3(\co_{K_3}) \subset g_3 M_2(\co_F) g_3^{-1} ,
\]
 the orbital integral of Definition \ref{def:orbital 03}  simplifies to 
\begin{align}
O_{(\Phi_0,\Phi_3)}(\mathbf 1;s, \eta) 
&=
 \int_{F^\times\backslash (F^\times\times F^\times \times K_3^\times) } 
 \mathbf{1}(\Phi_0(a,b)^{-1}\Phi_3(x) g_3 )\eta(x)   |a/b|^s \, da\, db\, dx     \nonumber  \\
&=
  \int_{F^\times\times F^\times\times\mathcal{O}_{K_3}^\times}
\mathbf{1} (\Phi_0(a,b)^{-1}\Phi_3(x) g_3 )   | a / b |^s\, da\, db\,  dx   \nonumber  \\ 
& \quad -\int_{F^\times\times F^\times\times  \mathcal{O}_{K_3}^\times}
\mathbf{1}(\Phi_0(a,b)^{-1}\Phi_3( x_3 x) g_3 ) | a/b |^s \, da\, db\, dx \nonumber  \\
&=
  \int_{F^\times\times F^\times }
\mathbf{1} (\Phi_0(a,b)^{-1} g_3 )   | a / b |^s\, da\, db    \nonumber  \\ 
& \quad -\int_{F^\times\times F^\times }
\mathbf{1}(\Phi_0(a,b)^{-1}\Phi_3( x_3 ) g_3 ) | a/b |^s \, da\, db  .
 \label{h=1 orbital decomp}
\end{align}

If $\ord_F(\mathbf{w}) \ge 0$ then we may take $g_3=1$, and compute
\[
\int_{F^\times\times F^\times }
\mathbf{1} (\Phi_0(a,b)^{-1}  )   | a/b | ^s\, da\, db  =1 .
\]
If  $\ord_F(\mathbf{w}) =0$ then 
\[
\Phi_3(x_3) \in   \begin{pmatrix}    1 & 0 \\ u & \pi \end{pmatrix} \GL_2(\co_F)
\]
for some $u \in \co_F^\times$, and one easily checks that 
\[
\int_{F^\times\times F^\times }
\mathbf{1}(\Phi_0(a,b)^{-1}\Phi_3( x_3)) | a/b |^s \, da\, db 
= 0 .
\]
If $\ord_F(\mathbf{w}) >0$ then we instead have 
\[
\Phi_3(x_3) \in   \begin{pmatrix}    1 & 0 \\ 0 & \pi \end{pmatrix} \GL_2(\co_F),
\]
and one easily checks that 
\[
\int_{F^\times\times F^\times }
\mathbf{1}(\Phi_0(a,b)^{-1}\Phi_3( x_3)) | a/b |^s \, da\, db
= 
q^{-s}.
\]
Combining these calculations with \eqref{h=1 orbital decomp} proves the claim for $\ord_F(\mathbf{w}) \ge 0$.

Now suppose  $\ord_F(\mathbf{w}) <0$. In this case we may choose
\[
g_3 = \begin{pmatrix}
1 & \mathbf{w} \\ 0 & \mathbf{m}
\end{pmatrix} \in M_2(F).
\]
Direct calculation shows that 
\[
 g_3 \not\in \Phi_0(a,b) \GL_2(\co_F),\qquad  \Phi_3( x_3 ) g_3 \not\in \Phi_0(a,b) \GL_2(\co_F)
\]
for all $a,b\in F^\times$, and so \eqref{h=1 orbital decomp} vanishes.
\end{proof}


\subsection{Central values} 


In this subsection we let $K_1$ and $K_2$ be any quadratic \'etale $F$-algebras, and fix $F$-algebra embeddings
\begin{align*}
\Phi_0 &: K_0 \to M_2(F), &  \Phi_1&: K_1 \to M_2(F) \\
\Phi_3 &: K_3 \to M_2(F), &   \Phi_2&: K_2 \to M_2(F) .
\end{align*}
 
 \begin{theorem}
If  $(\Phi_0,\Phi_3)$ and $(\Phi_1,\Phi_2)$ are  regular semisimple and matching,  then
 \begin{equation*}
 \pm O_{(\Phi_0,\Phi_3)} ( \mathbf{1} ; 0,\eta) =   O_{(\Phi_1,\Phi_2)} ( \mathbf{1}) .
\end{equation*}
 \end{theorem}

\begin{proof}
If either one of $K_1$ or $K_2$ is the split algebra $F\times F$, the claim is known by Proposition \ref{prop:split FL}.
If $K_1$ and $K_2$ are both unramified field extensions, then $K_1\iso K_2$ and the result is known by work of Guo (Remark \ref{rem:FLequivalence}).
Our assumption that $K/K_3$ is unramified excludes the possibility that  $K_1$ and $K_2$ are both ramified field extensions.

This leaves us with the case in which  $K_1$ and $K_2$ are both fields, and exactly one of them is ramified over $F$.
Under this assumption we will prove a more precise statement: 
 if there is a maximal order in $M_2(F)$ that contains both $\Phi_1(\co_{K_1})$ and $\Phi(\co_{K_2} )$, then both sides of the desired equality are equal to $1$.
If no such maximal order exists, then both sides are $0$.

Without loss of generality,  assume that $K_1$ is unramified and $K_2$ is ramified.
Let $x_1\in K_1$, $x_2\in K_2$,  $\mathbf{w} \in F$ and $\mathbf{z} \in M_2(F)$ be as in \S \ref{ss:h=1 prelim}.

Assume there is a maximal order in $M_2(F)$ that contains both $\Phi_1(\co_{K_1})$ and $\Phi(\co_{K_2} )$.
In particular  $\ord_F(\mathbf{w}) \ge 0$, 
and the calculations of \S \ref{ss:h=1 orbital calc} that
\[
\pm O_{(\Phi_0,\Phi_3)} ( \mathbf{1} ; 0 ,\eta )  = 1.
\]
After conjugating $(\Phi_1,\Phi_2)$ by an element of $\GL_2(\co_F)$, we may assume that this maximal order is $M_2(\co_F)$, and the orbital integral of Definition \ref{def:orbital 12}  simplifies to  
\begin{align*}
O_{(\Phi_1,\Phi_2)}( \mathbf{1} ) 
 & = \int_{  F^\times \backslash (K_1^\times \times K_2^\times)  } \mathbf{1} (  \Phi_1(h_1^{-1})  \Phi_2(h_2)   ) \, dh_1\, dh_2  \\
  & = \int_{  (\co_F^\times \backslash \co_{K_1}^\times)  \times K_2^\times   } \mathbf{1} (  \Phi_1(h_1^{-1})  \Phi_2(h_2)   ) \, dh_1\, dh_2  \\
   & = \int_{   K_2^\times   } \mathbf{1} (    \Phi_2(h_2)   ) \,  dh_2  \\
   & = 1.
\end{align*}

Now suppose there is no maximal order in $M_2(F)$ containing both $\Phi_1(\co_{K_1})$ and $\Phi_2(\co_{K_2} )$.
Replacing the pair $(\Phi_1,\Phi_2)$ by a conjugate, we may assume that $\Phi_1(\co_{K_1}) \subset M_2(\co_F)$, and then 
\begin{align}
O_{(\Phi_1,\Phi_2)}( \mathbf{1} ) 
& = \int_{  F^\times \backslash ( K_1^\times \times K_2^\times)   }  \mathbf{1} 
\big(  \Phi_1(h_1^{-1})  \Phi_2(h_2)  g_2 \big) \, dh_1\, dh_2    \nonumber \\
& = \int_{  ( \co_F^\times \backslash \co_{ K_1}^\times ) \times K_2^\times   }  \mathbf{1} 
\big(  \Phi_1(h_1^{-1})  \Phi_2(h_2)  g_2 \big) \, dh_1\, dh_2   \nonumber  \\
& = \int_{  K_2^\times   }  \mathbf{1} 
\big(   \Phi_2(h_2)  g_2 \big) \, dh_2   \label{no lattice orbital}
\end{align}
for any $g_2\in \GL_2(F)$ such that 
\[
\Phi_2( \co_{K_2})  \subset   g_2  M_2(\co_F)  g_2^{-1}.
\]
Using the fact that $K_2/F$ is ramified, we may scale $g_2$ by an element of $\Phi_2(K_2^\times)$ to assume that 
$\ord_F(\det(g_2)) =0$. 

If \eqref{no lattice orbital} is nonzero, there there is some  $h_2\in K_2^\times$ such that 
\[
\Phi_2(h_2)  g_2 \co_F^2 = \co_F^2.
\]
On the other hand, $\Phi_2(h_2) \in g_2 M_2(\co_F) g_2^{-1}$, satisfies
\[
\Phi_2(h_2)  g_2 \co_F^2 \subset   g_2 \co_F^2,
\]
and equality holds as
\[
\ord_F (  \det(  \Phi_2(h_2)) ) =  \ord_F ( \det(  \Phi_2(h_2) g_2  )  )=0
\]
Thus $ \co_F^2 =  g_2 \co_F^2$, which implies  $g_2 \in \GL_2(\co_F)$, which  implies 
$\Phi_2(\co_{K_2}) \subset M_2(\co_F)$.
This  contradicts  our hypothesis on the pair $(\Phi_1,\Phi_2)$, and we conclude that  \eqref{no lattice orbital} is equal to $0$.

Still assuming there is no maximal order in $M_2(F)$ containing both $\Phi_1(\co_{K_1})$ and $\Phi_2(\co_{K_2} )$,
we claim that 
\begin{equation}\label{w pole}
\ord_F(\mathbf{w}) <0.
\end{equation}
Indeed, if $\ord_F(\mathbf{w}) \ge 0$ then Lemma \ref{lem:quadratic simplified} and the relation 
\[  \mathbf{z} \cdot \Phi_1(x_1)    =\Phi_1(x_1^{\sigma_1})   \cdot \mathbf{z}\]
of Proposition \ref{prop:wz properties} imply that 
\[
\mathrm{Span}_{\co_F} \{ 1,  \mathbf{z} , \Phi_1(x_1)  ,\mathbf{z}  \Phi_1(x_1)    \} \subset M_2(F)
\]
is an $\co_F$-subalgebra.
It follows that there is a maximal order $R \subset M_2(F)$ that contains $\mathbf{w}$,  $\mathbf{z}$, and all of $\Phi_1(\co_{K_1})$.
Our assumption that  $K_1/F$ is unramified implies   $x_1 - x_1^{\sigma_1} \in \co_F^\times$, and so the relations
 \begin{align*}
\mathbf{w} - \mathbf{z} 
&= \Phi_2(x_2^{\sigma_2}) \Phi_1(x_1^{\sigma_1}) + \Phi_2(x_2)\Phi_1(x_1)  \\
&=  \mathrm{Tr}(x_2) \Phi_1(x_1^{\sigma_1})    + \Phi_2(x_2)   \Phi_1( x_1 - x_1^{\sigma_1})    
\end{align*}
imply  $\Phi_2(x_2) \in R$.  Thus $R$ contains both $\Phi_1(\co_{K_1})$ and $\Phi_2(\co_{K_2})$, contrary to our hypotheses.
Hence  \eqref{w pole} holds, and the vanishing of $O_{(\Phi_0,\Phi_3)} ( \mathbf{1} ; 0 ,\eta )$ follows from 
the calculations of \S \ref{ss:h=1 orbital calc}. 
\end{proof}


\subsection{Central derivatives}


Now return to the setting of \S \ref{ss:AFL}, with $h=1$.  Thus $\mathscr{G}$ is formal $\co_F$-module over $\kk$ of height $2$,  and we are given $F$-algebra embeddings
\[
\Phi_1 : K_1 \to B(\mathscr{G}) ,\qquad \Phi_2 : K_2 \to B(\mathscr{G}) 
\]
forming a regular semisimple pair $(\Phi_1,\Phi_2)$, with corresponding closed immersions   
$f_1: Y^\bullet_1 \to X^\bullet$ and  $f_2: Y^\bullet_2 \to X^\bullet$ of formal schemes over $\co_{\breve{F}}$. 

The \'etale $F$-algebra $L$ appearing in Conjecture \ref{conj:ABFL}, being of dimension $h=1$, is just $F$ itself.
This allows us to take $X^{(0)}$ and $X^{(1)}$ as representatives for the $L^\times$ orbits of $ \pi_0(X^\bullet)$,  and Conjecture \ref{conj:ABFL} is a consequence of the following result.

\begin{theorem}
Given $F$-algebra embeddings
\[
\Phi_0 : K_0 \to M_2(F)  ,\qquad \Phi_3 : K_3 \to M_2(F) 
\]
with $(\Phi_0,\Phi_3)$ matching $(\Phi_1,\Phi_2)$, we have
\[
   \frac{ \pm 1}{ \log(q)  }\frac{d}{ds} O_{(\Phi_0,\Phi_3)} ( \mathbf{1} ; s ,\eta)  \big|_{s=0}    
=
\mathrm{len}_{\co_{\breve{F}}} \, H^0 \big( X^{(0)}   ,  f_{ 1*} \co_{ Y^{(0)}_1 }  \otimes_{\co_{X^{(0)}}} f_{ 2*} \co_{ Y^{(0)}_2 }  \big)  
\]
and
\[
H^0 \big( X^{(1)}   ,  f_{ 1*} \co_{ Y^{(1)}_1 } \otimes_{\co_{X^{(1)}}} f_{ 2*} \co_{ Y^{(1)}_2 }  \big) =0.
\]
\end{theorem}

\begin{proof}
Note that $B(\mathscr{G})$ is a quaternion division algebra over $F$, and so  $K_1$ and $K_2$ are fields.
If $K_i/F$ is unramified, then any quasi-isogeny of formal $\co_{K_i}$-modules has even height when viewed as a quasi-isogeny of underlying formal $\co_F$-modules.  It follows that the image of
\[
f_i : Y_i^\bullet \to X^\bullet
\]
only meets those $X^{(\ell)}$ with $\ell$ even.
We deduce that if either  of $K_1$ or $K_2$ is unramified then  
\[
 f_{ 1*} \co_{ Y^{(1)}_1 }  \otimes_{\co_{X^{(1)}}}f_{ 2*} \co_{ Y^{(1)}_2 } =0.
 \]

If $K_1$ and $K_2$ are both unramified over $F$, then $K_1\iso K_2$ and the claim follows from the calculations of 
 \S 7 of \cite{Li}. 
   Our assumption that $K/K_3$ is unramified excludes the possibility that $K_1$ and $K_2$ are both ramified, and so it only remains to consider the case in which one of $K_1$ and $K_2$ is ramified and the other is unramified.  We will prove that in this case both sides of the first  equality in the theorem are equal to $1$.

Without loss of generality we may assume that $K_1$ is unramified and $K_2$ is ramified.
Let $x_1\in \co_{K_1}$ and $x_2\in \co_{K_2}$ be as in \S \ref{ss:h=1 prelim}, and let
\begin{align*}
\mathbf{w}  & =  \Phi_1(x_1)\Phi_2(x_2)+\Phi_2(x_2^{\sigma_2})\Phi_1(x_1^{\sigma_1})   \\
\mathbf{z}  & =  \Phi_1(x_1)\Phi_2(x_2)-\Phi_2(x_2 )\Phi_1(x_1)  
\end{align*}
be the corresponding elements of $B(\mathscr{G})$.  The unique maximal order of $B(\mathscr{G})$ must contain both $\Phi_1(\co_{K_1})$ and $\Phi_2(\co_{K_2})$, and so also contains $\mathbf{w}$ and $\mathbf{z}$.
By Lemma \ref{lem:quadratic simplified}, we therefore have $\mathbf{w} \in \co_F$ and $\mathbf{z}^2 \in \co_F$,
and 
\[
\mathbf{z}^2 = \mathbf{w}^2 + \pi \cdot(  a    \mathbf{w} + b )
\]
for some  $a,b\in \co_F$ satisfying  $\ord(a) \ge \ord(b)=0$. 
  If $\ord_F(\mathbf{w}) =0$, this would imply $\ord_F(\mathbf{z}^2) =0$, contradicting  Lemma \ref{lem:quaternion valuation}.  Therefore
\begin{equation}\label{w ord}
\ord_F(\mathbf{w}) >0,
\end{equation}
and the results of \S \ref{ss:h=1 orbital calc} imply 
\begin{equation}\label{h=1 derivative}
   \frac{ \pm 1}{ \log(q)  }\frac{d}{ds} O_{(\Phi_0,\Phi_3)} ( \mathbf{1} ; s ,\eta)  \big|_{s=0}  =1.
\end{equation}

Now fix $F$-algebra embeddings 
\[
\Psi_1 : K_1 \to M_2(F) ,\qquad \Psi_2: K_2 \to M_2(F)
\]
satisfying $\Psi_i(\co_{K_i}) \subset M_2(\co_F)$. 
As in \eqref{moving resultant},  for any $g\in \GL_2(\co_F)$ let 
\[
R(g) = \mathrm{Res} \big( \Inv_{(\Phi_1,\Phi_2)} , \Inv_{(\Psi_1, g \Psi_2 g^{-1} )}  \big)   \in K_3 .
\]

We claim that 
\begin{equation}\label{easy resultant}
\ord_{K_3}( R(g) ) = -\ord_F(d_1 d_2),
\end{equation}
where $d_i \in \co_F$  generates  the discriminant of $K_i/F$.  (Of course $d_1$ is a unit.)
Define  elements of $M_2(\co_F)$ by
\begin{align*}
\mathbf{w}'  & =  \Psi_1(x_1) \Psi_2(x_2)+\Psi_2(x_2^{\sigma_2})\Psi_1(x_1^{\sigma_1})   \\
\mathbf{z}'  & =  \Psi_1(x_1) \Psi_2(x_2)-\Psi_2(x_2 )\Psi_1(x_1) .
\end{align*} 
If  $\mathbf{s} \in C$ and  $\mathbf{s}' \in M_2(K_3)$ denote the elements 
  constructed from  $(\Phi_1,\Phi_2)$ and $(\Psi_1,\Psi_2)$  as in \eqref{better s}, then
Propositions \ref{prop:h=1 invariant} and   \ref{prop:alt st} imply 
\begin{equation}\label{explicit res}
R(1)=\mathrm{Res}( T- \mathbf{s}  , T- \mathbf{s}'  ) = \mathbf{s} -\mathbf{s}' 
= \frac{\mathbf{w}-\mathbf{w}' }{ (x_1-x_1^{\sigma_1})( x_2 - x_2^{\sigma_2} ) }.
\end{equation}
As above, Lemma \ref{lem:quadratic simplified} implies that  $\mathbf{w}' \in \co_F$ and $(\mathbf{z}')^2 \in \co_F$
 satisfy
\[
(\mathbf{z}')^2 = (\mathbf{w}')^2 + \pi \cdot(  a    \mathbf{w}' + b )
\]
for some  $a,b\in \co_F$ with  $\ord(a) \ge \ord(b)=0$.  
If $\ord_F(\mathbf{w}') >0$ then $\ord_F( (\mathbf{z}')^2)=1$, contradicting Lemma \ref{lem:quaternion valuation}.
Hence
$
\ord_F(\mathbf{w}') = 0.
$
Combining  this with \eqref{w ord} and \eqref{explicit res}  shows that
\[
 \ord_{K_3}(R(1))  = - \ord_{K_3} (  (x_1-x_1^{\sigma_1})( x_2 - x_2^{\sigma_2} ) ).
 \]
Elementary calculation shows that  the right hand side is  $-\ord_F(d_1d_2)$, completing the proof of \eqref{easy resultant} when $g=1$.  The proof for general $g\in \GL_2(\co_F)$ proceeds by replacing $\Psi_2$ with $g \Psi_2 g^{-1}$ throughout the argument.

Using \eqref{easy resultant},  the $m=0$ case of Theorem \ref{thm:final intersection} reduces to 
\[
\mathrm{len}_{\co_{\breve{F}}} \, H^0 \big( X^{(0)}   ,  f_{ 1*} \co_{ Y^{(0)}_1 } \otimes_{\co_X} f_{ 2*} \co_{ Y^{(0)}_1 }  \big)
 =
 | d_1d_2| ^{- \frac{1}{2}}  
 \int_{ \GL_2(\co_F)  }  \frac{ dg} {  | R(g) |  }    =  1 ,
\]
and comparison with \eqref{h=1 derivative} completes the proof.
\end{proof}

\appendix


\section{Comparisons with earlier work}


When $K_1=K_2$ our results and conjectures reduce to those of \cite{Li}, but some aspects  of this are not completely obvious.  
In this appendix we provide some results to help guide the reader in the comparison between this paper and \cite{Li}.  

One consequence of the comparison is Proposition \ref{prop:truerss},  which shows that our notion of regular semisimplicity from Definition \ref{def:rss} is equivalent to the more familiar notion from geometric invariant theory.


\subsection{An alternate construction of $\mathbf{s}$}


Return to the setting of \S \ref{ss:polynomial}.  Thus $F$ is an arbitrary field, $B$ is a central simple $F$-algebra of dimension $4h^2$, and we are given $F$-algebra embeddings 
\begin{equation*}
\Phi_1 : K_1 \to B,\qquad  \Phi_2: K_2 \to B 
\end{equation*}
in which each $K_i$ is a quadratic \'etale $F$-algebra.

We provide  a different construction of the element 
\[
\mathbf{s} \in C=B\otimes_F K_3
\]
defined by \eqref{better s}.
This will allow us to compare our invariant polynomials with the invariant polynomials defined (in the special case $K_1= K_2$) in \cite{Li}, and to compare our definition of regular semisimple pair with the more common one from  geometric invariant theory.

 The $K_3$-algebra embeddings $\Phi_1, \Phi_2 : K \to C$  of   \eqref{induced embeddings} are conjugate by Corollary \ref{cor:ns}, and hence there is a  $c\in C^\times$ such that
\begin{equation}\label{c conj}
	\Phi_2(x) = c^{-1} \cdot  \Phi_1(x)  \cdot c.
\end{equation}
There is a  $\Z/2\Z$-grading $C = C_+\oplus C_-$  in which 
\begin{align}
C_+  &= \{ a\in C :  \forall y\in K,\,  \Phi_1 (y) \cdot a  = a  \cdot \Phi_1 (y) \} \label{grading} \\
C_-  &= \{ a\in C :  \forall y\in K,\,  \Phi_1 (y) \cdot a = a \cdot  \Phi_1(y^{\tau_3}) \}. \nonumber
\end{align}
Denote by $c_\pm$ the projection of $c$ to $C_\pm$.

\begin{proposition}\label{prop:second invariant}
The element  $c_+ - c_-\in C$ is invertible, and 
	\begin{equation*} 
		 \mathbf{s} =  (c_+ + c_-)^{-1}  \cdot c_+  \cdot (c_+ - c_-)^{-1} \cdot c_+ .
	\end{equation*}
\end{proposition}

\begin{proof}
If $y\in K^\times$ is any element with $y+ y^{\tau_3} =0$,  then 
\begin{equation*}
c_+ - c_-=\Phi_1(y)^{-1} \cdot (c_+ + c_-)  \cdot \Phi_1(y) \in C^\times,
\end{equation*}
proving the first claim.

For the second claim,  fix a $K_3$-algebra generator  $y\in K$. 
The element $c$ was chosen so that  $\Phi_1(y)\cdot (c_++c_-)=(c_++c_-)\cdot \Phi_2(y)$, and therefore
\begin{equation}\label{equation: midstep}
c_+(\Phi_1(y)-\Phi_2(y))=c_-(\Phi_2(y)-\Phi_1(y^{\tau_3})).
\end{equation}
Adding $c_+(\Phi_2(y)-\Phi_1(y^{\tau_3}))$ to both sides of \eqref{equation: midstep}, we find
\[
c_+ \Phi_1(y - y^{\sigma_3}) =(c_++c_-)(\Phi_2(y)-\Phi_1(y^{\tau_3})).
\]
Subtracting  $c_+(\Phi_1(y)-\Phi_2(y^{\tau_3}))$ from  both sides of \eqref{equation: midstep}, we find
\[
	c_+ \Phi_2(y - y^{\tau_3})  =(c_+-c_-)(\Phi_1(y)-\Phi_2(y^{\tau_3})).
\]
Rewrite these two equalities as
\begin{align*}
(c_++c_-)^{-1}  \cdot c_+ &=(\Phi_2(y)-\Phi_1(y^{\tau_3}))   \cdot  \Phi_1(y  - y^{\tau_3})^{-1}     \\
(c_+-c_-)^{-1}   \cdot  c_+  & =(\Phi_1(y)-\Phi_2(y^{\tau_3}))  \cdot   \Phi_2(y -  y^{\tau_3})^{-1}
\end{align*}
to see that $(c_+ + c_-)^{-1}  \cdot c_+  \cdot (c_+ - c_-)^{-1} \cdot c_+$ is equal to
\[
(\Phi_2(y)-\Phi_1(y^{\tau_3}))   \cdot  \Phi_1(y  - y^{\tau_3})^{-1} 
\cdot (\Phi_1(y)-\Phi_2(y^{\tau_3}))  \cdot   \Phi_2(y -  y^{\tau_3})^{-1}, 
\]
and use  \eqref{intertwine 2} to see that this last expression is equal to  \eqref{better s}.
\end{proof}


\subsection{Invariant polynomials revisited}
\label{ss:alt invariants}


 We now use Proposition \ref{prop:second invariant} to compare our invariant polynomial with the notion of  invariant polynomial from \cite[Definition 1.1]{Li}.
   Let $E$ be a quadratic \'etale $F$-algebra,  fix an $F$-algebra embedding 
   \[
   \Phi : E \to B,
   \]
   and let   $B = B_+ \oplus B_-$ be the corresponding $\Z/2\Z$-grading, as in \eqref{grading}.
 
 \begin{definition}\label{def:old invariant}
 The \emph{invariant polynomial} $M_g$ of $g \in B^\times$, with respect to $\Phi : E \to B$, is the unique monic square root of the reduced characteristic polynomial of 
 \[
\mathbf{s}_g = (g_+ + g_-)^{-1} g_+ (g_+ - g_-)^{-1} g_+ \in B,
\]
where $g_\pm \in B_\pm$ is the projection of $g$.
\end{definition}

 To explain the connection with Definition \ref{def:invariant polynomial},  set $K_1=E$ and $K_2=E$, and define 
\[
 \Phi_1=\Phi :K_1 \to B   \quad\mbox{and}\quad   \Phi_2= g^{-1} \Phi g  :K_2 \to B. 
 \]
 Let $\sigma \in \Aut(E/F)$ be the nontrivial automorphism, and 
 identify 
 \[
 K = K_1 \otimes_F K_2 \iso E\times E
 \]
  via $a\otimes b\mapsto (ab, a b^\sigma)$.
 The subalgebra $K_3 \subset K$ is identified with  $F \times F \subset E\times E$, and so 
 $
 C = B\times B.
 $
The embeddings 
\[
\Phi_1,\Phi_2 : E\times E \to B\times B
\]
 of \eqref{induced embeddings} take the explicit form 
\begin{equation*} 
\Phi_1(a,b) = ( \Phi_1(a) ,\Phi_1(b) ),  \qquad
 \Phi_2(a,b) = (\Phi_2(a) , \Phi_2(b^\sigma) ).
\end{equation*}
Compare with Remark \ref{screwy 1}.

We may choose the element $c$  of \eqref{c conj} in the form 
$
c = (g , g' ) \in B^\times \times B^\times
$
 for some $g' \in B^\times$, and  (by Proposition \ref{prop:second invariant}) the element $\mathbf{s} \in C$ of \eqref{better s} is 
 $
 \mathbf{s} = ( \mathbf{s}_g , \mathbf{s}_{g'} ) .
 $
 What this shows is that the invariant polynomial 
 \[
 \Inv_{( \Phi_1,\Phi_2) } \in K_3[T] = F[T] \times F[T]
 \]
  of Definition \ref{def:invariant polynomial} is  related to that  of Definition \ref{def:old invariant} by  
  \begin{equation}\label{split invariant}
  \Inv_{( \Phi_1,\Phi_2) }  =  ( M_g, M_{g'}).
  \end{equation}
  Note that $M_{g'}(T) = (-1)^h M_g (1 -T)$ by the functional equation of Proposition \ref{prop:functional}.

We now use the above discussion and a result of Guo \cite{Guo} to compare Definition \ref{def:invariant polynomial} with the usual notion of regular semisimple from geometric invariant theory.  For the rest of this subsection we assume that $F$ is algebraically closed.
For $i\in \{1,2\}$ denote by 
\[
X_i = \{ \Phi_i : K_i \to B \}
\]
the set of all $F$-algebra embeddings of $K_i$ into $B$.  There is a natural action of the group $G=B^\times$ on $X_i$ by conjugation, and hence a diagonal action of $G$ on $X_1\times X_2$.

\begin{proposition}\label{prop:truerss}
A  point $(\Phi_1,\Phi_2) \in X_1\times X_2$ is regular semisimple in the sense of Definition \ref{def:invariant polynomial} if and only if its $G$-orbit is Zariski closed  of maximal dimension.
\end{proposition}

\begin{proof}
As we are assuming that $F$ is algebraically closed, we may 
fix  isomorphisms $B\iso M_{2h}(F)$ and  $K_i \iso F\times F$. There is a standard embedding $\Phi  : F \times F \to M_{2h}(F)$ defined by 
\[
\Phi (a,b) = \begin{pmatrix}   a I_h  &  \\  & b I_h  \end{pmatrix},
\]
which determines base point $\Phi  \in X_i$ with stabilizer
\[
H = \left\{
\begin{pmatrix}
A & \\ & B 
\end{pmatrix}  : A,B \in \GL_h(F)
\right\}  \subset \GL_{2h}(F)  \iso G.
\]
Using Corollary \ref{cor:ns} we identify $G/H \iso X_i$ as algebraic varieties (one may take this as the definition of the algebraic structure on $X_i$).

Consider the diagram
\begin{equation}\label{orbit swindle}
\xymatrix{
{   G\backslash (G\times G)   }  \ar[d]_{\iso} &  {  G \times G  } \ar[r] \ar[l]  &   {  G/H \times G/H  }    \ar[d]^{ \iso } \\
{  G  }  &   &  { X_1\times X_2  } 
}
\end{equation}
in which the vertical isomorphism on the left sends $(g_1,g_2) \mapsto g_1^{-1} g_2$.
The group  $H\times H$ acts on $G$ on the right by $  g\cdot (h_1,h_2)  = h_1^{-1} g h_2$.
The group $G$ acts diagonally on $G\times G$ by left multiplication, while $H\times H$ acts by right multiplication.  It is easy to see that the above diagram induces bijections
\[
  \left\{ \begin{array}{cc}  H \times H  \\ \mbox{orbits in } G  \end{array} \right\}  \iso
    \left\{ \begin{array}{cc}  G \times H\times H \\  \mbox{orbits in } G\times G  \end{array} \right\} 
    \iso    \left\{ \begin{array}{cc} G\mbox{ orbits in} \\ X_1\times X_2   \end{array}  \right\}    
\]
sending $g\mapsto (1,g) \mapsto  (\Phi , g \Phi g^{-1}) $.
Under both bijections, closed orbits correspond to closed orbits (use the fact that both horizontal arrows in \eqref{orbit swindle} are smooth, so images of open sets are open).  
  The stabilizers (in $H\times H$, $G\times H\times H$,  and $G$, respectively) of $g\in G$, $(g,1) \in G\times G$ and  $(\Phi , g \Phi g^{-1}) \in X_1\times X_2$ have the same dimension, as  all are isomorphic to $H\cap g Hg^{-1}$. 
  Thus, under these bijections, closed orbits of maximal dimension correspond to closed orbits of maximal dimension.

The closed $H\times H$ orbits in $G$ of maximal dimension were classified by Guo \cite{Guo}, but we follow the  discussion of this classification found in \cite[\S 2]{LM}. 
Guo's result, in the form of  \cite[Lemma 2.2]{LM},  implies that the $H\times H$ orbit of $g\in G$ is closed of maximal  dimension if and only if its invariant polynomial  (in the sense of Definition \ref{def:old invariant}) with respect to $\Phi: F \times F \to M_{2h}(F)$ has $h$ distinct roots, all different from $0$ and $1$.
 Let us denote this invariant polynomial by $M_g(T)$,
and note that it agrees with the polynomial $\mathrm{Inv}'(g,T)$ in \cite[Remark 2.4]{LM}.

Now start with a pair $(\Phi_1,\Phi_2)\in X_1\times X_2$ whose $G$ orbit corresponds under the above bijections to the $H\times H$ orbit of $g\in G$.  By Proposition \ref{prop:rss}, the pair $(\Phi_1,\Phi_2)$ is regular semisimple  if and only if each of the polynomials $M_g(T)$ and $ M_g(1-T)$  in 
\[
\mathrm{Inv}_{(\Phi_1,\Phi_2)} (T)   
\stackrel{ \eqref{split invariant}}{=}  \big( M_g(T) , (-1)^h M_g(1-T) \big) \in F[T] \times F[T]
\]
has $h$ distinct nonzero roots.  This is equivalent to $M_g(T)$ having $h$ distinct roots, all different from $0$ and $1$.  By Guo's result, this last condition   is equivalent to the $H\times H$ orbit of $g\in G$ being closed of maximal dimension, which is equivalent to the $G$ orbit of $(\Phi_1,\Phi_2) \in X_1\times X_2$ having the same property.
\end{proof}


\subsection{The Guo-Jacquet orbital integral}


Here we explain the connection between the orbital integrals of Definition \ref{def:orbital 03} and those defined in \cite{Guo, Li}.  
This is necessary to justify Remarks \ref{rem:FLequivalence} and \ref{rem:AFLequivalence}.
Throughout, $F$ is a local field,  $E$ is either $F\times F$ or an unramified separable quadratic field extension, and 
  \[ \eta_{E/F} : F^\times \to \{ \pm 1\}\] is the associated quadratic character. 
  As in the proof of Proposition \ref{prop:truerss}, 
the centralizer in $\GL_{2h}(F)$ of   the standard embedding  $\Phi : F\times F \to  M_{2h}(F)$   is the subgroup
\[
H = \left\{ \begin{pmatrix} A & \\ & B \end{pmatrix} : A,B \in \GL_h(F) \right\}.
\]

Set  $K_1=E$ and $K_2=E$.  With this choice we have 
\[
K_3= F\times F =K_0, 
\]
and  the character $\eta:K_3^\times \to \{ \pm 1\}$ determined by the quadratic extension $K/K_3$ is 
\[
\eta(a,b) = \eta_{E/F}(ab).
\]
Fix a $g\in \GL_{2h}(F)$, and define 
\[
\Phi_0:K_0 \to M_{2h}(F) \quad \mbox{and}\quad \Phi_3= g \Phi g^{-1} : K_3 \to M_{2h}(F). 
\]
The centralizers in $\GL_{2h}(F)$ of their images are 
\[
H_0=H \quad  \mbox{and} \quad  H_3= g H g^{-1}.
\]

We assume throughout that the pair $(\Phi_0,\Phi_3)$ is regular semisimple in the sense of Definition \ref{def:rss}.  
By \eqref{split invariant} and Proposition \ref{prop:rss}, this is equivalent to the polynomial $M_g(T) \in F[T]$ of Definition \ref{def:old invariant} having $h$ distinct roots, all different from $0$ and $1$.

For  every  compactly supported $f$ as in \eqref{test function}, 
define the \emph{Guo-Jacquet orbital integral}  
\begin{equation}\label{GJorbital}
O_g (f; s, \eta_{E/F} ) = \int_{ I_g    \backslash  ( H \times  H  )} f(  h^{-1} g  h'   )
\cdot | h h' |^s \cdot \eta_{E/F}( h') \, dh \, dh' .
\end{equation}
Here $\eta_{E/F}$ is viewed as a character of $H$ by $\eta_{E/F}(h) = \eta_{E/F}(\det(h))$, the character $|\cdot| : H \to \C^\times$ is defined by
\begin{equation*}
\left| \begin{pmatrix} A & \\ & B \end{pmatrix} \right| 
= \left|  \frac{\det(A)}{ \det(B)} \right| ,
\end{equation*}
  and 
\[
I_g = \{  ( h , h')  \in H \times H  :    h g= g h'  \} .
\]
(Our convention for $|h|$ differs from the one used in  \cite[(1.6)]{Li} by an inverse, so our orbital integral  differs from that one by the substitution $s\mapsto -s$.)

\begin{remark}\label{rem:Ig abelian}
The group $I_g$ is abelian.  In fact, $(h,h') \mapsto h$ defines an  isomorphism $I_g \iso  H_0 \cap H_3$, and 
 the right hand side is the unit group of an \'etale $F$-algebra $L \subset M_{2h}(F)$ of dimension $h$.
 See \eqref{toric intersection}. 
\end{remark}

\begin{remark}\label{GJcharacters}
The integral \eqref{GJorbital} is well-defined because the characters 
\begin{align*}
(h,h') \mapsto \eta_{E/F}(h'), & &   (h,h') \mapsto |h|, & &  (h,h') \mapsto |h'|
\end{align*}
are all trivial on the subgroup $I_g \subset H\times H$.   See Lemma \ref{lem:nocharacters} and the previous remark.
\end{remark}

The Guo-Jacquet orbital integral  \eqref{GJorbital} and the orbital integral of Definition \ref{def:orbital 03} do \emph{not} agree.
Nevertheless, the following holds.

\begin{proposition}\label{prop:split orbital comparison}
We have the equality
\[
O_{(\Phi_0,\Phi_3)} ( f ; 0, \eta )
= 
 O_g (f ; 0, \eta_{E/F} ) .
\]
If  $(\Phi_0,\Phi_3)$ matches a pair $(\Phi_1,\Phi_2)$ of embeddings of $E$ into a division algebra 
 (so that the two sides of the above equality vanish by  Proposition  \ref{prop:orbital vanish}),  then 
\[
\frac{d}{ds}  O_{(\Phi_0,\Phi_3)} ( f ; s , \eta ) \big|_{s=0} 
= 
  \frac{d}{ds}   O_g (f ; 2 s, \eta_{E/F} ) \big|_{s=0} .
\]
\end{proposition}

\begin{proof}
We may choose $g_0 = 1$ and $g_3= g$ in Definition \ref{def:orbital 03}, so that
\begin{equation*}
O_{(\Phi_0,\Phi_3)} (f; s, \eta ) = \int_{ ( H_0\cap H_3) \backslash (H_0\times H_3) } f(  h_0^{-1} h_3 g )
\cdot | h_0 |^s \cdot \eta( h_3) \, dh_0\, dh_3.
\end{equation*}
Making the substitution  $h=h_0$ and $h'=g^{-1} h_3 g$, we find 
\begin{align*}
 O_{(\Phi_0,\Phi_3)} (f; s, \eta )    = \int_{ I_g    \backslash  ( H \times  H  )} f(  h^{-1} g  h'   )
\cdot | h |^s \cdot \eta_{E/F}( h') \, dh \, dh' , 
\end{align*}
exactly as in Remark \ref{rem:classic orbital}.  
This last expression is not equal to the Guo-Jacquet orbital integral
\[
O_g (f; s, \eta_{E/F} ) = \int_{ I_g    \backslash  ( H \times  H  )} f(  h^{-1} g  h'   )
\cdot | h h' |^s \cdot \eta_{E/F}( h') \, dh \, dh' ,
\]
but they  visibly agree at $s=0$, proving the first claim.

From now on we assume that $(\Phi_0,\Phi_3)$ matches a pair $(\Phi_1,\Phi_2)$ of embeddings of $E$ into a division algebra.
Consider the partition $H\times H = \bigsqcup_{m\in \Z} \Omega(m)$ defined by 
\[
\Omega(m) = \{ (h,h') \in H\times H : | h'| = |h|\cdot |\varpi^m| \},
\]
where $\varpi\in F^\times$ is a uniformizing parameter.  
It follows from Remark \ref{GJcharacters} that each $\Omega(m) \subset H\times H$ is stable under both left and right multiplication by the subgroup $I_g$, and clearly
\[
O_g (f; s, \eta_{E/F} ) =
\sum_m 
|\varpi|^{ms}  \int_{ I_g    \backslash  \Omega(m) } f(  h^{-1} g  h'   )
\cdot | h  |^{2s} \cdot \eta_{E/F}( h') \, dh \, dh' .
\]
If we can prove that 
\begin{equation}\label{orbital chunk vanish}
 \int_{ I_g    \backslash  \Omega(m) } f(  h^{-1} g  h'   )
 \cdot \eta_{E/F}( h') \, dh \, dh' =0 
\end{equation}
for all $m\in \Z$, then we are done by
\begin{align*}
&  \frac{d}{ds} O_g (f; s, \eta_{E/F} ) \big|_{s=0} \\
&  =
\sum_m 
\frac{d}{ds} \left[ \int_{ I_g    \backslash  \Omega(m) } f(  h^{-1} g  h'   )
\cdot | h  |^{2s} \cdot \eta_{E/F}( h') \, dh \, dh'  \right]_{s=0} \\
& =  \frac{d}{ds} O_{( \Phi_0,\Phi_3)} (f; 2s, \eta_{E/F} ) \big|_{s=0} .
\end{align*}

We will prove  \eqref{orbital chunk vanish} by imitating the proof of $O_{(\Phi_0,\Phi_3)}(f;0,\eta)=0$ from Proposition  \ref{prop:orbital vanish}.
Let  $\mathbf{z}=\mathbf{z}_{03} \in M_{2h}(F)$ be as in \eqref{z03}, and set
\[
u = \begin{pmatrix} & I_h \\ I_h \end{pmatrix} \in \GL_{2h}(\co_F).
\]
As in the discussion leading to Proposition \ref{prop:functional equation}, if we set $u_0=u$ and $u_3=gu g^{-1}$ then 
$\mathbf{z}u_i\in H_i$.  This implies both  $\mathbf{z}u \in H$ and $g^{-1} \mathbf{z}  g u\in H$, and allows us to define
\[
 \gamma = (\mathbf{z}u , g^{-1}\mathbf{z}u g) \in I_g.
\]
 Because $I_g$ is abelian (Remark \ref{rem:Ig abelian}),   left multiplication by $\gamma$ commutes with the left multiplication action of $I_g$ on $\Omega(m)$.
Making the substitution $(h,h') \mapsto  \gamma  \cdot (h , h' )$   shows that
 \begin{align*}
&  \int_{ I_g    \backslash  \Omega(m) } f(  h^{-1} g  h'   )
 \cdot \eta_{E/F}( h') \, dh \, dh'   \\
&  =
 \eta_{E/F}( \mathbf{z} u )
 \int_{ I_g    \backslash  \Omega(m) } f(  h^{-1}  g h'   )  
 \cdot \eta_{E/F}( h') \, dh \, dh' .
\end{align*}
The equality \eqref{orbital chunk vanish}  follows from this and the relation $\eta_{E/F}( \mathbf{z} u) =-1$ from the proof of Proposition \ref{prop:orbital vanish}.  
  \end{proof}

\end{document}